\documentclass[pdflatex,sn-mathphys-num]{sn-jnl}% Math and Physical Sciences Numbered Reference Style
\usepackage{graphicx}%
\usepackage{multirow}%
\usepackage{amsmath,amssymb,amsfonts}%
\usepackage{amsthm}%
\usepackage{mathrsfs}%
\usepackage{hyperref}
\usepackage[title]{appendix}%
\usepackage{xcolor}%
\usepackage{textcomp}%
\usepackage{manyfoot}%
\usepackage{booktabs}%
\usepackage{algorithm}%
\usepackage{algorithmicx}%
\usepackage{algpseudocode}%
\usepackage{listings}%

\newcommand{\argmin}{\arg\min}
\newcommand{\cG}{\mathcal{G}}

\newcommand{\ve}{\varepsilon}
\theoremstyle{thmstyleone}%
\newtheorem{theorem}{Theorem}%  meant for continuous numbers
\newtheorem{lemma}[theorem]{Lemma}% 
\newtheorem{assumption}[theorem]{Assumption}% 
\newtheorem{proposition}[theorem]{Proposition}% 

\theoremstyle{thmstyletwo}%
\newtheorem{remark}{Remark}%

\theoremstyle{thmstylethree}%

\usepackage{tcolorbox}
\tcbset{
    mystyle/.style={
        colback=white!5, % 背景色（浅灰）
        colframe=black, % 边框色
        rounded corners, % 圆角
        boxrule=0.5pt,   % 边框粗细
        left=10pt,       % 内边距（左）
        right=10pt,      % 内边距（右）
        top=5pt,         % 内边距（上）
        bottom=5pt       % 内边距（下）
    }
}

\begin{document}

\title[A Stochastic Block-coordinate Proximal Newton Method]{A Stochastic Block-coordinate Proximal Newton Method for Nonconvex Composite Minimization}


\author[1]{\fnm{Hong} \sur{Zhu}}\email{zhuhongmath@126.com}

\author*[2]{\fnm{Xun} \sur{Qian}}\email{xunqian2099@163.com}

\affil[1]{\orgdiv{School of Mathematical Sciences}, \orgname{Jiangsu University}, \orgaddress{\street{Xuefu Road}, \city{Zhenjiang}, \postcode{2120123}, \state{Jiangsu}, \country{China}}}

\affil*[2]{\orgdiv{Center of Excellence for Generative AI}, \orgname{King Abdullah University of Science and Technology}, \orgaddress{\city{Thuwal}, \country{Saudi Arabia}}}


\abstract{This paper presents a stochastic block-coordinate proximal Newton method for minimizing the sum of a blockwise Lipschitz-continuously differentiable function and a separable nonsmooth convex function. At each iteration, the method randomly selects one block and approximately solves a strongly convex regularized quadratic subproblem built from a second-order local model of the smooth part of the objective function, with a backtracking line search to ensure monotonicity of the objective. Under mild sampling assumptions, we show that its convergence properties match those of the inexact proximal Newton method.
We further develop a line-search-free variant, where the strongly convex regularized quadratic subproblem is constructed using the Lipschitz constant of the gradient of the smooth component. For this variant, under a suitable parameter setting, we establish the global convergence rate of the residual mapping as well as the superlinear convergence rate of the iterates under the metric \(q\)-subregularity property with \(q > 1\) of the residual mapping for nonconvex composite problems. Under a suitable parameter setting, a more restrictive condition on the Hessian approximation, and the H\"olderian error bound condition (\(q\in(0, 1]\)) of the residual mapping, we also prove the local superlinear/quadratic convergence rate of both the residual mapping and the iterates for convex composite problems.
Finally, numerical experiments are conducted to demonstrate the effectiveness and convergence behavior of the proposed algorithm.}

\keywords{stochastic block-coordinate method, proximal Newton method, nonconvex composite optimization, H\"olderian error bound condition}


\pacs[MSC Classification]{49M15, 65K05, 90C06, 90C26 }

\maketitle

\section{Introduction}\label{intro}

In this paper, we propose a stochastic second-order method for addressing large-scale nonconvex and nonsmooth composite optimization problems
\begin{equation}\label{eq:ncp}
\min_{x\in\mathbb{R}^n}\varphi(x) : = f(x) + g(x),
\end{equation} 
where \(f\) is a twice continuously differentiable function and \(g\) is a convex, lower semicontinuous, and proper mapping.
Problems of the form~\eqref{eq:ncp} frequently occur in the fields of science, engineering, and machine learning~\cite{YHL12}. As the dimensionality of the problem increases, the computational cost associated with evaluating gradients and Hessian matrices can become prohibitively high. Consequently, block coordinate descent (BCD) methods~\cite{B1999,T01,TY09} and their variants have gained significant attention in the literature~\cite{W12,W15,RT16,TRG16}.

BCD methods iteratively select one coordinate block to significantly decrease the objective value while maintaining the other blocks fixed. A widely adopted technique for selecting such a block is by means of a cyclic strategy. Randomized selection strategies have been developed for BCD methods, proving particularly effective in addressing large-scale optimization problems that arise in machine learning applications~\cite{CHL08,SST09,SZ13}. 
The iteration complexity of randomized BCD methods for minimizing smooth convex functions has been studied in~\cite{N12,CHL08,LL10,SST09}, while the complexity for convex composite functions has been discussed in~\cite{RT14,LX15}.
For the nonconvex composite problems, the convergence properties of randomized BCD methods have been studied in~\cite{PN15,XY15,LX17}. All of the aforementioned methods are first-order methods, which indicates that only the gradient information of the smooth component of the objective function was used during each iteration.

Recently, second-order subspace methods have been proposed to utilize the local curvature information of the smooth component of the objective function for solving large-scale problems. These methods employ random subspace techniques to address high-dimensional  Hessian. For smooth convex optimization, Gower et al.~\cite{GKLP19} proposed a randomized subspace Newton method. For smooth nonconvex optimization, Fuji et al.~\cite{FPT24} proposed a randomized subspace variant of the regularized Newton method discussed by Ueda and Yamashita~\cite{UY10} and Zhao et al.~\cite{ZLD24} proposed a cubic regularized subspace Newton method.
The existing literature on randomized second-order methods for composite optimization is comparatively less extensive. Hanzely et al.~\cite{HDNR20} proposed a cubic regularization method to address convex composite optimization problems.
Both~\cite{HDNR20,ZLD24} require the exact solution of the cubic regularization subproblem at each iteration, which typically lacks a closed-form solution. 

The (inexact) proximal Newton methods (IPNM)~\cite{LSS12,LSS14,ST16,LW19,YZS19,KL21,L22,MYZZ22,LPWY23,Z24} have been studied to address Problem~\eqref{eq:ncp}. Given the current iterate \(x^k\), the fundamental approach of IPNM is to approximately  solve the subproblem 
  \begin{equation}\label{eq:smajf}
\min_x\{q_k(x) := l^k(x) + \frac{1}{2}\langle Q_k(x - x^k), x - x^k\rangle + \frac{\eta_k}{2}\|x - x^k\|^2 \},
\end{equation}
where \(l^k(x) := f(x^k) + \langle\nabla f(x^k), x - x^k\rangle + g(x)\), the symmetric positive semidefinite matrix \(Q_k\) is an approximation to \(\nabla^2f(x^k)\), and \(\eta_k > 0\) is the regularization parameter. 
Let \(\hat{x}^k\) be an approximate solution to Problem~\eqref{eq:smajf}. \(x^k\) will be updated along the direction \(\hat{x}^k - x^k\). 
The convergence rate of IPNM in terms of \(\min_k\{\|\mathcal{G}(x^k)\|\}\) is \(\mathcal{O}(1/\sqrt{k})\)~\cite{LW19,KL21,Z24}, where the KKT residual mapping \(\mathcal{G}(x) := x - {\rm prox}_g(x - \nabla f(x))\) with \({\rm prox}_g(u) := \argmin_x\{g(x) + \frac{1}{2}\|x - u\|^2\}\). Numerical experiments in~\cite{YZS19,MYZZ22} have demonstrated that IPNMs are highly effective for solving regularized logistic regression problems when \(n\) is large. The stochastic block-coordinate variants of IPNM have been studied for convex composite optimization problems~\cite{L17,FT18,LW20}. In~\cite{L17}, \(f\) was assumed to be self-concordant. The termination condition of the subproblem solver  proposed by~\cite{FT18} may be costly to verify, except for specific choices of the regularizer. 
A more practical termination criterion for the subproblem solver was introduced in~\cite{LW20}.  
The global convergence analysis in terms of the expected minimal squared norm of the  residual mapping for nonconvex composite optimization problems was studied in~\cite{LW20}. However, neither the convergence of the expected objective value nor the local convergence rate of the algorithm was addressed in their analysis.
In this paper, we introduce a stochastic block-coordinate proximal Newton method (SBCPNM) for solving Problem~\eqref{eq:ncp} and present a comprehensive convergence analysis. Throughout this paper, we assume that \(\varphi\) is lower bounded and denote \(x_*\) as any minimizer of \(\varphi\), with \(\varphi_* := \varphi(x_*)\) representing the corresponding optimal value. We establish the following assumption.
\begin{assumption}\label{assume:ncp}
\begin{itemize}\item[(i)] \(f: \mathbb{R}^n\to(-\infty, +\infty]\) is twice continuously differentiable and \(\nabla f\) is coordinatewise Lipschitz continuous with constants \(L_{S}\) for any index set \(S \subseteq [n] := \{1, \ldots, n\}\), that is
\[
\|\nabla f(x + h)_{S} - \nabla f(x)_{S}\|\leq L_{S}\|h\|, \quad \forall h\in R^n_{S},~\forall x\in\mathbb{R}^n,
\]
where \(\|\cdot\|\) denotes the Euclidean norm, \(\nabla f(x)_S\) denotes the subvector of \(\nabla f(x)\) indexed by the index set \(S\), and \(R^n_{S} := \{h\in\mathbb{R}^n~\vert~h_i = 0, \forall i\notin S\}\). 
\item[(ii)] \(g: \mathbb{R}^n\to(-\infty, +\infty]\) is block-coordinate separable, that is, \(g\) takes the form of 
\[
g(x) = \sum_{i=1}^{n_l}\psi_i(x_{I_i}),
\]
where \(\cup_{i=1}^{n_l}I_i = [n]\), \(I_i\cap I_j = \emptyset\), \(\forall i\neq j\), \(x_{I_i}\) denotes the subvector of \(x\) indexed by the index set \(I_i\), \(\psi_i: \mathbb{R}^{\vert I_i\vert}\to(-\infty, +\infty]\) is a proper closed convex function, \(\min_{z_i}\{\psi_i(z_i) + \frac{1}{2}\|z_i - u\|^2\}\) is efficiently solvable, and \(0\in {\rm dom}\psi_i\), \(i = 1, \ldots, n_l\). 
\item[(iii)] For any \(x^0\in {\rm dom}g\), the level set \(\mathcal{L}_{\varphi}(x^0) = \{x\vert \varphi(x) \leq \varphi(x^0)\}\) is bounded. 
\end{itemize}
\end{assumption}
Without loss of generality, we assume that the elements in the index set \(I_i\), \(i = 1, \ldots, n_l\), are arranged in ascending order. Block-coordinate separable functions are widely used in fields such as machine learning and signal processing. Typical examples include \(\|x\|_1\), \(\|x\|_2^2\), \(\sum_{i=1}^{n_l}\|x_{I_i}\|_2\), and box constraints.
From Assumption~\ref{assume:ncp} (i), we have 
\begin{equation}\label{eq:nfslc}
f(x + h) \leq f(x) + \langle \nabla f(x), h\rangle + \frac{L_S}{2}\|h\|^2, \quad \forall h\in R^n_S,~\forall x\in\mathbb{R}^n. 
\end{equation}
Define \(L_g := \max_{S\subseteq [n]}\{L_{S}\}\), \(\nabla f\) is \(L_g\)-Lipschitz continuous. Hence, \(\|\nabla^2f(x)\| \leq L_g\) over \(\mathcal{L}_{\varphi}(x^0)\). Moreover, there exist \(\bar{\epsilon}_0 >0\) and \(\bar{\epsilon}_1 > 0\), such that for every \(x\in  \mathcal{L}_{\varphi}(x^0)\), we have
\[
\varphi(x) \geq \varphi_*, \quad \|x\|\leq \bar{\epsilon}_0, \quad \|\nabla f(x)\| \leq \bar{\epsilon}_1.
\]

\noindent\textbf{Contribution.} SBCPNM can be regarded as a stochastic block-coordinate variant of the IPNM proposed by Zhu in~\cite{Z24}. The main convergence results are consistent with the corresponding ones for IPNM in the literature. Notably, the method is called ``proximal Newton"  (as in~\cite{Z24}), since it builds and solves a regularized quadratic subproblem using Hessian approximations of \(f\) at each iteration, irrespective of whether these approximations converge to the true Hessian or satisfy a Dennis-Mor\'{e}-type condition~\cite{DM74}. For appropriate choices of the function \(g\) and related parameters, SBCPNM exhibits similarities to several existing methods, and importantly, knowledge of blockwise Lipschitz constants is not required. i) We show that the sequence of expected objective values generated by SBCPNM converges to the expected limit of the objective values. ii) We study the convergence rate of the (expected) minimal squared norms of residual mappings under various sampling assumptions. We prove that, under a suitable sampling condition, any accumulation point of the sequence generated by SBCPNM is a stationary point of Problem~\eqref{eq:ncp}, and its main convergence results are consistent with those of IPNM. iii) We show that SBCPNM with a unit step size is well-defined when the Lipschitz constant \(L_{S_k}\) is used to formulate the regularized subproblem. Under an appropriate sampling assumption and parameter setting, we establish a local superlinear convergence rate of the residual mapping for nonconvex composite problems via the metric \(q\)-subregularity property with \(q > 1\), and a local superlinear/quadratic convergence rate of both the residual mapping and the iterates for convex composite problems under a more restrictive condition on the Hessian approximation and the H\"olderian error bound condition (\(q \in (0, 1]\)). Relative to the most relevant reference~\cite{LW20}, our study on the convergence of the expected objective values sequence, as well as the local convergence of the residual mapping and iterates for SBCPNM with unit step size, is new.

\noindent\textbf{Notation and facts.} \(\|\cdot\|\) denotes the Euclidean norm or its induced norm on matrices. Let \(S\subseteq [n]\) be sampled from an arbitrary but fixed distribution \(\mathcal{D}\), we use \(\vert S\vert\) to denote the cardinality of \(S\) and denote \(\overline{S} := [n]\backslash S\) as the complementary set of \(S\). For any \(x\in\mathbb{R}^n\) and \(A\in\mathbb{R}^{n\times n}\), denote \(x_{[S]}\in\mathbb{R}^n\) and \(A_{[S]}\in\mathbb{R}^{n\times n}\) by \((x_{[S]})_i = x_i\) if \(i\in S\) and \((x_{[S]})_i = 0\), otherwise; and \((A_{[S]})_{ij} = A_{ij}\) if \(i, j\in S\) and \((A_{[S]})_{ij} = 0\), otherwise. 
We denote \(A_S\in\mathbb{R}^{\vert S\vert\times \vert S\vert}\) as the  
principal submatrix of \(A\) induced by \(S\). For vector \(x\), \(\vert x\vert\) denotes the absolute value vector of \(x\), where \((\vert x\vert)_i = \vert x_i\vert\). 
For any symmetric matrix \(Q\), \(Q \succeq 0\) indicates that \(Q\) is a semidefinite positive matrix. We use \(\mathbb{E}[\cdot]\) and \(\mathbb{P}(\cdot)\) to denote expectation and probability, respectively. \(\mathbb{B}(x, r)\) denotes the open Euclidean norm ball centered at \(x\) with radius \(r > 0\).

For any local minimum \(\bar{x}\) of Problem~\eqref{eq:ncp}, we have \(0\in \nabla f(\bar{x}) + \partial g(\bar{x})\). Any vector \(\bar{x}\) satisfying this relation is called a stationary point for Problem~\eqref{eq:ncp}. 
Let \(\mathcal{S}^*\) be the set of stationary points of Problem~\eqref{eq:ncp}. It immediately follows from Assumption~\ref{assume:ncp} that \(\bar{x}\in \mathcal{S}^*\) if and only if \(\mathcal{G}(\bar{x}) = 0\). Define 
\[
\cG_{S}(y) = y - {\rm prox}_{\tilde{g}}(y - \nabla f(x)_{S}),
\] 
where \(S\subseteq [n]\) is a sampled index set, \(y \in\mathbb{R}^{\vert S\vert}\) is the subvector of \(x\) induced by \(S\), and \(\tilde{g}(y) = \sum_{\{i\vert I_i\subseteq S\}}\psi_i(x_{I_i})\). The following property holds. 

\begin{proposition}\label{prop:gsy}
Under Assumption~\ref{assume:ncp} (ii), we have
\begin{itemize}
\item[(i)]~\(\left({\rm prox}_g(x)\right)_j = ({\rm prox}_{\psi_i}(x_{I_i}))_{J_j}\), where \(j\in I_i\) is the \(J_j\)-th element in \(I_i\), \(j = 1, \ldots, n\);
\item[(ii)]~\(\cG_S(y) = \cG(x)_{S}\) with \(y = x_S\).
\end{itemize}
\end{proposition}

Statement (i) follows from~\cite[Theorem 6.6]{B17}. Statement (ii) follows from statement (i) and definitions of \(\cG(x)\) and \(\cG_{S}(y)\). 

\noindent\textbf{Organization.} The rest of the paper is organized as follows. In Section~\ref{sec:pnmethods}, we present SBCPNM and provide a detailed global convergence analysis. In Section~\ref{appendix: lhisgiven}, we discuss a line-search-free variant of SBCPNM (with \(L_{S_k}\) used to design subproblmes) and establish its global convergence rate for nonconvex composite problems, along with its local convergence rate for nonconvex/convex composite problems under particular parameter settings and different error bound conditions. In Section~\ref{sec:numerical}, we conduct numerical experiments on the \(\ell_1\)-regularized Student's \(t\)-regression, nonconvex binary classification, and biweight loss with group regularization. We make some conclusions in Section~\ref{sec:conclusions}.

%%%%%%%%%%%%%%%%%%%%%%%%%%%%%%%%%%%%%%%%%%%%%%%%
% ******* The Stochastic Block-coordinate Proximal Newton Method
%%%%%%%%%%%%%%%%%%%%%%%%%%%%%%%%%%%%%%%%%%%%%%%%
\section{The Stochastic Block-coordinate Proximal Newton Method}\label{sec:pnmethods}

In this section, we present SBCPNM for Problem~\eqref{eq:ncp}. Knowledge of coordinatewise Lipschitz constants is not assumed.  

%%%%%%%%%
\subsection{The stochastic block-coordinate  proximal Newton method}\label{sec:pmm}

Given the current iterate \(x^k\), pick \(S_k\subseteq [n]\) from an arbitrary but fixed distribution \(\mathcal{D}\). We approximately solve the following problem:
\[
\min_{y\in\mathbb{R}^{\vert S_k\vert}}\{q_{S_k}^k(y) := l^k_{S_k}(y)  + \frac{1}{2}\langle (Q_k)_{S_k}(y - y^k), y - y^k\rangle + \frac{\eta_k}{2}\|y - y^k\|^2 \},
\]
where \(l^k_{S_k}(y) := f(x^k) + \langle \nabla f(x_k)_{S_k}, y - y^k\rangle \!+\! \tilde{g}_k(y)\), \(y^k = x^k_{S_k}\), \(\tilde{g}_k(y) \!= \!\sum_{\{i\vert I_i \subseteq S_k\}}\!\psi_i(y_i)\), and \(y_i\in\mathbb{R}^{\vert I_i\vert}\). Let \(\hat{y}^k\) be an approximate solution of the above problem. We then update \(x^k\) by setting \(x^{k+1}_{S_k} = x^k_{S_k} + \alpha_k(\hat{y}^k - y^k)\) and \(x^{k+1}_{\overline{S_k}} = x^k_{\overline{S_k}}\), where \(\alpha_k > 0\) is the step size. 

We use the following criterion for the approximate solution \(\hat{y}^k\): there exists \(\varsigma_k\in \partial q^k_{S_k}(\hat{y}^k) := \nabla f(x_k)_{S_k} + (Q_k)_{S_k}(\hat{y}^{k} - y^k) + \eta_k(\hat{y}^k - y^k) +\partial \tilde{g}_k(\hat{y}^k)\), such that   
\begin{equation}\label{eq:errvek}
\|\varsigma_k\| \leq \frac{\mu}{2}\|\hat{y}^k - y^k\|
\end{equation}
for some \(\mu > 0\). 
Notice that \(\mu = 0\) corresponds to the case in which the subproblems are solved exactly. By applying the first-order optimality condition at a point \(y\) when a proximal-type method is employed to minimize \(q_{S_k}^k(y)\), an element of the \(\partial q^k_{S_k}(\hat{y}^k)\) can be readily obtained. Accuracy criterion~\eqref{eq:errvek} can be satisfied by the proximal gradient method~\cite{B17} and the FIAST method~\cite{B17} when \(\|\mathcal{G}(x^k)_{S_k}\| \neq 0\). Further discussions on solvers satisfying~\eqref{eq:errvek} can be found in~\cite{Z24}. 
In addition, \(\hat{y}^k\) can be stated as an exact solution of the problem
\begin{equation}\label{eq:xsub}
\hat{y}^k = \arg\min_y\{q_{S_k}^k(y) + \langle \hat{\varepsilon}_k, y - y^k\rangle\}
\end{equation}
for some \(\|\hat{\varepsilon}_k\| \leq \frac{\mu}{2}\|\hat{y}^k - y^k\|\) since the first-order optimality condition of Problem~\eqref{eq:xsub} yields \(0 \in \partial q^k_S(\hat{y}^{k}) + \hat{\varepsilon}_k\), which implies \(-\hat{\varepsilon}_k \in \partial q^k_S(\hat{y}^{k})\). Equation~\eqref{eq:xsub} plays a crucial role in the subsequent theoretical analysis. We summarize SBCPNM in Algorithm~\ref{alg:pnewton}.

\begin{algorithm*}[th!]
\caption{Stochastic block-coordinate proximal Newton (SBCPN) method with backtracking line search.}\label{alg:pnewton}
\begin{algorithmic}[1]
\Require{\(x^0\in {\rm dom}g\), \(\bar{\eta} > 0\), \(\mu\in(0, 1)\), \(\tau\in(0, \mu)\), and \(\theta\in(0, 1)\), distribution \(\mathcal{D}\) of random index set.}
\For{\(k = 0, 1, \ldots, \)}
\State{sample \(S_k\) from \(\cal{D}\);}
\State{set \(\eta_k\in(0, \bar{\eta}]\) and \(Q_k\) satisfy \((Q_k)_{S_k} + (\eta_k - \mu)I_{\vert S_k\vert} \succeq 0\);}
\State{set \(y^k = x^k_{S_k}\), compute}
\[
\hat{y}^k \approx \argmin_{y}\{q^k_{S_k}(y)\}\quad {\rm with}~\varsigma_k\in \partial q^k_{S_k}(\hat{y}^k)~{\rm satisfies}~\eqref{eq:errvek}.
\] 
\State{define \(\hat{x}^k\) with \(\hat{x}_{S_k}^k = \hat{y}^k\) and \(\hat{x}_{\overline{S_k}}^k = x_{\overline{S_k}}^k\);}
\State{set \(d_k = \hat{x}^k - x^k\), \(x^{k+1} = x^k + \alpha_kd_k\), where \(\alpha_k = \theta^{j_k}\) and \(j_k\) is the smallest nonnegative integer such that  }
\begin{equation}\label{eq:ls}
\varphi(x^k + \theta^{j_k}d_k) \leq \varphi(x^k) - \frac{\tau}{2}\theta^{j_k}\|d_k\|^2.
\end{equation}
\EndFor\\
\Return{\(\{x^k\}\)}
\end{algorithmic}
\end{algorithm*}
\begin{remark}\label{rem:prob}
The main differences between Algorithm~\ref{alg:pnewton} and the inexact variable metric block-coordinate descent method proposed in~\cite{LW20} are the termination condition of the subproblem and the line search condition. In the latter method, the approximate solution \(\hat{y}^k\) for each \(k\) satisfies
\[
q_{S_k}^k(\hat{y}^k) - q_{S_k}^{k,*} \leq -\eta_{vm}[q_{S_k}^{k,*} - f(x^k) - \tilde{g}_k(y^k)],
\]
where \(q_{S_k}^{k,*} = \inf_yq_{S_k}^{k}(y)\) and \(\eta_{vm}\in (0, 1)\). Adaptive choices of \(\eta_{vm}\) are allowed and \(q_{S_k}^{k,*}\) is not required in calculating. 
\(x^{k+1}\) is updated by \(x^k + \alpha_k d_k\), where \(\alpha_k\) satisfies 
\[
\varphi(x^k + \alpha d_k) \leq \varphi(x^k) + \alpha_k \gamma_{vm}(\nabla f(x^k)^\top d_k + \tilde{g}_k(\hat{y}^k) - \tilde{g}_k(y^k))
\]
for some \(\gamma_{vm} \in (0, 1)\). 
\end{remark}

\begin{remark}
Algorithm~\ref{alg:pnewton} reduces to IPNM proposed in~\cite{Z24} when \(S_k \equiv [n]\), \(\forall k\). 
With specific choices of \(g\) and parameters, Algorithm~\ref{alg:pnewton} reduces to several existing methods.
\begin{enumerate}
\item If \((Q_k)_{S_k} \equiv 0\), \(\hat{\ve}_k \equiv 0\) (\(\mu = 0\)), and \(\eta_k = L_{S_k}\), \(\forall k\in\mathbb{N}\), where \(L_{S_k}\) denotes the Lipschitz constant of \(\nabla f(x)_{S_k}\), then in Algorithm~\ref{alg:pnewton}, 
\[
\hat{y}^k = {\rm prox}_{\frac{1}{L_{S_k}}\tilde{g}_k}(y^k - \frac{1}{L_{S_k}}\nabla f(x^k)_{S_k});\quad (d_k)_{S_k} = \hat{y}^k - y^k; \quad (d_k)_{\overline{S_k}} = 0. 
\]

In this case, \(\alpha_k = 1\) satisfies the line search condition~\eqref{eq:ls} by noting that 
\begin{align*}
\varphi(x^k + d_k) =& f(x^k + d_k) + g(x^k + d_k) =  f(x^k + d_k) + \tilde{g}_k(\hat{y}^k) + \sum_{i\in\overline{S_k}}\psi_i(x^k_i)\\
\leq& f(x^k) + \langle\nabla f(x^k), d_k\rangle + \frac{L_{S_k}}{2}\|d_k\|^2 + \tilde{g}_k(\hat{y}^k) + \sum_{i\in\overline{S_k}}\psi_i(x^k_i)\\
\leq& f(x^k) + \langle\nabla f(x^k), d_k\rangle \!+\! \frac{L_{S_k}}{2}\|d_k\|^2 \!+\! \tilde{g}_k(y^k) \!-\! \langle \hat{y}^k \!-\! y^k, \nabla f(x_k)_{S_k}\rangle \\
&- L_{S_k}\|\hat{y}^k - y^k\|^2 + \sum_{i\in\overline{S_k}}\psi_i(x^k_i) = \varphi(x^k) - \frac{L_{S_k}}{2}\|d_k\|^2,
\end{align*}
where the first inequality follows from~\eqref{eq:nfslc} and the second inequality follows from the optimality condition of problem with respect to \(\hat{y}^k\) and the convexity of \(\tilde{g}_k(y)\) (see (3.12) in~\cite{BST14}). Hence, we have 
\[
\varphi(x^k + d_k) \leq \varphi(x^k) - \frac{\tau}{2}\|d_k\|^2 \quad {\rm with}\quad \tau \leq L_{S_k}. 
\] 

The above iterate can be viewed as the randomized block-coordinate descent method~\cite{RT14} if we further assume \(f\) to be convex. In addition, 
\begin{enumerate}
\item if \(g \equiv 0\) and we further assume \(f\) to be convex, then the above iterate can be viewed as the coordinate descent method~\cite{N12}.
\item if \(g = \delta_{C_1\times\ldots\times C_m}(x) = \left\{\begin{array}{ll}
0 & {\rm if}~x_{(i)}\in C_i,~\forall i,\\
+\infty & {\rm otherwise}
\end{array}
\right.\), where \(C_1, \ldots, C_m\) are convex sets, \(x_{(i)}\) denotes the \(i\)-th block of \(x\), and we further assume \(f\) to be convex, then the above iterate can be viewed as the constrained coordinate descent method~\cite{N12}. 
\end{enumerate}

\item If \(g \equiv 0\) and \(\hat{\ve}_k\equiv 0\) (\(\mu = 0\)), then in  Algorithm~\ref{alg:pnewton},
\[
\left\{
\begin{array}{l}
\hat{y}^k = y^k - ((Q_k)_{S_k} + \eta_kI)^{-1}\nabla f(x^k)_{S_k};\\
(d_k)_{S_k} = \hat{y}^k - y^k; \quad (d_k)_{\overline{S_k}} = 0. 
\end{array}
\right.
\]
%Next, we show \(\alpha_k = 1\) satisfies~\eqref{eq:ls} if \((Q_k)_{S_k} + (\eta_k - \frac{\tau + L_{S_k}}{2})I_{\vert S_k\vert} \succeq 0\). Notice that
%\begin{align*}
%\varphi(x^k + d_k) =& f(x^k + d_k)\leq f(x^k) + \nabla f(x^k)^\top d_k + \frac{L_{S_k}}{2}\|d_k\|^2\\
%=&f(x^k) - \nabla f(x^k)_{S_k}^\top ((Q_k)_{S_k} + \eta_k I_{\vert S_k\vert})^{-1}\nabla f(x^k)_{S_k}\\
%& + \frac{L_{S_k}}{2}\nabla f(x^k)_{S_k}^\top ((Q_k)_{S_k} + \eta_k I_{\vert S_k\vert})^{-2}\nabla f(x^k)_{S_k}.
%\end{align*}
%To satisfy~\eqref{eq:ls}, it is sufficient to ensure 
%\[
%-\frac{\tau + L_{S_k}}{2}((Q_k)_{S_k} + \eta_k I_{\vert S_k\vert})^{-2} + ((Q_k)_{S_k} + \eta_k I_{\vert S_k\vert})^{-1}  \succeq 0.
%\]
\begin{enumerate}
\item If we further assume \(f\) to be \(\hat{\mu}\)-strongly convex, it can be proved that \(\alpha_k \equiv \frac{\hat{\mu}}{L_g}\) satisfies~\eqref{eq:ls} if \((Q_k)_{S_k} + (\eta_k - \frac{\tau + L_{S_k}\hat{\mu}^2}{2L_g\hat{\mu}})I_{\vert S_k\vert} \succeq 0\). When \(\eta_k \equiv 0\) and \(Q_k \equiv \nabla^2f(x^k)\), the above iterate can be viewed as a special case of the randomized subspace Newton method~\cite{GKLP19}. 
\item If \((Q_k)_{S_k} \!=\! (\nabla^2f(x^k))_{S_k} \!+\! c_1\max\{0, \!-\lambda_{\min}(\!(\!\nabla^2f(x^k)\!)_{S_k})\}I \!+\! c_2\|\nabla f(x_k)\|^{\delta}I\) for some \(c_1 > 1\), \(c_2 > 0\), and \(\delta \geq 0\), then Algorithm~\ref{alg:pnewton} is similar to the randomized subspace regularized Newton method~\cite{FPT24} except that the line search conditions are different. 
\end{enumerate}
\end{enumerate}
\end{remark}

%%%%%%%%%%%%%%%%%%%%%%%%%%%%%%%%%%%%%%%%%%%%%%%%% ******* Properties of Algorithm~1
%%%%%%%%%%%%%%%%%%%%%%%%%%%%%%%%%%%%%%%%%%%%%%%%
\subsection{Properties of Algorithm~\ref{alg:pnewton}}

In this subsection, we focus on a particular iteration \(k\).  We first show that the line search condition~\eqref{eq:ls} is well defined. 

\begin{lemma}\label{lemma:dqkng}
Suppose for \(k\in\mathbb{N}\), \((Q_k)_{S_k} + (\eta_k - \mu) I_{\vert S_k\vert} \succeq 0\).  
Let \(x^k\) and \(\hat{x}^k\) be points generated by Algorithm~\ref{alg:pnewton}. We have 
\[
q_k(\hat{x}^k) \leq \varphi(x^k). 
\]
\end{lemma}
\begin{proof}
Notice that \(q^k_{S_k}\) is a \(\mu\)-strongly convex function since \((Q_k)_{S_k} + (\eta_k - \mu) I_{\vert S_k\vert} \succeq 0\). 
For any \(y\) and \(u^k\in \partial q^k_{S_k}(\hat{y}^k)\), it holds that 
\[
q^k_{S_k}(y) \geq q^k_{S_k}(\hat{y}^k) + \langle u^k, y - \hat{y}^k\rangle + \frac{\mu}{2}\|y - \hat{y}^k\|^2. 
\] 
According to the optimality condition of Problem~\eqref{eq:xsub}, we have \(0\in \partial q^k_{S_k}(\hat{y}^k) + \hat{\ve}_k\). 
Hence, by setting \(y := y^k\) and \(u^k := -\hat{\ve}_k\), we have 
\begin{align*}
q^k_{S_k}(y^k) \geq& q^k_{S_k}(\hat{y}^k) - \langle\hat{\ve}_k, y^k - \hat{y}^k\rangle + \frac{\mu}{2}\|y^k - \hat{y}^k\|^2\\
\geq& q^k_{S_k}(\hat{y}^k) - \|\hat{\ve}_k\|\|y^k - \hat{y}^k\|+ \frac{\mu}{2}\|y - \hat{y}^k\|^2\geq q^k_{S_k}(\hat{y}^k). 
\end{align*}
Notice that by the definition of \(\hat{x}^k\) and \(q_{S_k}^k\), we have \(q^k_{S_k}(y^k) = f(x^k) + \tilde{g}_k(y^k) = \varphi(x^k) - \sum_{i\notin S_k}\psi_i(x_i^k) \geq q^k_{S_k}(\hat{y}^k)\), which yields
\[
\varphi(x^k) \geq q^k_{S_k}(\hat{y}^k) + \sum_{i\notin S_k}\psi_i(x_i^k).
\]
Notice that  
\begin{align}\label{eq:qskyqkx}
q^k_{S_k}(\hat{y}^k) =&  l_{S_k}^k(\hat{y}^k)+ \frac{1}{2}\langle (Q_k)_{S_k}(\hat{y}^k \!-\! y^k), \hat{y}^k \!-\! y^k\rangle \!+\! \frac{\eta_k}{2}\|\hat{y}^k \!-\! y^k\|^2 \! =q_k(\hat{x}^k)  - \sum_{i\notin S_k}\psi_i(x_i^k). 
\end{align}
Therefore, we have \(\varphi(x^k) \geq q_k(\hat{x}^k)\). The statement holds. 
\end{proof}

\begin{lemma}\label{lem:alpdphi}
Suppose for \(k\in\mathbb{N}\), \((Q_k)_{S_k} + (\eta_k - \mu) I_{\vert S_k\vert} \succeq 0\) and Assumption~\ref{assume:ncp} hold. 
Let \(\alpha_k\) be chosen by the backtracking line search~\eqref{eq:ls} in Algorithm~\ref{alg:pnewton} at iteration \(k\). Then we have the step size estimate  
\[
\alpha_k \geq \min\{1, \frac{\theta(\mu -\tau)}{L_{S_k}}\}
\]
with the cost function decrease satisfying
\begin{equation}\label{eq:dvarphi}
\varphi(x^{k+1}) - \varphi(x^k) \leq -\frac{\tau}{2}\min\{1, \frac{\theta(\mu -\tau)}{L_{S_k}}\}\|d_k\|^2.
\end{equation}
\end{lemma}
\begin{proof}
From~\eqref{eq:qskyqkx}, we have 
\begin{align*}
q_k(\hat{x}^k) =& q^k_{S_k}(\hat{y}^k) + \sum_{i\notin S_k}\psi_i(x_i^k)\\
=&l^k_{S_k}(\hat{y}^k) + \frac{1}{2}\langle(Q_k)_{S_k}(\hat{y}^k - y^k), \hat{y}^k - y^k\rangle + \frac{\eta_k}{2}\|\hat{y}^k - y^k\|^2 + \sum_{i\notin S_k}\psi_i(x_i^k).
\end{align*}
Notice that \(\varphi(x^k) = f(x^k) + \tilde{g}_k(y^k) + \sum_{i\notin S_k}\psi_i(x_i^k) = l^k_{S_k}(y^k) + \sum_{i\notin S_k}\psi_i(x_i^k)\). It follows from Lemma~\ref{lemma:dqkng} that
\begin{align*}
0 \geq q_k(\hat{x}^k) - \varphi(x^k) =& l^k_{S_k}(\hat{y}^k) + \frac{1}{2}\langle(Q_k)_{S_k}(\hat{y}^k - y^k), \hat{y}^k - y^k\rangle + \frac{\eta_k}{2}\|\hat{y}^k - y^k\|^2 - l^k_{S_k}(y^k) \nonumber \\
\geq & l^k_{S_k}(\hat{y}^k) - l^k_{S_k}(y^k) + \frac{\mu}{2}\|d_k\|^2, 
\end{align*}
where the last inequality holds since \((Q_k)_{S_k} + (\eta_k - \mu) I_{\vert S_k\vert} \succeq 0\). Therefore, we have 
\begin{equation}\label{eq:dl}
l^k_{S_k}(y^k) - l^k_{S_k}(\hat{y}^k) \geq \frac{\mu}{2}\|d_k\|^2.
\end{equation} 
Notice that for any \(t\in[0, 1]\), 
\begin{align*}
&\varphi(x^k) - \varphi(x^k + td_k)\\
=& l^k_{S_k}(y^k) + \sum_{i\in\overline{S_k}}\psi_i(x_i^k) - f(x^k + td_k) - g(x^k + td_k)\\
=& l^k_{S_k}(y^k)  \!-\! l^k_{S_k}(y^k + t(\hat{y}^k - y^k)) \!-\! (f(x^k + td_k) - f(x^k) \!-\! t\langle \nabla f(x^k)_{S_k}, \hat{y}^k \!-\! y^k\rangle)\\
=& l^k_{S_k}(y^k)  \!-\! l^k_{S_k}(y^k + t(\hat{y}^k - y^k)) \!-\! (f(x^k + td_k) \!-\! f(x^k) \!-\! t\langle\nabla f(x^k), d_k\rangle)\\
\geq & l^k_{S_k}(y^k)  - l^k_{S_k}(y^k + t(\hat{y}^k - y^k)) - \frac{L_{S_k}}{2}t^2\|d_k\|^2,
\end{align*}
where the third equality holds since \((d_k)_{S_k} = \hat{y}^k - y^k\) and \((d_k)_{\overline{S_k}} = 0\), the last inequality holds since \(\nabla f(\cdot)_{S_k}\) is \(L_{S_k}\)-Lipschitz continuous. 
Therefore, 
\begin{align*}
\varphi(x^k) \!-\! \varphi(x^k \!+\! td_k) \!-\! \frac{\tau }{2}t\|d_k\|^2\geq& l^k_{S_k}(y^k)  \!-\! l^k_{S_k}(y^k \!+\! t(\hat{y}^k \!-\! y^k)) \!-\! \frac{L_{S_k}}{2}t^2\|d_k\|^2  \!-\! \frac{\tau }{2}t\|d_k\|^2\\
\geq& t(l^k_{S_k}(y^k) - l^k_{S_k}(\hat{y}^k)) - \frac{L_{S_k}}{2}t^2\|d_k\|^2  - \frac{\tau }{2}t\|d_k\|^2 \\
\geq& \frac{\mu}{2}t\|d_k\|^2 - \frac{L_{S_k}}{2}t^2\|d_k\|^2  - \frac{\tau }{2}t\|d_k\|^2 \\
=&\frac{1}{2}((\mu - \tau) - L_{S_k}t) t\|d_k\|^2,
\end{align*}
where the second inequality holds since \(l_{S_k}^k\) is convex and the last inequality follows from~\eqref{eq:dl}. Hence,~\eqref{eq:ls} holds for any \(t\) that satisfies 
\[
0 < t\leq \frac{\mu - \tau}{L_{S_k}}. 
\] 
Combining with the backtracking technique used in Algorithm~\ref{alg:pnewton}, we have \(\alpha_k \geq \min\{1, \frac{\theta(\mu - \tau)}{L_{S_k}}\}\). Therefore,  
\[
\varphi(x^k) - \varphi(x^k + \alpha_k d_k) \geq \frac{\tau }{2}\alpha_k\|d_k\|^2 \geq \frac{\tau}{2}\min\{1, \frac{\theta(\mu - \tau)}{L_{S_k}}\}\|d_k\|^2.
\]
This completes the proof of the lemma.  
\end{proof}

Define \(\cG_{S_k}(y) = y - {\rm prox}_{\tilde{g}_k}(y - \nabla f(x^k)_{S_k})\). Next, we establish the bound of \(\|\cG_{S_k}(y^k)\|\), which will be used in the subsequent analysis on the convergence rate of Algorithm~\ref{alg:pnewton}. Throughout this paper we assume that
\begin{equation}\label{eq:dishfqk}
\|\nabla^2f(x^k) - Q_k\| \leq \zeta, \quad \forall k\in\mathbb{N}
\end{equation}
for some \(\zeta > 0\). 
Notice that \(\|\nabla^2f(x^k)_{S_k} - (Q_k)_{S_k}\| \leq \|\nabla^2f(x^k) - Q_k\|\). 
Combine with Assumption~\ref{assume:ncp} (i), inequality~\eqref{eq:dishfqk} implies that \(\max\{\|(Q_k)_{S_k}\|\} \leq \max\{\|Q_k\|\} \leq \zeta + \|\nabla^2f(x)\| \leq \zeta + L_g\). Without loss of generality, we can assume that 
\[
0 \leq \eta_k \leq \bar{\eta}: = \mu + 2L_g + \zeta, \quad \forall k \in \mathbb{N}. 
\]

\begin{lemma}\label{lem:ngk}
Suppose for \(k\in\mathbb{N}\), \((Q_k)_{S_k} + (\eta_k - \mu) I_{\vert S_k\vert} \succeq 0\), the boundedness~\eqref{eq:dishfqk}, and Assumption~\ref{assume:ncp} hold. Let \(x^k\) be the point generated by Algorithm~\ref{alg:pnewton} and \(y^k = x^k_{S_k}\). We have 
\[
\|\cG_{S_k}(y^k)\| \leq c_1\|d_k\|, 
\]
where \(c_1 = 1 + L_g + \zeta + \bar{\eta} + \frac{\mu}{2}\). 
\end{lemma}
\begin{proof}
Define \(r_{S_k}^k(y) = y - {\rm prox}_{\tilde{g}_k}(y - (\nabla f(x^k)_{S_k} + ((Q_k)_{S_k} + \eta_kI_{\vert S_k\vert})(y - y^k)))\). We have 
\begin{equation}\label{eq:ek1}
\hat{y}^k - r^k_{S_k}(\hat{y}^k) = {\rm prox}_{\tilde{g}_k}(\hat{y}^k - \nabla f(x^k)_{S_k} - ((Q_k)_{S_k}  + \eta_kI_{\vert S_k\vert})(\hat{y}^k - y^k)). 
\end{equation}
Recall~\eqref{eq:xsub}, we have 
\begin{equation}\label{eq:delta1}
\hat{y}^k ={\rm prox}_{\tilde{g}_k}(\hat{y}^k - \nabla f(x_k)_{S_k} - ((Q_k)_{S_k} + \eta_kI_{\vert S_k\vert})(\hat{y}^k - y^k) - \hat{\ve}_k). 
\end{equation}
Using the nonexpansivity of \({\rm prox}_{\tilde{g}_k}\)~\cite[Th. 6.42]{B17},~\eqref{eq:ek1} and~\eqref{eq:delta1} yield 
\begin{equation}\label{eq:evare}
\|r_{S_k}^k(\hat{y}^k)\| \leq \|\hat{\ve}_k\|.
\end{equation}
Notice that \eqref{eq:ek1} also implies
\begin{equation}\label{eq:ek2}
r_{S_k}^k(\hat{y}^k) - \nabla f(x^k)_{S_k} - ((Q_k)_{S_k} + \eta_kI_{\vert S_k\vert})(\hat{y}^k - y^k) \in \partial \tilde{g}_k(\hat{y}^k - r_{S_k}^k(\hat{y}^k)). 
\end{equation}
Form the definition of \(\cG_{S_k}(y)\), we have
\begin{equation}\label{eq:gsxk}
\cG_{S_k}(y^k) - \nabla f(x^k)_{S_k} \in\partial \tilde{g}_k(y^k - \cG_{S_k}(y^k)).
\end{equation}
Using the monotonicity of \(\partial \tilde{g}_k\),~\eqref{eq:ek2} and~\eqref{eq:gsxk} yield
\[
\langle \cG_{S_k}(y^k) + ((Q_k)_{S_k} + \eta_kI_{\vert S_k\vert})(\hat{y}^k - y^k) - r_{S_k}^k(\hat{y}^k), y^k - \cG_{S_k}(y^k) - \hat{y}^k + r_{S_k}^k(\hat{y}^k)\rangle \geq 0.
\]
Combine the above inequality with \((Q_k)_{S_k} + (\eta_k - \mu) I_{\vert S_k\vert} \succeq 0\), we have 
\begin{align*}
\|\cG_{S_k}(y^k) - r_{S_k}^k(\hat{y}^k)\|^2 \leq \langle \cG_{S_k}(y^k) - r_{S_k}^k(\hat{y}^k), y^k - \hat{y}^k + ((Q_k)_{S_k} + \eta_kI_{\vert S_k\vert})(y^k - \hat{y}^k)\rangle.
\end{align*}
By Cauchy inequality and~\eqref{eq:dishfqk}, we have 
\[
\|\cG_{S_k}(y^k) - r_{S_k}^k(\hat{y}^k)\|\leq \|((Q_k)_{S_k} + (1 + \eta_k)I_{\vert S_k\vert})(y^k - \hat{y}^k)\| \leq \hat{\eta}\|d_k\|, 
\]
where \(\hat{\eta} = 1 + L_g + \zeta + \bar{\eta}\). Therefore, 
\[
\|\cG_{S_k}(y^k)\| \leq \|\cG_{S_k}(y^k) - r_{S_k}^k(\hat{y}^k)\| + \|r_{S_k}^k(\hat{y}^k)\| \leq (\hat{\eta} + \frac{\mu}{2})\|d_k\| = c_1\|d_k\|.
\]
The statement holds.  
\end{proof}

%%%%%%%%%%%%%%%%%%%%%%%%%%%%%%%%%%%%%%%%%%%%%%%%
% ******* Convergence of expected objective value
%%%%%%%%%%%%%%%%%%%%%%%%%%%%%%%%%%%%%%%%%%%%%%%%
\subsection{Convergence of expected objective value}

In this subsection, we show that the expected objective values sequence
generated by Algorithm~\ref{alg:pnewton} converges to the expectation of the limit of the objective values.

After \(k\) iterations, Algorithm~\ref{alg:pnewton} generates a random output \(\{(x^k, \varphi(x^k))\}\), which depends on the observed realization of the history of random index selection. Denote 
\[
\xi_k = \{S_0, S_1, \ldots, S_k\}
\]
and \(\mathbb{E}_{\xi_{-1}}[\varphi(x^0)] = \varphi(x^0)\). 

\begin{theorem}\label{th:cfv}
Suppose for any \(k\in\mathbb{N}\), \((Q_k)_{S_k} + (\eta_k - \mu) I_{\vert S_k\vert} \succeq 0\) and Assumption~\ref{assume:ncp} hold. Let \(\{x^k\}\) and \(\{d_k\}\) be the sequences generated by Algorithm~\ref{alg:pnewton}. 
Then the following statements hold:
\begin{itemize}
\item[(i)] \(\lim_{k\to\infty}\|d_k\| = 0\) and \(\lim_{k\to\infty}\varphi(x^k) = \varphi_{\xi_{\infty}}^*\) for some \(\varphi_{\xi_{\infty}}^*\in\mathbb{R}\), where \(\xi_{\infty} = \{S_0, S_1, \ldots\}\). 
\item[(ii)] \(\lim_{k\to\infty}\mathbb{E}_{\xi_k}[\|d_k\|] = 0\) and \(\lim_{k\to\infty}\mathbb{E}_{\xi_{k-1}}[\varphi(x^k)] = \mathbb{E}_{\xi_{\infty}}[\varphi_{\xi_{\infty}}^*]\). 
\end{itemize}
\end{theorem}
\begin{proof}
From~\eqref{eq:dvarphi}, we have   
\[
\varphi(x^{k+1}) \leq \varphi(x^k) \quad {\rm and}\quad \mathbb{E}_{\xi_k}[\varphi(x^{k+1})] \leq \mathbb{E}_{\xi_{k-1}}[\varphi(x^k)] \quad \forall k\geq 0.
\]
Hence, \(\{\varphi(x^k)\}\) and \(\{\mathbb{E}_{\xi_{k-1}}[\varphi(x^k)]\}\) are nonincreasing. Since 
\(\varphi\) is bounded below, so are \(\{\varphi(x^k)\}\) and \(\{\mathbb{E}_{\xi_{k-1}}[\varphi(x^k)]\}\). It follows that there exist some \(\varphi_{\xi_{\infty}}^*\), \(\widetilde{\varphi}^*\in\mathbb{R}\) such that 
\[
\lim_{k\to\infty}\varphi(x^k) = \varphi_{\xi_{\infty}}^* \quad {\rm and}\quad \lim_{k\to\infty}\mathbb{E}_{\xi_{k-1}}[\varphi(x^k)] = \widetilde{\varphi}^*. 
\]
In addition, it follows from~\eqref{eq:dvarphi} that \(\lim_{k\to\infty}\|d_k\| = 0\) and 
\[
\mathbb{E}_{\xi_k}[\varphi(x^{k+1})] \leq \mathbb{E}_{\xi_k}[\varphi(x^k)] - \frac{\tau}{2}\min\{1, \frac{\theta(\mu- \tau)}{L_g}\}\mathbb{E}_{\xi_k}[\|d_k\|^2], \quad \forall k\geq 0.
\]
Taking \(k\to \infty\) on both side of the above inequality and noting that 
\[
\lim_{k\to\infty}\mathbb{E}_{\xi_k}[\varphi(x^k)] = \lim_{k\to\infty}\mathbb{E}_{\xi_{k-1}}[\varphi(x^k)] = \widetilde{\varphi}^* = \lim_{k\to\infty}\mathbb{E}_{\xi_k}[\varphi(x^{k+1})],
\]
we conclude that \(\lim_{k\to\infty}\mathbb{E}_{\xi_{k}}[\|d_k\|^2] = 0\), which yields \(\lim_{k\to\infty}\mathbb{E}_{\xi_{k}}[\|d_k\|] = 0\). Notice that \(\varphi_* \leq \varphi(x^k) \leq \varphi(x^0)\), which implies that \(\vert \varphi(x^k)\vert \leq \max\{\vert \varphi(x^0)\vert, \vert \varphi_*\vert\}\) for all \(k\) and \(\{\varphi(x^k)\}\) is uniformly bounded. Then by~\cite[Theorem 5.4]{B95}, we have
\[
\mathbb{E}_{\xi_{\infty}}[\varphi_{\xi_{\infty}}^*]  = \lim_{k\to\infty}\mathbb{E}_{\xi_{\infty}}[\varphi(x^k)]. 
\] 
Together with \(\lim_{k\to\infty}\mathbb{E}_{\xi_{k-1}}[\varphi(x^k)] = \lim_{k\to\infty}\mathbb{E}_{\xi_{\infty}}[\varphi(x^k)]\), we have 
\[
\lim_{k\to\infty}\mathbb{E}_{\xi_{k-1}}[\varphi(x^k)] = \mathbb{E}_{\xi_{\infty}}[\varphi_{\xi_{\infty}}^*].
\]
\end{proof}

%%%%%%%%%%%%%%%%%%%%%%%%%%%%%%%%%%%%%%%%%%%%%%%%% *******  Global convergence
%%%%%%%%%%%%%%%%%%%%%%%%%%%%%%%%%%%%%%%%%%%%%%%%
\subsection{Global convergence of Algorithm~\ref{alg:pnewton}}

In this subsection, we present the global convergence of Algorithm~\ref{alg:pnewton} in terms of the minimum (expected) norm of \(\mathcal{G}(x^k)\) under different sampling assumptions. 

\begin{assumption}\label{assume:diffs}
Suppose \(\{S_k\}\) satisfies one of the following assumptions. Let \(p_i^k = \mathbb{P}(i\in S_k)\) for any \(k\in\mathbb{N}\).
\begin{itemize}
\item[\textbf{S1}.] The sampling \(S_k\) satisfies \(p_i^k \equiv p_i\) with \(p_i \geq p_{\rm min}> 0\), \(i = 1, \cdots, n\). %
\item[\textbf{S2}.] The sampling \(S_k\) satisfies 
\begin{equation}\label{eq:s4}
\|\mathcal{G}(x)_{S_k}\|^2 \geq c\|\mathcal{G}(x)\|^2, \quad \forall x
\end{equation}
for some \(c > 0\). 
\end{itemize}
\end{assumption}
\begin{remark}
Assumption~\ref{assume:diffs} \textbf{S1} holds with \(p_i^k = \frac{s}{n}\) if \(S_k\) is uniformly sampled with \(\vert S_k\vert \equiv s\) for some \(1\leq s \leq n\). 
The top-\(\mathbf{k}\) Gauss-Southwell sampling~\cite{FLD18} satisfies Assumption~\ref{assume:diffs} \textbf{S2} with \(c = \frac{\mathbf{k}}{n}\), where \(S_k\) consists of the indices corresponding to the top \(\mathbf{k}\) largest entries of \(\vert \mathcal{G}(x^k)\vert\). To avoid full gradient computation, we can adopt the following sampling strategy to determine \(S_k\): uniformly partition \([n]\) into \(q\) disjoint blocks \(\Omega_1, \ldots, \Omega_q\); at each iteration, randomly select one block and applying top-\(\mathbf{k}\) sampling to the gradient restricted to this block. Let \(\Omega_{s_k}\) be the block chosen at iteration \(k\). Then  
\[
\|\cG(x)_{S_k}\|^2 \geq \frac{\mathbf{k}}{\vert \Omega_{s_k}\vert}\|\cG(x)_{\Omega_{s_k}}\|^2
\]
by adaptivity of top-\(\mathbf{k}\) sampling. Let \(F(x) = \|\cG(x)_{\Omega_{s_k}}\|^2\). Then \(\mathbb{E}[F] = \frac{\vert \Omega_{s_k}\vert}{n}\|\cG(x)\|^2\). For any \(i \in [n]\), let \(x^{i'}\in\mathcal{L}_{\varphi}(x^0)\) be the vector that differs from \(x\) only at the \(i\)-th entry. Under Assumption~\ref{assume:ncp}, we have 
\[
\vert F(x) - F(x^{i'})\vert = \vert \|\cG(x)_{\Omega_{s_k}}\|^2 - \|\cG(x^{i'})_{\Omega_{s_k}}\|^2\vert \leq (\|\cG(x)\| + \|\cG(x^{i'})\|)^2 \leq 4\bar{\epsilon}_1^2. 
\]
By McDiarmid's inequality~\cite{M1989}, for any \(\varepsilon > 0\), we have 
\[
\mathbb{P}(\vert F(x) - \mathbb{E}[F(x)] \vert \geq \varepsilon) \leq 2\exp(-\frac{\varepsilon^2}{8n\bar{\epsilon}_1^4}).
\]
Setting \(\varepsilon = \delta_0 \frac{\vert \Omega_{s_k}\vert}{n}\|\cG(x)\|^2\) with \(\delta_0 \in (0, 1)\) yields
\[
\mathbb{P}\left(\big\vert \|\cG(x)_{\Omega_{s_k}}\|^2 - \frac{\vert \Omega_{s_k}\vert}{n}\|\cG(x)\|^2 \big\vert\right) \geq \delta_0 \frac{\vert \Omega_{s_k}\vert}{n}\|\cG(x)\|^2 \leq 2\exp(-\frac{\delta_0^2\vert \Omega_{s_k}\vert^2\|\cG(x)\|^4}{8n^3\bar{\epsilon}_1^4}).
\]
Thus, 
\[
\|\cG(x)_{\Omega_{s_k}}\|^2 \geq (1 - \delta_0)\frac{\vert \Omega_{s_k}\vert}{n}\|\cG(x)\|^2
\]
with probability at least \(1 - 2\exp(-\frac{\delta_0^2\vert \Omega_{s_k}\vert^2\|\cG(x)\|^4}{8n^3\bar{\epsilon}_1^4})\), which implies 
\[
\|\cG(x)_{S_k}\|^2 \geq \frac{\mathbf{k}}{n}(1 - \delta_0)\|\cG(x)\|^2
\]
with the same probability.
\end{remark}

\begin{proposition}\label{prop:gsg}
Suppose Assumption~\ref{assume:ncp} holds. 
Let \(\{x^k\}\) and \(\{y^k\}\) be the sequences generated by Algorithm~\ref{alg:pnewton}. 
%Then the following statements hold. 
Then, under Assumption~\ref{assume:diffs}~\textbf{S1}, we have 
\begin{equation}\label{eq:s3}
\mathbb{E}_{\xi_k}[\|\cG_{S_k}(y^k)\|^2] \geq p_{\min}\mathbb{E}_{\xi_k}[\|\cG(x^k)\|^2], \quad \forall k\in\mathbb{N}.
\end{equation}
%\begin{itemize}
%\item[(i)] Under Assumption~\ref{assume:diffs}~\textbf{S0}, 
%\begin{equation}\label{eq:s2}
%\mathbb{E}_{\xi_k}[\|\cG_{S_k}(y^k)\|^2] \geq \min_{1 \leq i\leq n}\{p_i^k\}\mathbb{E}_{\xi_k}[\|\cG(x^k)\|^2], \quad \forall k\in\mathbb{N}. 
%\end{equation}
%\item[(ii)] Under Assumption~\ref{assume:diffs}~\textbf{S1}, 
%\begin{equation}\label{eq:s3}
%\mathbb{E}_{\xi_k}[\|\cG_{S_k}(y^k)\|^2] \geq p_{\min}\mathbb{E}_{\xi_k}[\|\cG(x^k)\|^2], \quad \forall k\in\mathbb{N}.
%\end{equation}
%\end{itemize}
\end{proposition}
\begin{proof}
Recall Proposition~\ref{prop:gsy}, \(\cG_{S_k}(y^k)\) is a subvector of \(\cG(x^k)\) corresponding to \(S_k\), which leads to 
\begin{align*}%\begin{equation}\label{eq:eyex}
\mathbb{E}_{\xi_k}[\|\cG_{S_k}(y^k)\|^2] =& \sum_{i=1}^n\mathbb{E}_{\xi_k}[(\cG(x^k)_i\delta_{S_k}^i)^2] = \sum_{i=1}^n\mathbb{E}_{\xi_k}[(\cG(x^k)_i)^2p_i^k] \geq p_{\min}\sum_{i=1}^n\mathbb{E}_{\xi_k}[(\cG(x^k)_i)^2]\\
=& p_{\min}\sum_{i=1}^n\mathbb{E}_{\xi_k}[\|\cG(x^k)\|^2],
\end{align*}%\end{equation}
where \(\delta_{S_k}^i = 1\) if \(i \in S_k\) and \(\delta_{S_k}^i = 0\) if \(i \notin S_k\) and the inequality holds since \(p_i^k \geq p_{\min}\) under Assumption~\ref{assume:diffs}~\textbf{S1}.

%(i) Under Assumption~\ref{assume:diffs}~\textbf{S0}, we have \(p_i^k \geq \min_{1\leq i\leq n}\{p_i^k\}\). Hence,~\eqref{eq:s2} holds from~\eqref{eq:eyex}.

%(ii)~\eqref{eq:s3} holds by noting that \(p_i^k \geq p_{\min}\) under Assumption~\ref{assume:diffs}~\textbf{S1}.
\end{proof}

\begin{theorem}\label{th:clusterpoint}
Suppose for any \(k\in\mathbb{N}\), \((Q_k)_{S_k} + (\eta_k - \mu) I_{\vert S_k\vert} \succeq 0\), the boundedness~\eqref{eq:dishfqk}, and Assumption~\ref{assume:ncp} hold. Let \(\{x^k\}\) be the sequence generated by Algorithm~\ref{alg:pnewton} and \(\omega(x^0)\) be the cluster points set of \(\{x^k\}\). 
Then the following statements hold. 
\begin{itemize}
\item[(i)] Under Assumption~\ref{assume:diffs}~\textbf{S1}, we have 
\begin{equation}\label{eq:expectstationary}
\lim_{k\to\infty}\mathbb{E}_{\xi_k}[\|\cG(x^k)\|] = 0. 
\end{equation}
\item[(ii)] Under Assumption~\ref{assume:diffs}~\textbf{S2}, we have 
\begin{equation}\label{eq:stationary}
\lim_{k\to\infty}\|\cG(x^k)\| = 0, 
\end{equation}
that is, \(\omega(x^0)\subseteq \mathcal{S}^*\). Moreover, \(\omega(x^0)\) is nonempty and compact. 
\end{itemize}
\end{theorem}
\begin{proof}
(i) Under Assumption~\ref{assume:diffs}~\textbf{S1}, from~\eqref{eq:s3} and Lemma~\ref{lem:ngk}, we have 
\[
\mathbb{E}_{\xi_k}[\|\mathcal{G}(x^k)\|^2] \leq \frac{1}{p_{\min}}\mathbb{E}_{\xi_k}[\|\cG_{S_k}(y^k)\|^2] \leq \frac{c_1^2}{p_{\min}}\mathbb{E}_{\xi_k}[\|d_k\|^2].  
\]
Hence, \eqref{eq:expectstationary} holds.

(ii) Under Assumption~\ref{assume:diffs}~\textbf{S2}, from~\eqref{eq:s4} and Lemma~\ref{lem:ngk}, we have 
\begin{equation}\label{eq:xd}
\|\cG(x^k)\|^2 \leq \frac{1}{c}\|\cG_{S_k}(y^k)\|^2 \leq \frac{c_1^2}{c}\|d_k\|^2. 
\end{equation}
\eqref{eq:stationary} holds by taking \(k\) to \(\infty\) on the both side of the above inequality and combining with Theorem~\ref{th:cfv} (i). Hence, we have \(\omega(x^0) \subseteq \mathcal{S}^*\). 
\(\omega(x^0)\) is nonempty and bounded since \(\{x^k\}\subseteq \mathcal{L}_{\varphi}(x^0)\) is bounded. 
The continuity of \(\mathcal{G}\) ensures the closedness of \(\omega(x^0)\) and \(\|\mathcal{G}(\bar{x})\| = 0\) for any \(\bar{x}\in\omega(x^0)\).
\end{proof}

\begin{theorem}\label{th:limitsppn}
Suppose for any \(k\in\mathbb{N}\), \((Q_k)_{S_k} + (\eta_k - \mu) I_{\vert S_k\vert} \succeq 0\), the boundedness~\eqref{eq:dishfqk}, and Assumption~\ref{assume:ncp} hold.  Let \(\{x^k\}\) be the sequence generated by Algorithm~\ref{alg:pnewton}. Then, the following statements hold. 
\begin{itemize}
\item[(i)] Under Assumption~\ref{assume:diffs}~\textbf{S1}, we have  
\begin{equation}\label{eq:minexpectcglsnew12}
\min_{0\leq k\leq K}\mathbb{E}_{\xi_k}[\|\cG(x^{k})\|^2]  \leq \frac{1}{p_{\min}}\cdot\frac{2c_1^2(\varphi(x^0) - \varphi_*)}{\tau\min\{1, \frac{\theta(\mu - \tau)}{L_g}\}K}.
\end{equation}
\item[(ii)] Under Assumption~\ref{assume:diffs}~\textbf{S2}, we have
\begin{equation}\label{eq:mincglsnew}
\min_{0\leq k\leq K}\|\cG(x^{k})\|^2  \leq \frac{1}{c}\cdot\frac{2c_1^2(\varphi(x^0) - \varphi_*)}{\tau\min\{1, \frac{\theta(\mu - \tau)}{L_g}\}K}. 
\end{equation}
\end{itemize}
\end{theorem}
\begin{proof}
From Lemmas~\ref{lem:alpdphi} and~\ref{lem:ngk}, we have 
\[
\varphi(x^{k+1}) \leq \varphi(x^k) - \frac{\tau}{2c_1^2}\min\{1, \frac{\theta(\mu - \tau)}{L_g}\}
\|\cG_{S_k}(y^k)\|^2,\quad \forall k\in\mathbb{N},
\]
which yields 
\begin{equation}\label{eq:psistar}
\varphi_* \!\leq\! \mathbb{E}_{\xi_{K\!-\!1}}\![\varphi(x^{K})] \!\leq\! \mathbb{E}_{\xi_{K\!-\!1}}\![\varphi(x^{K\!-\!1})]  \!-\! \frac{\tau}{2c_1^2}\min\{1, \frac{\theta(\mu \!-\! \tau)}{L_g}\}\mathbb{E}_{\xi_{K\!-\!1}}[\|\cG_{S_{K\!-\!1}}\!(y^{K\!-\!1})\|^2]. 
\end{equation}

(i) Under Assumption~\ref{assume:diffs}~\textbf{S1}, from~\eqref{eq:s3} and \eqref{eq:psistar}, we have  
\begin{align*}
\varphi_* %\leq& \mathbb{E}_{\xi_{K-2}}[\varphi(x^{K-1})]  - \frac{\tau p_{\min}}{2c_1^2}\min\{1, \frac{\theta(\mu - \tau)}{L_g}\}\mathbb{E}_{\xi_k}[\|\mathcal{G}(x^{K-1})\|^2]\nonumber\\
\leq& \mathbb{E}_{\xi_{-1}}[\varphi(x^{0})] - \frac{\tau p_{\min}}{2c_1^2}\min\{1, \frac{\theta(\mu - \tau)}{L_g}\}\sum_{k=0}^{K-1}\mathbb{E}_{\xi_k}[\|\mathcal{G}(x^k)\|^2].
\end{align*}
Hence, we have 
\[
\frac{\tau p_{\min}}{2c_1^2}\min\{1, \frac{\theta(\mu - \tau)}{L_g}\}\sum_{k=0}^{K-1}\mathbb{E}_{\xi_k}[\|\mathcal{G}(x^k)\|^2] \leq \varphi(x^0) - \varphi_*,
\]
which yields \eqref{eq:minexpectcglsnew12}. 

(ii) Under Assumption~\ref{assume:diffs}~\textbf{S2}, from~\eqref{eq:s4},~\eqref{eq:psistar} becomes to  
\begin{align*}
\varphi_* \leq& \varphi(x^K) - \frac{\tau c}{2c_1^2}\min\{1, \frac{\theta(\mu - \tau)}{L_g}\}
\|\cG_{S_k}(x^{K - 1})\|^2\\
\leq& \varphi(x^0) - \frac{\tau c}{2c_1^2}\min\{1, \frac{\theta(\mu - \tau)}{L_g}\}
\sum_{k=0}^{K - 1}\|\cG_{S_k}(x^k)\|^2. 
\end{align*}
Hence, \eqref{eq:mincglsnew} holds.
\end{proof}
%From Theorem~\ref{th:limitsppn} (ii), if \(\|\cG(x^k)\| \geq \varepsilon\) for all \(0 \leq k \leq K\), then we have 
%\[
%K \leq \left\lceil \frac{1}{c}\cdot\frac{2L_gc_1^2(\varphi(x^0) - \varphi_*)}{\tau\min\{L_g, %\theta(\mu - \tau)\}\varepsilon^2}\right\rceil.
%\]
Theorems~\ref{th:clusterpoint} (ii) and~\ref{th:limitsppn} (ii) match \cite[Theorem 1]{Z24} for IPNM. 

%%%%%%%%%%%%%%%%%%%%%%%%%%%%%%%%%%%%%%%%%%%%
% ******* The  SBCPN Method When \(L_{S_k}\) is Known
%%%%%%%%%%%%%%%%%%%%%%%%%%%%%%%%%%%%%%%%%%%%
\section{The  SBCPN Method When \(L_{S_k}\) is Known}\label{appendix: lhisgiven}

In this section, we first show that if \(\) (line 3 in Algorithm~\ref{alg:pnewton}) is replaced by 

\begin{equation}\label{eq:qsketa}
(Q_k)_{S_k} + (\eta_k -  \vartheta)I_{\vert S_k\vert} \succeq  0 \quad {\rm and}\quad Q_k + (\eta_k - L_{S_k} - \mu) I_n \succeq 0, \quad \forall k\in\mathbb{N},
\end{equation}
where \(\vartheta \geq 1.1\mu\times \max\{\frac{1}{2}(1 + 2\zeta + 3L_g + \mu), \frac{1}{2-\mu}(1 + 2\zeta + 2L_g)\}\), then Algorithm~\ref{alg:pnewton} is well-defined with unit step size. Without loss of generality, we can assume that 
\[
0 \leq \eta_k \leq \bar{\eta}:= \max\{\mu + 2L_g + \zeta, \vartheta + {L_g} + \zeta\}, \quad \forall k\in\mathbb{N}. 
\]

We present the SBCPN method for this case in Algorithm~\ref{alg:pmm}. 
\begin{algorithm*}[h!]
\caption{SBCPN method with unit step size.}\label{alg:pmm}
\begin{algorithmic}[1]
\Require{\(x^0\in{\rm dom}g\), \(\bar{\eta}\),  and \(\mu \in (0, 1]\), distribution \(\mathcal{D}\) of random index set.}
\For{\(k = 0, 1, \ldots, \)}
\State{sample \(S_k\) from \(\cal{D}\);}
\State{set \(\eta_k\in(0, \bar{\eta}]\) and \(Q_k\) satisfy~\eqref{eq:qsketa};}
\State{let \(y^k = x^k_{S_k}\), compute}
\[
\hat{y}^k \approx \argmin_{y}\{q^k_{S_k}(y)\}\quad {\rm with}~\varsigma_k\in \partial q^k_{S_k}(\hat{y}^k)~{\rm satisfies}~\eqref{eq:errvek}.
\] 
\State{set \(x^{k+1}\) with \(x^{k+1}_{S_k} = \hat{y}^k\) and \(x^{k+1}_{\overline{S_k}} = x^k_{\overline{S_k}}\).} 
\EndFor\\
\Return{\(\{x^k\}\)}
\end{algorithmic}
\end{algorithm*}

\begin{remark}
As can be seen from the subsequent proof of Theorem~\ref{th:limitspwithoutls} (b), the main role of the constant \(L_{S_k}\) is to obtain \(\varphi(x^k) - \varphi(x^{k+1}) \geq \frac{\mu}{2}\|x^{k+1} - x^k\|^2\) by using
\[
f(x^{k+1}) - f(x^k) - \langle \nabla f(x^k), x^{k+1} - x^k\rangle \leq \frac{L_{S_k}}{2}\|x^{k+1} - x^k\|^2. 
\]
Define \(\tilde{f}(x; L) = f(x) - f(x^k) - \langle \nabla f(x^k), x - x^k\rangle - \frac{L}{2}\|x^{k+1} - x^k\|^2\) and \(q_{S_k}^k(y; Q^k(L)) = l_{S_k}^k(y) + \frac{1}{2}\langle Q^k(L)_{S_k}(y - y^k), y - y^k\rangle\), where \(Q^k(L) = Q_k + \eta_kI_n\) satisfies \(Q_k(L) - (L + \mu)I_n \succeq 0\). If \(L_{S_k}\) is unknown, then a line search strategy on \(L_{S_k}\) (Algorithm~\ref{alg:sbcpnlslsk}) can be employed. Algorithm~\ref{alg:sbcpnlslsk} is well-defined under Assumption~\ref{assume:ncp}. 
\begin{algorithm*}[h!]
\caption{SBCPN method with line search on \(L_{S_k}^k\).}\label{alg:sbcpnlslsk}
\begin{algorithmic}[1]
\Require{\(x^0\in{\rm dom}g\), \(\bar{\eta}\),  and \(\mu \in (0, 1]\), distribution \(\mathcal{D}\) of random index set. \(L^0_{S_0} = \overline{L}\).}
\For{\(k = 0, 1, \ldots, \)}
\State{sample \(S_k\) from \(\cal{D}\);}
\State{set \(\eta_k\in(0, \bar{\eta}]\) and \(Q_k\) satisfy \(Q_k + (\eta_k - L_{S_k}^k - \mu)I_n \succeq 0\);}
\State{let \(y^k = x^k_{S_k}\), compute}
\[
\hat{y}^k(L_{S_k}^k) \approx \argmin_{y}\{q^k_{S_k}(y)\}\quad {\rm with}~\varsigma_k\in \partial q^k_{S_k}(\hat{y}^k)~{\rm satisfies}~\eqref{eq:errvek}.
\] 
\State{set \(x(L_{S_k}^k)\) with \(x(L_{S_k}^k)_{S_k} = \hat{y}^k(L_{S_k}^k)\) and \(x(L_{S_k}^k)_{\overline{S_k}} = x^k_{\overline{S_k}}\).} 
\While{\(\tilde{f}(x(L_{S_k}^k), L_{S_k}^k) > 0\)}
\State{\(L_{S_k}^k = 2L_{S_k}^k\), set \(\eta_k\in(0, \bar{\eta}]\) and \(Q_k\) satisfy \(Q_k + (\eta_k - L_{S_k}^k - \mu)I_n \succeq 0\).}
\State{compute}
\[
\hat{y}^k(L_{S_k}^k) \approx \argmin_{y}\{q^k_{S_k}(y)\}\quad {\rm with}~\varsigma_k\in \partial q^k_{S_k}(\hat{y}^k)~{\rm satisfies}~\eqref{eq:errvek}.
\] 
\State{set \(x(L_{S_k}^k)\) with \(x(L_{S_k}^k)_{S_k} = \hat{y}^k(L_{S_k}^k)\) and \(x(L_{S_k}^k)_{\overline{S_k}} = x^k_{\overline{S_k}}\).} 
\EndWhile
\State{set \(x^{k+1} = x(L_{S_k}^k)\).}
\State{set \(L_{S_{k+1}}^{k+1} = \max\{\frac{1}{2}L_{S_k}^k, \overline{L}\}\).}
\EndFor\\
\Return{\(\{x^k\}\)}
\end{algorithmic}
\end{algorithm*}

\end{remark}

%%%%%%%%%%%%%%%%%%%%%%%%%%%%%%%%%%%%%%%%%%%%%%%%
% ******** Global convergence
%%%%%%%%%%%%%%%%%%%%%%%%%%%%%%%%%%%%%%%%%%%%%%%%
\subsection{Global convergence of Algorithm~\ref{alg:pmm}}

Similar results to Lemma~\ref{lem:ngk} and Theorem~\ref{th:limitsppn} hold for Algorithm~\ref{alg:pmm}.
\begin{theorem}\label{th:limitspwithoutls}
Suppose that  Assumption~\ref{assume:ncp},~\eqref{eq:dishfqk}, and~\eqref{eq:qsketa} are satisfied. Let \(\{x^k\}\) and \(\{y^k\}\) be the sequences generated by Algorithm~\ref{alg:pmm}.  
Then the following statements hold. 
\begin{itemize}
\item[(a)] \(\|\cG_{S_k}(y^k)\| \leq c_1\|x^{k+1} - x^k\|\), where \(c_1 = 1 + L_g + \zeta + \bar{\eta} + \frac{\mu}{2}\). 
\item[(b)] \(\varphi(x^k) - \varphi(x^{k+1}) \geq  \frac{\mu}{2}\|x^{k+1} - x^k\|^2\). 
\item[(c)] \(\lim_{k\to\infty}\|x^{k+1} - x^k\| = 0\) and \(\lim_{k\to\infty}\varphi(x^k) = \varphi_{\xi_{\infty}}^*\) for some \(\varphi_{\xi_{\infty}}^*\in\mathbb{R}\), where \(\xi_{\infty} = \{S_0, S_1, \ldots\}\). 
\item[(d)] \(\lim_{k\to\infty}\mathbb{E}_{\xi_k}[\|x^{k+1} - x^k\|] = 0\) and \(\lim_{k\to\infty}\mathbb{E}_{\xi_{k-1}}[\varphi(x^k)] = \mathbb{E}_{\xi_{\infty}}[\varphi_{\xi_{\infty}}^*]\). 
\item[(e)] Suppose Assumption~\ref{assume:diffs}~\textbf{S1} holds. We have \(\lim_{k\to\infty}\mathbb{E}_{\xi_k}[\|\cG(x^k)\|] = 0\) and
\[
\min_{0\leq k\leq K}\mathbb{E}_{\xi_k}[\|\cG(x^{k})\|^2]  \leq \frac{1}{p_{\min}}\cdot\frac{2c_1^2(\varphi(x^0) - \varphi_*)}{\mu K}.  
\]
\item[(f)] Suppose Assumption~\ref{assume:diffs}~\textbf{S2} holds. We have \(\lim_{k\to\infty}\|\cG(x^k)\| = 0\). 
Let \(\omega(x^0)\) be the cluster points set of \(\{x^k\}\). Then \(\omega(x^0)\subseteq \mathcal{S}^*\) is nonempty and compact. Moreover, 
\[
\min_{0\leq k\leq K}\|\cG(x^{k})\|^2  \leq \frac{1}{c^2}\cdot\frac{2c_1^2(\varphi(x^0) - \varphi_*)}{\mu K}. 
\]
\end{itemize}
\end{theorem}
\begin{proof}
The proof is given in Appendix~\ref{app:proofsbcpnusz}.
\end{proof}

%%%%%%%%%%%%%%%%%%%%%%%%%%%%%%%%%%%%%%%%%%%%%%%%
\subsection{Local convergence for nonconvex problem}

Next, we establish the superlinear local convergence rate of Algorithm~\ref{alg:pmm} for nonconvex Problem~\eqref{eq:ncp}. 
We assume that \(f\) and \(g\) satisfy the following assumption.
\begin{assumption}\label{assume:gold}
\begin{itemize}
\item[(i)] \(f: \mathbb{R}^n\to(-\infty, +\infty]\) is twice continuously differentiable on an open set \(\Omega_2\) containing the effective domain \({\rm dom}g\) of \(g\), \(\nabla f\) is \(L_g\)-Lipschitz continuous over \(\Omega_2\); \(\nabla^2 f\) is \(L_C\)-Lipschitz continuous over an open neighborhood of \(\omega(x^0)\) with radius \(\epsilon_0\) for some \(\epsilon_0 \in (0, 1)\). 
\item[(ii)] \(g: \mathbb{R}^n\to(-\infty, +\infty]\) is block-coordinate separable, that is, \(g\) takes the form of 
\[
g(x) = \sum_{i=1}^{n_l}\psi_i(x_{I_i}),
\]
where \(\cup_{i=1}^{n_l}I_i = [n]\), \(I_i\cap I_j = \emptyset\), \(\forall i\neq j\), \(x_{I_i}\) denotes the subvector of \(x\) indexed by the index set \(I_i\), \(\psi_i: \mathbb{R}^{\vert I_i\vert}\to(-\infty, +\infty]\) is a proper closed convex function, \(\min_{z_i}\{\psi_i(z_i) + \frac{1}{2}\|z_i - u\|^2\}\) is efficiently solvable, and \(0\in {\rm dom}\psi_i\), \(i = 1, \ldots, n_l\). 
\item[(iii)] For any \(x^0 \in {\rm dom}g\), the level set \(\mathcal{L}_{\varphi}(x^0) = \{x\vert \varphi(x) \leq \varphi(x^0)\}\) is bounded. 
\end{itemize}
\end{assumption}
Under Assumption~\ref{assume:gold} (i) and (ii), \(\nabla f(\cdot) + \partial g(\cdot)\) is outer semicontinuous over \({\rm dom}g\)~\cite[Prop. 8.7]{RW04}. Hence, the stationary set \(\mathcal{S}^*\) is closed. We have \(\varphi \equiv \varphi_*:=\lim_{k\to\infty}\varphi(x^k)\) on \(\omega(x^0)\) by the continuity of \(\varphi\). We assume that the residual mapping \(\mathcal{G}\) satisfies the metric \(q\)-subregularity property.
\begin{assumption}\label{assume:errorboundold}
For any \(\bar{x}\in\omega(x^0)\), the metric \(q\)-subregularity at \(\bar{x}\) with \(q > 1\) on \(\mathcal{S}^*\) holds, that is, there exist \(\epsilon \in (0, 1)\) and \(\kappa > 0\) such that 
\[
{\rm dist}(x, \mathcal{S}^*) \leq \kappa \|\mathcal{G}(x)\|^q, \quad \forall x\in\mathbb{B}(\bar{x}, \epsilon). 
\]
\end{assumption}
Assumption~\ref{assume:errorboundold} has been used to analyze the local convergence rate of IPNMs~\cite{MYZZ22,LPWY23,Z24}. For any \(k \in\mathbb{N}\), define \(\bar{y}^k = \argmin_y\{q_{S_k}^k(y)\}\). We first establish the error bound between \(\bar{y}^k\) and \(\hat{y}^k\). 
\begin{lemma}\label{lem:42old}
Suppose that~\eqref{eq:qsketa} is satisfied. Let \(\{\hat{y}^k\}\) be the sequence generated by Algorithm~\ref{alg:pmm}. Then we have  
\begin{equation*}%\label{eq:errhbxk}
\|\hat{y}^k - \bar{y}^k\| \leq \frac{(1 + \bar{\eta} + \zeta + L_g)\mu}{2\vartheta}\|x^{k+1} - x^k\|, \quad \forall k\in\mathbb{N}.
\end{equation*}
\end{lemma}
\begin{proof}
By the definition of \(\bar{y}^k\) and using the first-order optimality condition, we have 
\begin{equation}\label{eq:optbarxold}
-\nabla f(x^k)_{S_k} - (Q_k)_{S_k}(\bar{y}^k - y^k) - \eta_k(\bar{y}^k - y^k) \in \partial \tilde{g}_k(\bar{y}). 
\end{equation}
Combining with~\eqref{eq:ek2}, using the monotonicity of \(\partial \tilde{g}_k\), we have 
\begin{align*}
0 \leq & \langle \hat{y}^k - r_{S_k}^k(\hat{y}^k) - \bar{y}^k, r_{S_k}^k(\hat{y}^k) + ((Q_k)_{S_k} + \eta_k I_{\vert S_k\vert})(\bar{y}^k - \hat{y}^k)\rangle\\
\leq \!&-\!\langle\! (\!(Q_k)_{S_k} \!+\! (1 \!+\! \eta_k) I_{\vert S_k\vert})(\bar{y}^k \!-\! \hat{y}^k), r_{S_k}^k(\hat{y}^k)\rangle \!+\! \langle \hat{y}^k \!-\! \bar{y}^k,  ((Q_k)_{S_k} \!+\! \eta_k I_{\vert S_k\vert})(\bar{y}^k \!-\! \hat{y}^k)\rangle. 
\end{align*}
From~\eqref{eq:qsketa} and Cauchy inequality, we have 
\[
\vartheta \|\hat{y}^k - \bar{y}^k\| \leq (1 + \bar{\eta} + L_g + \zeta)\|r_{S_k}^k(\hat{y}^k)\| \leq \frac{(1 + \bar{\eta} + \zeta + L_g)\mu}{2}\|d_k\|,
\]
where the last inequality follows from~\eqref{eq:evare} with \(\hat{\epsilon}_k = -\varsigma_k\). The statement holds. 
\end{proof}
Next, we estimate the error bound between \(y^k\) and \(\bar{y}^k\) in terms of \({\rm dist}(x^k, \mathcal{S}^*)\). 

\begin{lemma}\label{lemma:43old}
Consider any \(\bar{x}\in\omega(x^0)\). Suppose that Assumption~\ref{assume:gold} and~\eqref{eq:qsketa} are satisfied.
Let \(\{x^k\}\) and \(\{\hat{y}^k\}\) be the sequence generated by Algorithm~\ref{alg:pmm}.  Then, for all \(x^k \in \mathbb{B}(\bar{x}, \epsilon_0/2)\), we have
\[
\|y^k - \bar{y}^k\| \leq  \frac{L_C}{\vartheta}{\rm dist}^2(x^k, \mathcal{S}^*) + (1 + \frac{\zeta + \bar{\eta}}{\vartheta}){\rm dist}(x^k, \mathcal{S}^*).
\]
\end{lemma}
\begin{proof}
For any \(x^k\in\mathbb{B}(\bar{x}, \epsilon_0/2)\), let \(\Pi_{\mathcal{S}^*}(x^k)\)
%\(\Pi_{\mathcal{S}^*}(x^k) \triangleq \{x\in\mathcal{S}^*\vert {\rm dist}(x, x^k) \leq {\rm dist}(z, x^k),~\forall z\in\mathcal{S}^*\}\) 
be the projection set of \(x^k\) onto \(\mathcal{S^*}\). Then \(\Pi_{\mathcal{S}^*}(x^k) \neq \emptyset\) since \(\mathcal{S}^*\) is closed. Pick \(x^{k, *}\in \Pi_{\mathcal{S}^*}(x^k)\). Notice that \(\bar{x} \in\omega(x^0)\subseteq \mathcal{S}^*\), we have 
\[
\|x^{k, *} -\bar{x}\| \leq \|x^{k, *} - x^k\| + \|x^k - \bar{x}\| \leq 2\|x^k - \bar{x}\| \leq \epsilon_0,
\]
which implies that \(x^{k, *}\in\mathbb{B}(\bar{x}, \epsilon_0)\). Hence, \((1- t)x^k + tx^{k, *}\in\mathbb{B}(\bar{x}, \epsilon_0)\cap {\rm dom}g\) for all \(t\in[0,1]\). Notice that \(x^{k, *} \in\mathcal{S}^*\), we have \(-\nabla f(x^{k, *}) \in \partial g(x^{k, *})\). Moreover, \(-\nabla f(x^{k, *})_{S_k} \in \partial \tilde{g}_k(x^{k, *}_{S_k})\) under Assumption~\ref{assume:gold} (ii). Combined with~\eqref{eq:optbarxold}, using the monotonicity of \(\partial \tilde{g}_k\), we have 
\begin{align*}
0 \leq& \langle x^{k, *}_{S_k} - \bar{y}^k, -\nabla f(x^{k, *})_{S_k} + \nabla f(x^k)_{S_k} + ((Q_k)_{S_k} + \eta_kI_{\vert S_k\vert})(\bar{y}^k - y^k)\rangle\\
=& \langle x^{k, *}_{S_k} - \bar{y}^k, -\nabla f(x^{k, *})_{S_k} + \nabla f(x^k)_{S_k} + ((Q_k)_{S_k} + \eta_kI_{\vert S_k\vert})(x_{S_k}^{k, *} - y^k)\rangle\\
&+\langle x^{k, *}_{S_k} - \bar{y}^k, ((Q_k)_{S_k} + \eta_kI_{\vert S_k\vert})(\bar{y}^k - x_{S_k}^{k, *})\rangle.
\end{align*}
By \eqref{eq:qsketa} and Cauchy inequality, we have 
\begin{align*}
\|x^{k, *}_{S_k} - \bar{y}^k\| \leq& \frac{1}{\vartheta}\| \nabla f(x^k)_{S_k} - \nabla f(x^{k, *})_{S_k} + ((Q_k)_{S_k} + \eta_kI_{\vert S_k\vert})(x_{S_k}^{k, *} - y^k)\|\\
=& \frac{1}{\vartheta}\|E_k^\top\int_0^1[Q_k + \eta_kI_n - \nabla^2f(x^k + t(x^{k, *} - x^k))](x^{k, *} - x^k)_{[S_k]}dt\|\\
\leq& \frac{1}{\vartheta}\|\int_0^1[Q_k + \eta_kI_n - \nabla^2f(x^k + t(x^{k, *} - x^k))](x^{k, *} - x^k)_{[S_k]}dt\|\\
\leq& \frac{1}{\vartheta}\|\int_0^1[\nabla^2f(x^k) - \nabla^2f(x^k + t(x^{k, *} - x^k))](x^{k, *} - x^k)_{[S_k]}dt\| \\
&+ \frac{1}{\vartheta}\|\int_0^1[Q_k - \nabla^2f(x^k) + \eta_kI_n](x^{k, *} - x^k)_{[S_k]}dt\|\\
\leq& \frac{L_C}{2\vartheta}\|x^{k, *} - x^k\|^2 + \frac{\zeta + \bar{\eta}}{\vartheta}\|x^{k, *} - x^k\|,
\end{align*}
where \(E_k\in\mathbb{R}^{n\times \vert S_k\vert}\) is the column submatrix of \(I_n\) that corresponds to \(S_k\) and the last inequality follows from Assumption~\ref{assume:gold} (i), \(\|(x^{k, *} - x^k)_{[S_k]}\| \leq \|x^{k, *} - x^k\|\),  and~\eqref{eq:dishfqk}. Therefore, 
\begin{align*}
\|y^k - \bar{y}^k\| \leq& \|y^k - x^{k, *}_{S_k} \| + \|x^{k, *}_{S_k} - \bar{y}^k\|\\
\leq & \|x^k - x^{k, *}\| + \frac{L_C}{\vartheta}\|x^{k, *} - x^k\|^2 + \frac{\zeta + \bar{\eta}}{\vartheta}\|x^{k, *} - x^k\|\\
\leq& \frac{L_C}{\vartheta}\|x^k - x^{k, *}\|^2 + (1 + \frac{\zeta + \bar{\eta}}{\vartheta})\|x^k - x^{k, *}\|. 
\end{align*}
The statement holds. 
\end{proof}
Invoking Lemmas~\ref{lem:42old} and \ref{lemma:43old}, for all \(x_k\in\mathbb{B}(\bar{x}, \epsilon_0/2)\), we have
\begin{align*}
\|x^{k+1} - x^k\| =& \|\hat{y}^k - y^k\| \leq \|\hat{y}^k - \bar{y}^k\| + \|\bar{y}^k - y^k\|\\
\leq& \frac{(1 \!+\! \bar{\eta} \!+\! \zeta \!+\! L_g)\mu}{2\vartheta}\|x^{k+1} \!-\! x^k\| \!+\! \frac{L_C}{\vartheta}{\rm dist}^2(x^k\!, \mathcal{S}^*) \!+\! (1 \!+\! \frac{\zeta \!+\! \bar{\eta}}{\vartheta}){\rm dist}(x^k\!, \mathcal{S}^*),
\end{align*}
which yields 
\begin{equation}\label{eq:ndkold}
\|x^{k+1} - x^k\| \leq \frac{2L_C}{\tilde{\tilde{\eta}}}{\rm dist}^2(x^k, \mathcal{S}^*) + \frac{2(\vartheta + \zeta + \bar{\eta})}{\tilde{\tilde{\eta}}}{\rm dist}(x^k, \mathcal{S}^*),
\end{equation}
where \(\tilde{\tilde{\eta}} = 2\vartheta - \mu(1 + \bar{\eta} + \zeta + L_g) \geq 2\vartheta - (1 + 2\zeta + 2L_g + \vartheta) > 0\). Therefore, \(\|x^{k+1} - x^k\| = \mathcal{O}({\rm dist}(x^k, \mathcal{S}^*))\).

\begin{theorem}\label{th:lcrxkold}
Suppose that Assumptions~\ref{assume:diffs} \textbf{S2},~\ref{assume:gold}, and~\ref{assume:errorboundold}, the boundedness~\eqref{eq:dishfqk}, and~\eqref{eq:qsketa} are satisfied. Let \(\{x^k\}\) be the sequence generated by Algorithm~\ref{alg:pmm}. 
Then for any \(\bar{x}\in\omega(x^0)\), \(\{x^k\}\) converges to \(\bar{x}\) with order \(q\). 
\end{theorem}
\begin{proof}
Recall that \(\lim_{k\to\infty}\|\mathcal{G}(x^k)\| = 0\) under Assumptions~\ref{assume:diffs} \textbf{S2}. Combine with Assumption~\ref{assume:errorboundold}, and~\eqref{eq:ndkold}, we know there exists \(\hat{k}\in\mathbb{N}\), such that for all \(k \!\geq\! \hat{k}\), \(\|\mathcal{G}(x^k)\| \!\leq\! 1\), \(\|x^{k+1} - x^k\| \!\leq\! c_6{\rm dist}(x^k, \mathcal{S}^*)\) for some \(c_6 > 0\) if \(x^k \in\mathbb{B}(\bar{x}, \epsilon_1)\) with \(\epsilon_1 = \min\{\epsilon, \epsilon_0/2\}\). 

We first show that for all \(k \geq \hat{k}\), if \(x^k \in \mathbb{B}(\bar{x}, \epsilon_1)\), then 
\begin{equation}\label{eq:disxk1old}
{\rm dist}(x^{k+1}, S^*) = \mathcal{O}({\rm dist}^q(x^{k}, S^*)).
\end{equation}
Under Assumptions~\ref{assume:diffs} \textbf{S2} and~\ref{assume:errorboundold}, we have 
\begin{align*}
{\rm dist}(x^{k+1}, \mathcal{S}^*) \leq& \kappa \|\mathcal{G}(x^{k+1})\|^q \leq  \kappa c^{-q/2}\|\mathcal{G}(x^{k+1})_{[S_{k+1}]}\|^q.
\end{align*}
Let \(r^k(x) = x - {\rm prox}_g(x - \nabla f(x^k) - (Q_k + \eta_k I)(x - x^k))\), \(\forall k \in\mathbb{N}\). 
Then it follows from Assumption~\ref{assume:gold}(ii) that \(r^k_{S_k}(y) = r^k(x)_{S_k}\) for any \(y = x_{S_k}\) and \(y^k = x^k_{S_k}\). Notice that 
\begin{align*}
\|\mathcal{G}(x^{k+1})_{[S_{k+1}]}\| =& \|\mathcal{G}(x^{k+1})_{[S_{k+1}]}\| - \|r^k(x^{k+1})_{[S_k]}\| + \|r_{S_k}^k(\hat{y}^k)\|\\
\leq & \|\mathcal{G}(x^{k+1})_{[S_{k+1}]} - r^k(x^{k+1})_{[S_k]}\| + \frac{\mu}{2}\|x^{k+1} - x^k\|\\
\leq& \|\mathcal{G}(x^{k+1})_{[S_{k+1}]} - \mathcal{G}(x^{k})\| + \|\mathcal{G}(x^{k}) - \mathcal{G}(x^{k})_{[S_{k}]}\| \\
&+ \|\mathcal{G}(x^{k})_{[S_{k}]}- r^k(x^{k+1})_{[S_k]}\| + \frac{\mu}{2}\|x^{k+1} - x^k\|. 
\end{align*}
For all \(k \geq \hat{k}\), if \(x^k\in\mathbb{B}(\bar{x}, \epsilon_1)\), then we have 
\begin{align*}
 &   \|\mathcal{G}(x^{k+1})_{[S_{k+1}]} - \mathcal{G}(x^{k})\| \\
    =&\|\mathcal{G}(x^{k+1})_{[S_{k+1}]} - \mathcal{G}(x^k)_{[S_{k+1}]} + \mathcal{G}(x^k)_{[S_{k+1}]} - \mathcal{G}(x^{k})\|\\
    \leq& \|\mathcal{G}_{S_{k+1}}(x^{k+1}_{S_{k+1}}) - \mathcal{G}_{S_{k+1}}(x^k_{S_{k+1}})\| + \|\mathcal{G}(x^{k})_{[S_{k+1}]} - \mathcal{G}(x^{k})\|\\
    \leq& \|x^{k+1}_{S_{k+1}} - x^k_{S_{k+1}}\| + \|x^{k+1}_{S_{k+1}} - \nabla f(x^{k+1})_{S_{k+1}} - x^k_{S_{k+1}} + \nabla f(x^{k})_{S_{k+1}}\| + \|\mathcal{G}(x^{k})\|\\
    \leq& 2\|x^{k+1}_{S_{k+1}} - x^k_{S_{k+1}}\| + \|\nabla f(x^{k+1})_{S_{k+1}} - \nabla f(x^{k})_{S_{k+1}}\| + \|\mathcal{G}(x^{k})\|\\
    \leq& (2 + L_g)\|x^{k+1} - x^k\| + \|\mathcal{G}(x^{k})\|,
    \end{align*}
    where the second inequality follows from the definition of \(\mathcal{G}_{S_{k+1}}(\cdot)\), the nonexpansivity of \({\rm prox}_{g_{k+1}}\), and the fact \(\|\mathcal{G}(x^{k})_{[S_{k+1}]} - \mathcal{G}(x^{k})\|  \leq \|\mathcal{G}(x^{k})\|\). In addition, 
\begin{align*}
\|\mathcal{G}(x^{k})_{[S_{k}]}-r^k(x^{k+1})_{[S_k]}\| =&\|\mathcal{G}_{S_{k}}(x_{S_k}^{k})-r^k_{S_k}(x_{S_k}^{k+1})\| \leq  (2 + L_g + \zeta + \bar{\eta})\|x^{k}_{S_{k}} - x^{k+1}_{S_k}\|\\
\leq& (2 + L_g + \zeta + \bar{\eta})\|x^{k+1} - x^k\|,
\end{align*}
where the first inequality follows from the definition of \(\mathcal{G}_{S_{k}}(\cdot)\) and \(r^k_{S_{k}}(\cdot)\) and the nonexpansivity of \({\rm prox}_{\tilde{g}_k}\). 
Hence, we have 
\[
\|\mathcal{G}(x^{k+1})_{[S_{k+1}]}\| \leq (4 + 2L_g + \zeta + \bar{\eta} + \frac{\mu}{2})\|x^{k+1} - x^k\| + 2\|\mathcal{G}(x^{k})\|,
\]
which yields 
\begin{align*}
\|\mathcal{G}(x^{k+1})_{[S_{k+1}]}\|^2 \leq& 2((4 + 2L_g + \zeta + \bar{\eta} + \frac{\mu}{2})^2\|x^{k+1} - x^k\|^2 + 4\|\mathcal{G}(x^{k})\|^2)\\
\leq& 2(4 + 2L_g + \zeta + \bar{\eta} + \frac{\mu}{2})^2\|x^{k+1} - x^k\|^2 + 8\frac{c_1^2}{c}\|x^{k+1} - x^k\|^2,
\end{align*}
where the last inequality follows from~\eqref{eq:xd}. 
Therefore, 
\begin{align*}
{\rm dist}(x^{k+1}, \mathcal{S}^*)\leq& \kappa c^{-q/2}\|\mathcal{G}(x^{k+1})_{[S_{k+1}]}\|^q\\
\leq& \kappa c^{-q/2}2^{q/2}((4 + 2L_g + \zeta + \bar{\eta} + \frac{\mu}{2})^2 + 4\frac{c_1^2}{c})^{q/2}\|x^{k+1} - x^k\|^q\\
\leq& \kappa c^{-q/2}2^{q/2}((4 + 2L_g  + \zeta + \bar{\eta} + \frac{\mu}{2})^2 + 4\frac{c_1^2}{c})^{q/2}c_6^q{\rm dist}^q(x^k, \mathcal{S}^*), 
\end{align*}
which yields~\eqref{eq:disxk1old}. 

Recall that \(\lim_{k\to\infty}{\rm dist}(x^k, \mathcal{S}^*) = 0\) under Assumption~\ref{assume:errorboundold} and Theorem~\ref{th:limitspwithoutls} (f). 
By~\eqref{eq:disxk1old}, for any \(c_7\in(0, 1)\), there exist \(\epsilon_2\in(0, \epsilon_1)\) and  \(\tilde{k} \geq \hat{k}\), such that for all \(k \geq \tilde{k}\), if \(x^k \in \mathbb{B}(\bar{x}, \epsilon_2)\), then we have
\[
{\rm dist}(x^{k+1}, \mathcal{S}^*) \leq c_7{\rm dist}(x^k, \mathcal{S}^*). 
\]
Define \(\bar{\epsilon} = \min\{\frac{\epsilon_2}{2}, \frac{(1 - c_7)\epsilon_2}{2c_6}\}\). Next, we show that if \(x^{k_0} \in \mathbb{B}(\bar{x}, \bar{\epsilon})\) for some \(k_0 \geq \tilde{k}\), then \(x^{k + 1} \in \mathbb{B}(\bar{x}, \epsilon_2)\) for all \(k \geq k_0\) by induction. 

Notice that \(\bar{x} \in \omega(x^0)\), there exists \(k_0 \geq \tilde{k}\), such that \(x^{k_0} \in \mathbb{B}(\bar{x}, \bar{\epsilon})\). Therefore, 
\begin{align*}
\|x^{k_0+1} - \bar{x}\| \leq \|x^{k_0} - \bar{x}\| + \|x^{k_0} - x^{k_0 + 1}\| \leq  \|x^{k_0} - \bar{x}\| + c_6{\rm dist}(x^{k_0}, S^*) \leq (1 + c_6)\bar{\epsilon} \leq \epsilon_2,
\end{align*}
which implies \(x^{k_0+1} \in \mathbb{B}(\bar{x}, \epsilon_2)\). For any \(k > k_0\), suppose that for all \(k_0 \leq l \leq k-1\), we have \(x^{k+1} \in \mathbb{B}(\bar{x}, \epsilon_2)\). Then we have 
\begin{align*}
\|x^{k+1} \!-\! x^{k_0}\| \!\leq\! \sum_{l = k_0}^k\!\|x^{l+1} \!-\! x^l\| \!\leq\! c_6\!\sum_{l = k_0}^k\!{\rm dist}(x^l, S^*) \!\leq\! c_6\!\sum_{l = k_0}^k\!c_7^{l - k_0}{\rm dist}(x^{k_0}, \mathcal{S}^*) \!\leq\! \frac{c_6}{1 \!-\! c_7}\|x^{k_0} \!-\! \bar{x}\|. 
\end{align*}
Therefore, \(\|x^{k + 1} - \bar{x}\| \leq \|x^{k+1} - x^{k_0}\| + \|x^{k_0} - \bar{x}\| \leq (1 + \frac{c_6}{1 - c_7})\|x^{k_0} - \bar{x}\| \leq (1 + \frac{c_6}{1 - c_7})\bar{\epsilon} \leq \epsilon_2\). Hence, \(x^{k + 1} \in \mathbb{B}(\bar{x}, \epsilon_2)\). 

Notice that for any \(\epsilon > 0\), there exists \(\bar{\bar{k}} \geq k_0\), such that 
\[
{\rm dist}(x^k, \mathcal{S}^*) < \tilde{\epsilon}, \quad \forall k > \bar{\bar{k}},
\]
where \(\tilde{\epsilon} = \frac{1 - c_7}{c_6}\epsilon\). For any \(k_1, k_2 > \bar{\bar{k}}\), 
without loss of generality we assume \(k_1 > k_2\), the following inequality holds:
\begin{align*}
\|x^{k_1} - x^{k_2}\|\leq& \!\sum_{j=k_2}^{k_1 - 1}\|x^{j+1} - x^j\| \leq c_6\!\sum_{j=k_2}^{k_1 - 1}{\rm dist}(x^j, \mathcal{S}^*) \leq c_6\!\sum_{j=k_2}^{k_1 - 1}c_7^{j - k_2}{\rm dist}(x^{k_2}, \mathcal{S}^*)\\
\leq&  \frac{c_6}{1 - c_7}{\rm dist}(x^{k_2}, \mathcal{S}^*) <  \frac{c_6}{1 - c_7}\tilde{\epsilon} = \epsilon.
\end{align*}
Hence, \(\{x^k\}_{k\in\mathbb{N}}\) is a Cauchy sequence. Recall that the cluster point set \(\omega(x^0)\) of \(\{x^k\}_{k\in\mathbb{N}}\) is closed. We have \(\{x^k\}_{k\in\mathbb{N}}\) converges to some \(\bar{x}\in\omega(x^0)\). By setting \(k_2 = k + 1\) and passing the limit \(k_1\to\infty\), we have for any \(k > \bar{\bar{k}}\), 
\[
\|x^{k\!+\!1} - \bar{x}\| \leq \frac{c_6}{1 - c_7}{\rm dist}(x^{k+1}, \mathcal{S}^*)\leq \frac{c_6c_8}{1 - c_7}{\rm dist}^q(x^k, \mathcal{S}^*)\leq \frac{c_6c_8}{1 - c_7}\|x^k - \bar{x}\|^q,
\]
where \(c_8 = \kappa c^{-q/2}2^{q/2}((4 + 2L_g + \zeta + \bar{\eta} + \frac{\mu}{2})^2 + 4\frac{c_1^2}{c})^{q/2}c_6^q\). Therefore, \(\{x^k\}_{k\in\mathbb{N}}\) converges to \(\bar{x}\) with the Q-superlinear rate of order \(q\). 
\end{proof}

A briefly comparison between Theorem~\ref{th:lcrxkold} and Assumption~\ref{assume:gold} with existing local convergence results for IPNM proposed in~\cite{Z24}. The main novelty of our local convergence analysis lies in the extension to the stochastic block‑coordinate setting, rather than in achieving a sharper convergence order.

% ******** Local convergence
%%%%%%%%%%%%%%%%%%%%%%%%%%%%%%%%%%%%%%%%%%%%%%%%
\subsection{Local convergence for convex problem}

Next, we establish the superlinear/quadratic local convergence rate of sequences \(\{\|\mathcal{G}(x^k)\|\}\) and \(\{x^k\}\) generated by Algorithm~\ref{alg:pmm} for convex composite problems under particular parameter settings. We assume that the order-\(q\) H\"olderian error bound condition holds locally to \(\mathcal{S}^*\). 

\begin{assumption}\label{assume:errorbound}
For any \(\bar{x}\in\omega(x^0)\), the order-\(q\) H\"olderian error bound condition at \(\bar{x}\) on \(\mathcal{S}^*\) holds, that is, there exist \(q\in (0, 1]\), \(\epsilon \in (0, 1)\), and \(\kappa > 0\) such that 
\[
{\rm dist}(x, \mathcal{S}^*) \leq \kappa \|\mathcal{G}(x)\|^q, \quad \forall x\in\mathbb{B}(\bar{x}, \epsilon). 
\]
\end{assumption}

We also specify the parameters in Algorithm~\ref{alg:pmm}. 
%\begin{tcolorbox}[mystyle]
We replace~\eqref{eq:errvek} and ~\eqref{eq:dishfqk} by
\begin{equation}\label{eq:newqk}
\begin{aligned}
&\|\varsigma_k\| \leq \frac{\mu_0}{2}\min\{1, \|\cG(x^k)\|^{\theta}\}\|\hat{y}^k - y^k\|,\\
&\|\nabla^2f(x^k) - Q_k\|\leq \zeta\min\{1, \|\cG(x^k)\|^{\theta}\}, 
\end{aligned}\quad \quad \forall x^k\in \mathbb{B}(\bar{x}, \epsilon),~\forall \bar{x} \in \omega(x^0), 
\end{equation}
for some \(\theta \in[0, 1]\)\footnote{Inequality in~\eqref{eq:newqk} on \(Q_k\) is well-defined since \(f\) is assumed to be convex. The results presented in subsequence still hold if we assume that \(\nabla^2f(\bar{x}) \succ 0\) for any \(\bar{x}\in \mathcal{S}^*\) instead of assuming that f is convex. }. 
Moreover, we set \(\eta_k = \eta\min\{1, \|\mathcal{G}(x^k)\|^{\theta}\}\), \(\forall k\in\mathbb{N}\), \((\mu_0, \eta, \zeta)\) satisfies \(2\eta - \mu_0(1 + \eta + L_g + \zeta) > 0\) and \(\eta  - L_g - \mu_0 > 0\). 
%\end{tcolorbox}

Similarly to Lemmas~\ref{lem:42old} and~\ref{lemma:43old}, we first establish the error between \(\bar{y}^k\) and \(\hat{y}^k\), as well as the error between \(y^k\) and \(\bar{y}^k\).
 
\begin{lemma}\label{lem:42}
Suppose that \(f\) is convex and~\eqref{eq:newqk} is satisfied. Let \(\{\hat{y}^k\}\) be the sequence generated by Algorithm~\ref{alg:pmm}. Then, for any \(x^k \in\mathbb{B}(\bar{x}, {\epsilon})\), we have  
\begin{equation*}%\label{eq:errhbxk}
\|\hat{y}^k - \bar{y}^k\| \leq \frac{\mu_0(1 + \eta + L_g + \zeta)}{2\eta}\|x^{k+1} - x^k\|.
\end{equation*}
\end{lemma}
\begin{proof}
By the definition of \(\bar{y}^k\) and using the first-order optimality condition, we have 
\begin{equation}\label{eq:optbarx}
-\nabla f(x^k)_{S_k} - (Q_k)_{S_k}(\bar{y}^k - y^k) - \eta_k(\bar{y}^k - y^k) \in \partial \tilde{g}_k(\bar{y}). 
\end{equation}
Combining with~\eqref{eq:ek2}, using the monotonicity of \(\partial \tilde{g}_k\), we have 
\begin{align*}
0 \leq & \langle \hat{y}^k - r_{S_k}^k(\hat{y}^k) - \bar{y}^k, r_{S_k}^k(\hat{y}^k) + ((Q_k)_{S_k} + \eta_k I_{\vert S_k\vert})(\bar{y}^k - \hat{y}^k)\rangle\\
\leq &-\!\langle ((Q_k)_{S_k} \!+\! (1 \!+\! \eta_k) I_{\vert S_k\vert})(\bar{y}^k \!-\! \hat{y}^k), r_{S_k}^k(\hat{y}^k)\rangle \!+\! \langle \hat{y}^k \!-\! \bar{y}^k,  ((Q_k)_{S_k} \!+\! \eta_k I_{\vert S_k\vert})(\bar{y}^k \!-\! \hat{y}^k)\rangle. 
\end{align*}
From~\eqref{eq:newqk} and Cauchy inequality, we have 
\begin{align*}
\eta_k \|\hat{y}^k - \bar{y}^k\| \leq& (1 + \eta_k + L_g + \zeta\min\{1, \|\cG(x^k)\|^{\theta}\})\|r_{S_k}^k(\hat{y}^k)\| \\
\leq& \frac{\mu_0(1 + \eta + L_g + \zeta)}{2}\min\{1, \|\mathcal{G}(x^k)\|^{\theta}\}\|d_k\|,
\end{align*}
where the last inequality follows from~\eqref{eq:evare} with \(\hat{\epsilon}_k = -\varsigma_k\) and~\eqref{eq:newqk}. The statement holds. 
\end{proof}

\begin{lemma}\label{lemma:43}
Consider any \(\bar{x}\in\omega(x^0)\). Suppose that \(f\) is convex and  Assumption~\ref{assume:gold} and~\eqref{eq:newqk} are satisfied.
Let \(\{x^k\}\) and \(\{\hat{y}^k\}\) be the sequence generated by Algorithm~\ref{alg:pmm}.  Then, for all \(x^k \in \mathbb{B}(\bar{x}, {\epsilon}_0/2)\), we have
\[
\|y^k - \bar{y}^k\| \leq  \frac{L_C}{2\eta\min\{1, \|\mathcal{G}(x^k)\|^{\theta}\}}{\rm dist}^2(x^k, \mathcal{S}^*) + (2 + \frac{\zeta}{\eta}){\rm dist}(x^k, \mathcal{S}^*).
\]
\end{lemma}
\begin{proof}
For any \(x^k\in\mathbb{B}(\bar{x}, \epsilon_0/2)\), let \(\Pi_{\mathcal{S}^*}(x^k)\)
be the projection set of \(x^k\) onto \(\mathcal{S^*}\). Then \(\Pi_{\mathcal{S}^*}(x^k) \neq \emptyset\) since \(\mathcal{S}^*\) is closed. Pick \(x^{k, *}\in \Pi_{\mathcal{S}^*}(x^k)\). Notice that \(\bar{x} \in\omega(x^0)\subseteq \mathcal{S}^*\), we have 
\[
\|x^{k, *} -\bar{x}\| \leq \|x^{k, *} - x^k\| + \|x^k - \bar{x}\| \leq 2\|x^k - \bar{x}\| \leq \epsilon_0,
\]
which implies that \(x^{k, *}\in\mathbb{B}(\bar{x}, \epsilon_0)\). Hence, \((1- t)x^k + tx^{k, *}\in\mathbb{B}(\bar{x}, \epsilon_0)\cap {\rm dom}g\) for all \(t\in[0,1]\). Notice that \(x^{k, *} \in\mathcal{S}^*\), we have \(-\nabla f(x^{k, *}) \in \partial g(x^{k, *})\). Moreover, \(-\nabla f(x^{k, *})_{S_k} \in \partial \tilde{g}_k(x^{k, *}_{S_k})\) under Assumption~\ref{assume:gold} (ii). By~\eqref{eq:optbarx} and the monotonicity of \(\partial \tilde{g}_k\), it follows that 
\begin{align*}
0 \leq& \langle x^{k, *}_{S_k} - \bar{y}^k, -\nabla f(x^{k, *})_{S_k} + \nabla f(x^k)_{S_k} + ((Q_k)_{S_k} + \eta_kI_{\vert S_k\vert})(\bar{y}^k - y^k)\rangle\\
=& \langle x^{k, *}_{S_k} - \bar{y}^k, -\nabla f(x^{k, *})_{S_k} + \nabla f(x^k)_{S_k} + ((Q_k)_{S_k} + \eta_kI_{\vert S_k\vert})(x_{S_k}^{k, *} - y^k)\rangle\\
&+\langle x^{k, *}_{S_k} - \bar{y}^k, ((Q_k)_{S_k} + \eta_kI_{\vert S_k\vert})(\bar{y}^k - x_{S_k}^{k, *})\rangle.
\end{align*}
By \eqref{eq:newqk} and Cauchy inequality, we have 
\begin{align*}
\|x^{k, *}_{S_k} - \bar{y}^k\| \leq& \frac{1}{\eta_k}\| \nabla f(x^k)_{S_k} - \nabla f(x^{k, *})_{S_k} + ((Q_k)_{S_k} + \eta_kI_{\vert S_k\vert})(x_{S_k}^{k, *} - y^k)\|\\
=& \frac{1}{\eta_k}\|E_k^\top\int_0^1[Q_k + \eta_kI_n - \nabla^2f(x^k + t(x^{k, *} - x^k))](x^{k, *} - x^k)_{[S_k]}dt\|\\
\leq& \frac{1}{\eta_k}\|\int_0^1[Q_k + \eta_kI_n - \nabla^2f(x^k + t(x^{k, *} - x^k))](x^{k, *} - x^k)_{[S_k]}dt\|\\
\leq& \frac{1}{\eta_k}\|\int_0^1[\nabla^2f(x^k) - \nabla^2f(x^k + t(x^{k, *} - x^k))](x^{k, *} - x^k)_{[S_k]}dt\| \\
&+ \frac{1}{\eta_k}\|\int_0^1[Q_k - \nabla^2f(x^k) + \eta_kI_n](x^{k, *} - x^k)_{[S_k]}dt\|\\
=&\frac{L_C}{2\eta_k}\|x^{k, *} - x^k\|^2 + \frac{\zeta + \eta}{\eta}\|x^{k, *} - x^k\|,
\end{align*}
where \(E_k\in\mathbb{R}^{n\times \vert S_k\vert}\) is the column submatrix of \(I_n\) that corresponds to \(S_k\) and the last inequality follows from Assumption~\ref{assume:gold} (i), \(\|(x^{k, *} - x^k)_{[S_k]}\| \leq \|x^{k, *} - x^k\|\),  and~\eqref{eq:newqk}. Therefore, 
\begin{align*}
\|y^k - \bar{y}^k\| \leq& \|y^k - x^{k, *}_{S_k} \| + \|x^{k, *}_{S_k} - \bar{y}^k\|\\
\leq & \|x^k - x^{k, *}\| + \frac{L_C}{2\eta_k}\|x^{k, *} - x^k\|^2 + \frac{\zeta + \eta}{\eta}\|x^{k, *} - x^k\|\\
\leq& \frac{L_C}{2\eta_k}\|x^k - x^{k, *}\|^2 + (2 + \frac{\zeta}{\eta})\|x^k - x^{k, *}\|. 
\end{align*}
The statement holds. 
\end{proof}

By invoking Lemmas~\ref{lem:42} and \ref{lemma:43}, for all \(x^k\in\mathbb{B}(\bar{x}, \epsilon_0/2)\), we have
\begin{align*}
\|x^{k+1} \!- x^k\| \!=& \|\hat{y}^k - y^k\| \leq \|\hat{y}^k - \bar{y}^k\| + \|\bar{y}^k - y^k\|\\
\leq& \frac{\mu_0(1 \!+\! \eta \!+\! L_g \!+\! \zeta)}{2\eta}\|x^{k+1} \!-\! x^k\| \!+\! \frac{L_C}{2\eta_k}{\rm dist}^2(x^k, \mathcal{S}^*) \!+\! (2 \!+\! \frac{\zeta}{\eta}){\rm dist}(x^k, \mathcal{S}^*),
\end{align*}
which yields
\begin{align}\label{eq:ndk}
\|x^{k+1} - x^k\| \leq& \frac{L_C}{(2\eta - \mu_0(1 + \eta + L_g + \zeta))\min\{1, \|\mathcal{G}(x^k)\|^{\theta}\}}{\rm dist}^2(x^k, \mathcal{S}^*)  \nonumber \\
&+ \frac{2(2\eta + \zeta)}{2\eta - \mu_0(1 + \eta + L_g + \zeta)}{\rm dist}(x^k, \mathcal{S}^*).
\end{align}

\begin{theorem}\label{th:lcrxk}
Suppose that \(f\) is convex,  Assumptions~\ref{assume:diffs} \textbf{S2},~\ref{assume:gold}, and~\ref{assume:errorbound}, and~\eqref{eq:newqk} are satisfied with \(\theta \in [0, q]\) and \(q(1 + \theta) > 1\). Let \(\{x^k\}\) be the sequence generated by Algorithm~\ref{alg:pmm}. 
Then for any \(\bar{x}\in\omega(x^0)\), the sequences \(\{\|\mathcal{G}(x^k)\|\}\) and \(\{x^k\}\) converges to \(0\) and \(\bar{x}\) with the same convergence rate \(q(1 + \theta)\). 
\end{theorem}
\begin{proof}
From Theorem~\ref{th:limitspwithoutls} (f) and Assumption~\ref{assume:errorbound}, we have \(\lim_{k\to\infty}{\rm dist}(x^k, \mathcal{S}^*) = \lim_{k\to\infty}\|\mathcal{G}(x^k)\| = 0\). Hence, there exists \(\hat{k}\in\mathbb{N}\), such that for all \(k \geq \hat{k}\), \({\rm dist}(x^k, \mathcal{S}^*) < 1\) and \(\|\mathcal{G}(x^k)\| \leq 1\). 
From Assumption~\ref{assume:errorbound} and~\eqref{eq:ndk}, for any \(x^k\in \mathbb{B}(\bar{x}, \epsilon_1)\) with \(k \geq \hat{k}\) and \(\epsilon_1 = {\epsilon}_0/2\), we have 
\begin{align}\label{eq:dxkk}
\|x^{k+1} - x^k\| \leq& \frac{L_C}{2\eta - \mu_0(1 + \eta + L_g + \zeta)}\|\mathcal{G}(x^k)\|^{q - \theta}{\rm dist}(x^k, \mathcal{S}^*)  \nonumber  \\
&+ \frac{2(2\eta + \zeta)}{2\eta - \mu_0(1 + \eta + L_g + \zeta)}{\rm dist}(x^k, \mathcal{S}^*),
\end{align}
which implies that 
\begin{subequations}
\begin{align}
\|x^{k+1} - x^k\| \leq& c_6{\rm dist}(x^k, \mathcal{S}^*)  \quad {\rm with}\quad c_6 = \frac{L_C + 2(2\eta + \zeta)}{2\eta - \mu_0(1 + \eta + L_g + \zeta)}; \label{eq:dxkk1}\\
\|x^{k+1} - x^k\| \leq& \tilde{c}_6\|\mathcal{G}(x^k)\|^q \quad\quad {\rm with}\quad \tilde{c}_6 = \kappa c_6. \label{eq:dxkk2}
\end{align}
\end{subequations}

We first show that for all \(k \geq \hat{k}\), if \(x^k \in \mathbb{B}(\bar{x}, \epsilon_1)\), then 
\begin{equation}\label{eq:disxk1}
{\rm dist}(x^{k+1}, S^*) = \mathcal{O}({\rm dist}^{q(1 + \theta)}(x^{k}, S^*)).
\end{equation}
Denote \(\tilde{x}^k := x^{k+1} - r^k({x^{k+1}})\) and \(\tilde{y}^k := \hat{y}^k - r^k_{S_k}(\hat{y}^k)\), where \(r^k(x) = x - {\rm prox}_g(x - \nabla f(x^k) - (Q_k + \eta_k I)(x - x^k))\), \(\forall k \in\mathbb{N}\). From~\eqref{eq:ek2}, we have 
\[
r_{S_k}^k(\hat{y}^k) - \nabla f(x^k)_{S_k} - ((Q_k)_{S_k} + \eta_kI_{\vert S_k\vert})(\tilde{y}^k - y^k + r_{S_k}^k(\hat{y}^k)) \in \partial \tilde{g}_k(\tilde{y}^k).
\]
Let \(\mathcal{S}\mathcal{G}(\tilde{y}^k) :=  \nabla f(\tilde{x}^k)_{S_k} + r_{S_k}^k(\hat{y}^k) - \nabla f(x^k)_{S_k} - ((Q_k)_{S_k} + \eta_kI_{\vert S_k\vert})(\tilde{y}^k - y^k + r_{S_k}^k(\hat{y}^k))\). Then we have  
\[
\mathcal{S}\mathcal{G}(\tilde{y}^k) \in  \nabla f(\tilde{x}^k)_{S_k} + \partial\tilde{g}_k(\tilde{y}^k). 
\]
Moreover,  
\begin{align*}
\|\mathcal{S}\mathcal{G}(\tilde{y}^k)\| \leq &\| \nabla f(\tilde{x}^k) - f(x^k) - \nabla^2f(x^k)(\tilde{x}^k - x^k)\| + \|\nabla^2f(x^k) - Q_k - \eta_kI_n\|\|\tilde{y}^k - y^k\| \\
&+ \|(Q_k)_{S_k} + (1 + \eta_k)I_{\vert S_k\vert}\|\|r_{S_k}^k(\hat{y}^k)\|\\
\leq& \frac{L_C}{2}\|\tilde{y}^k \!-\! y^k\|^2 \!+\! (\zeta \!+\! \eta)\|\mathcal{G}(x^k)\|^{\theta}\|\tilde{y}^k \!-\! y^k\| \!+\! \frac{\mu_0}{2}(1 \!+\! L_g \!+\! \zeta \!+\! \eta)\|\mathcal{G}(x^k)\|^{\theta}\|x^{k+1} \!-\! x^k\|. 
\end{align*}
Notice that 
\begin{align*}
\|\tilde{y}^k - y^k\| \leq& \|\tilde{y}^k - \hat{y}^k\| + \|\hat{y}^k - y^k\|\leq \|r_{S_k}^k(\hat{y}^k)\| + \|x^{k+1} - x^k\| \leq (1+\frac{\mu}{2})\|x^{k+1} - x^k\|.
\end{align*}
We have 
\begin{align*}
\|\mathcal{S}\mathcal{G}(\tilde{y}^k)\| \leq&  \frac{L_C}{2}(1+\frac{\mu}{2})^2\|x^{k+1} - x^k\|^2 
+ c_g\|\mathcal{G}(x^k)\|^{\theta}\|x^{k+1} - x^k\|\\
\overset{\eqref{eq:xd}}{\leq}&  \frac{L_C}{2}(1+\frac{\mu}{2})^2\|x^{k+1} - x^k\|^2 + \frac{c_1^{\theta}}{c^{\theta/2}}c_g\|x^{k+1} - x^k\|^{1 + \theta},
\end{align*}
where \(c_g = (\zeta + \eta)(1+\frac{\mu}{2}) + \frac{\mu_0}{2}(1 + L_g + \zeta + \eta)\).  
Recalling that \(\mathcal{S}\mathcal{G}(\tilde{y}^k) \in  \nabla f(\tilde{x}^k)_{S_k} + \partial\tilde{g}_k(\tilde{y}^k)\), we have 
\begin{align*}
\|\mathcal{G}(\tilde{x}^k)_{S_k}\| \leq \|\mathcal{S}\mathcal{G}(\tilde{y}^k)\| \leq& \frac{L_C}{2}(1+\frac{\mu}{2})^2\|x^{k+1} - x^k\|^2 + \frac{c_1^{\theta}}{c^{\theta/2}}c_g\|x^{k+1} - x^k\|^{1 + \theta}\\
\leq & (\frac{L_C}{2}(1+\frac{\mu}{2})^2 + \frac{c_1^{\theta}}{c^{\theta/2}}c_g)\|x^{k+1} - x^k\|^{1 + \theta}, 
\end{align*}
where the first inequality follows from the fact that \(\|\mathcal{G}(x)\| \leq {\rm dist}(0, \nabla f(x) + \partial g(x))\)~\cite[Th. 3.5]{DL18} and block-coordinate separable of \(g\). 
Notice that  
\begin{align*}
 \|\mathcal{G}(x^{k+1})_{S_{k}} - \mathcal{G}(\tilde{x}^k)_{S_{k}}\|  \leq& \|x^{k+1}_{S_{k}} - \tilde{x}^k_{S_{k}}\| + \|x^{k+1}_{S_{k}} - \nabla f(x^{k+1})_{S_{k}} - \tilde{x}^k_{S_{k}} + \nabla f(\tilde{x}^{k})_{S_{k}}\| \\
    \leq& 2\|x^{k+1}_{S_{k}} - \tilde{x}^k_{S_{k}}\| + \|\nabla f(x^{k+1})_{S_{k}} - \nabla f(\tilde{x}^{k})_{S_{k}}\| \\
    \leq& (2 + L_g)\|r^k_{S_{k}}(x^{k+1})\| \overset{\eqref{eq:evare}, \eqref{eq:newqk}}{\leq} \frac{\mu}{2}(2 + L_g)\|\mathcal{G}(x^k)\|^{\theta}\|x^{k+1} - x^k\|\\
    \leq& \frac{\mu c_1^{\theta}}{2c^{\theta/2}}(2 + L_g)\|x^{k+1} - x^k\|^{1+\theta},
    \end{align*}
where the first inequality follows from the definition of \(\mathcal{G}_{S_{k}}(\cdot)\) and the nonexpansivity of \({\rm prox}_{g_{k}}\). Hence, we have 
\begin{align*}
\|\mathcal{G}(x^{k+1})_{S_{k}}\| \leq& \|\mathcal{G}(x^{k+1})_{S_{k}} - \mathcal{G}(\tilde{x}^k)_{S_{k}}\| + \|\mathcal{G}(\tilde{x}^k)_{S_{k}}\|  \leq \tilde{c}\|x^{k+1} - x^k\|^{1+\theta},
\end{align*}
where \(\tilde{c} = \frac{L_C}{2}(1+\frac{\mu}{2})^2 + \frac{c_1^{\theta}}{c^{\theta/2}}((\zeta + \eta)(1+\frac{\mu}{2}) + \frac{\mu}{2}(1 + L_g + \zeta + \eta)) + \frac{\mu c^{\theta}}{2c^{\theta/2}}(2 + L_g)\). 
Therefore, from Assumptions~\ref{assume:diffs} \textbf{S2} and~\ref{assume:errorbound}, we have 
\begin{align*}
\|\mathcal{G}(x^{k+1})\| \leq  c^{-\frac{1}{2}}\|\mathcal{G}(x^{k+1})_{S_{k}}\| \leq  c^{-\frac{1}{2}}\tilde{c}\|x^{k+1} - x^k\|^{1+\theta} \overset{\eqref{eq:dxkk2}}{\leq}  c^{-\frac{1}{2}}\tilde{c}\tilde{c}_6^{1+\theta}\|\mathcal{G}(x^k)\|^{q(1+\theta)}
\end{align*}
and 
\begin{align*}
{\rm dist}(x^{k+1}, \mathcal{S}^*) \leq& \kappa \|\mathcal{G}(x^{k+1})\|^q \leq \kappa c^{-1/2}\tilde{c}\|x^{k+1} \!- x^k\|^{q(1+\theta)} \\
\overset{\eqref{eq:dxkk1}}{\leq}& \kappa c^{-1/2}\tilde{c}c_6^{q(1 + \theta)}{\rm dist}^{q(1 + \theta)}(x^k, \mathcal{S}^*).
\end{align*}
Hence, we obtain~\eqref{eq:disxk1} and \(\{\|\mathcal{G}(x^k)\|\}\) converges to \(0\) with the convergence rate of order \(q(1+\theta)\).

From \(\lim_{k\to\infty}{\rm dist}(x^k, \mathcal{S}^*) = 0\) and~\eqref{eq:disxk1}, for any \(c_7\in(0, 1)\), there exist \(\epsilon_2\in(0, \epsilon_1)\) and  \(\tilde{k} \geq \hat{k}\), such that for all \(k \geq \tilde{k}\), if \(x^k \in \mathbb{B}(\bar{x}, \epsilon_2)\), then we have
\[
{\rm dist}(x^{k+1}, \mathcal{S}^*) \leq c_7{\rm dist}(x^k, \mathcal{S}^*). 
\]
Define \(\bar{\epsilon} = \min\{\frac{\epsilon_2}{2}, \frac{(1 - c_7)\epsilon_2}{2c_6}\}\). Next, we show that if \(x^{k_0} \in \mathbb{B}(\bar{x}, \bar{\epsilon})\) for some \(k_0 \geq \tilde{k}\), then \(x^{k + 1} \in \mathbb{B}(\bar{x}, \epsilon_2)\) for all \(k \geq k_0\) by induction. 

Notice that \(\bar{x} \in \omega(x^0)\), there exists \(k_0 \geq \tilde{k}\), such that \(x^{k_0} \in \mathbb{B}(\bar{x}, \bar{\epsilon})\). Therefore, 
\begin{align*}
\|x^{k_0+1} - \bar{x}\| \leq& \|x^{k_0} - \bar{x}\| + \|x^{k_0} - x^{k_0 + 1}\| \overset{\eqref{eq:dxkk1}}{\leq}  \|x^{k_0} - \bar{x}\| + c_6{\rm dist}(x^{k_0}, S^*) \leq (1 + c_6)\bar{\epsilon} \leq \epsilon_2,
\end{align*}
which implies \(x^{k_0+1} \in \mathbb{B}(\bar{x}, \epsilon_2)\). For any \(k > k_0\), suppose that for all \(k_0 \leq l \leq k-1\), we have \(x^{k+1} \in \mathbb{B}(\bar{x}, \epsilon_2)\). Moreover, 
\[
\|x^{k+1} \!-\! x^{k_0}\| \!\leq\!\! \sum_{l = k_0}^k\|x^{l+1} \!-\! x^l\| \!\leq\! c_6\!\!\!\sum_{l = k_0}^k{\rm dist}(x^l, S^*) \!\leq\! c_6\!\!\!\sum_{l = k_0}^kc_7^{l - k_0}{\rm dist}(x^{k_0}, \mathcal{S}^*)\!\leq\! \frac{c_6}{1 \!-\! c_7}\|x^{k_0} \!-\! \bar{x}\|. 
\]
Therefore, \(\|x^{k + 1} - \bar{x}\| \leq \|x^{k+1} - x^{k_0}\| + \|x^{k_0} - \bar{x}\| \leq (1 + \frac{c_6}{1 - c_7})\|x^{k_0} - \bar{x}\| \leq (1 + \frac{c_6}{1 - c_7})\bar{\epsilon} \leq \epsilon_2\). Hence, \(x^{k + 1} \in \mathbb{B}(\bar{x}, \epsilon_2)\). 

Notice that for any \(\epsilon > 0\), there exists \(\bar{\bar{k}} \geq k_0\), such that 
\[
{\rm dist}(x^k, \mathcal{S}^*) < \tilde{\epsilon}, \quad \forall k > \bar{\bar{k}},
\]
where \(\tilde{\epsilon} = \frac{1 - c_7}{c_6}\epsilon\). For any \(k_1, k_2 > \bar{\bar{k}}\), 
without loss of generality we assume \(k_1 > k_2\), the following inequality holds:
\begin{align*}
\|x^{k_1} - x^{k_2}\|\leq& \!\sum_{j=k_2}^{k_1 - 1}\|x^{j+1} - x^j\| \leq c_6\!\sum_{j=k_2}^{k_1 - 1}{\rm dist}(x^j, \mathcal{S}^*) \leq c_6\!\sum_{j=k_2}^{k_1 - 1}c_7^{j - k_2}{\rm dist}(x^{k_2}, \mathcal{S}^*)\\
\leq&  \frac{c_6}{1 - c_7}{\rm dist}(x^{k_2}, \mathcal{S}^*) <  \frac{c_6}{1 - c_7}\tilde{\epsilon} = \epsilon.
\end{align*}
Hence, \(\{x^k\}_{k\in\mathbb{N}}\) is a Cauchy sequence. Recall that the cluster point set \(\omega(x^0)\) of \(\{x^k\}_{k\in\mathbb{N}}\) is closed. We have \(\{x^k\}_{k\in\mathbb{N}}\) converges to some \(\bar{x}\in\omega(x^0)\). By setting \(k_2 = k + 1\) and passing the limit \(k_1\to\infty\), we have for any \(k > \bar{\bar{k}}\), 
\[
\|x^{k\!+\!1} - \bar{x}\| \leq \frac{c_6}{1 - c_7}{\rm dist}(x^{k+1}, \mathcal{S}^*)\leq \frac{c_6c_8}{1 - c_7}{\rm dist}^{q(1 + \theta)}(x^k, \mathcal{S}^*)\leq \frac{c_6c_8}{1 - c_7}\|x^k - \bar{x}\|^{q(1 + \theta)},
\]
where \(c_8 = \kappa c^{-1/2}\tilde{c}c_6^{q(1+\theta)}\). Therefore, \(\{x^k\}_{k\in\mathbb{N}}\) converges to \(\bar{x}\) with the rate of order \(q(1 + \theta)\). 
\end{proof}
By Theorem~\ref{th:lcrxk}, under the linear error bound (i.e., \(q = 1\)), \(\{\|\mathcal{G}(x^k)\|\}\) and \(\{x^k\}\) converge superlinearly for \(\theta \in (0, 1)\) and quadratically for \(\theta = 1\). It is clear that the conclusion remains valid when \(S_k \equiv [n]\). Namely, for convex composite optimization problems, IPNM  achieves superlinear and even quadratic convergence under the H\"olderian error bound condition. This result is consistent with the literature; see, e.g.,~\cite[Th. 4.1]{MYZZ22}.
%%%%%%%%%%%%%%%%%%%%%%%%%%%%%%%%%%%%%%%%%%%
% ******* Numerical Experiments
%%%%%%%%%%%%%%%%%%%%%%%%%%%%%%%%%%%%%%%%%%%
\section{Numerical Experiments}\label{sec:numerical}

In this section, we evaluate the effectiveness and efficiency of Algorithm~\ref{alg:pnewton} on the \(\ell_1\)-regularized Student's \(t\)-regression, nonconvex binary classification with Geman-McClure loss function, and biweight loss with group regularization. All numerical experiments are implemented in MATLAB R2023b running on a computer with an Intel(R) Core(TM) i9-10885U CPU @ 2.40GHz \(\times\) 2.4 and 32GB of RAM.

%%%%%%%%%%%%%%%%%%%%%%%%%%%%%%%%%%%%%%%%%%%%%%%%% ******* L1-regularized Student's  t-regression
%%%%%%%%%%%%%%%%%%%%%%%%%%%%%%%%%%%%%%%
\subsection{\(\ell_1\)-regularized Student's \(t\)-regression}\label{subsec:st}

We first consider the following \(\ell_1\)-regularized Student's \(t\)-regression~\cite{AFHL12} problem:  \begin{equation}\label{eq:st}
 \min_x\sum_{i=1}^m\log(1 + (Ax-b)_i^2/\nu) + \lambda \|x\|_1,
 \end{equation}
 where \(\nu > 0\) and \(\lambda > 0\) is the regularized parameter. Problem~\eqref{eq:st} is a special case of Problem~\eqref{eq:ncp} with \(f(x) := \sum_{i=1}^m\log(1 + (Ax-b)_i^2/\nu)\) and \(g(x): = \lambda \|x\|_1\). 
 In the following test, we generate the reference signal \(x^{\rm true}\in\mathbb{R}^n\) of length \(n\) with \(k = [n/40]\) nonzero entries, where the \(k\) different indices \(i\in\{1, \cdots, n\}\) of nonzero entries are randomly chosen and the magnitude of each nonzero entry is determined via \(x^{\rm true}_i = \eta_1(i)10^{20\eta_2(i)}\), \(\eta_1(i) \in\{-1, +1\}\) is a symmetric random sign and \(\eta_2(i)\) is uniformly distributed in \([0, 1]\).
The matrix \(A\in\mathbb{R}^{m\times n}\) takes \(m\) random cosine measurements, i.e., \(Ax^{\rm true} = ({\rm dct}(x^{\rm true}))_{J}\), where \(J \subset \{1, \cdots, m\}\) with \(\vert J\vert = n\) is randomly chosen and \({\rm dct}\) denotes the discrete cosine transform.  The measurement \(b\) is obtained by adding Student's t-noise with degree of freedom \(5\) and rescaled by \(0.1\) to \(Ax^{\rm true}\). 
We set \(\lambda = 0.1\|\nabla f(0)\|_{\infty}\) and \(\nu = 0.25\) in Problem~\eqref{eq:st}. 
Denote \(\psi(z) = \sum_{i=1}^m\log(1 + z_i^2/\nu)\). We have \(f(x) = \psi(Ax - b)\) and \(\nabla^2f(x) = A^\top\nabla^2\psi(Ax - b) A\). It noted that \(\nabla^2\psi(z)= {\rm Diag}(p)\) is a diagonal matrix, where \(p_i = \frac{2(\nu - z_i^2)}{\nu + z_i^2}\) and \(\|p\| \leq \frac{2}{\nu}\sqrt{m}\). For each \(x^k\), let \(p^k = {\rm diag}(\nabla^2\psi(Ax^k - b))\) and define \(\tilde{p}^k_i = p^k_i\) if \(p^k_i > 0\); otherwise, \(\tilde{p}^k_i = 10^{-5}\). We then set \(Q_k = A^\top {\rm Diag}(\tilde{p}^k)A\), and further specify \(\eta_k = \frac{1}{2}\times 10^{-2}\max\{\|\mathcal{G}(x^k)_{S_k}\|, 10^{-2}\}\) if \(\min\{\tilde{p}^k\} + \eta_k \geq \mu\); otherwise, \(\eta_k = \frac{1}{2}\times 10^{-2}\max\{\|\mathcal{G}(x^k)_{S_k}\|, 10^{-2}\} + \mu\), where \(\mu = \max\{0.05\|\mathcal{G}(x^k)_{S_k}\|, 10^{-2}\}\). For each \(k\in\mathbb{N}\), it can be verified that \(\|\nabla^2f(x^k) - Q_k\| \leq \|A\|^2\|\tilde{p}^k\| \leq (\frac{2}{\nu} + 10^{-5})\sqrt{m}\|A\|^2\) and \((Q_k)_{S_k} + (\eta_k - \mu)I_{\vert S_k\vert} \succeq 0\). Consequently, the conditions of Theorem~\ref{th:limitsppn} are satisfied under this setting. 
For each \(k\in\mathbb{N}\), we obtain the approximate solution \(\hat{y}^k\) by using the semismooth Newton (SSN) method~\cite{QS93,LST18}. Details are similar to that used in~\cite{Z24} so we omit it. 

We consider the following three sampling strategies: i). cyclic sampling with continuous indices (named as {\bf SBCPNM\_cycr}). The sampling order is randomly determined for each cycle. ii). cyclic sampling with random indices (named as {\bf SBCPNM\_cycrd}). iii). Top-\(\mathbf{k}\) sampling (named as {\bf SBCPNM\_topk}); We name the algorithm with \(S_k = [n]\) as {\bf IPNM}. In the following tests, we set \((m, n) = (2n, 2^{11})\). We stop Algorithm~\ref{alg:pnewton} when \(\|\mathcal{G}(x^k)\| \leq 10^{-4}\) and set \(\tau = 10^{-5}\) and \(\theta = 0.6\), respectively.  
Figure~\ref{fig:t_students} shows the norm of the residual mapping at iterates generated by each method along with running time and iteration, respectively. It can be seen that stochastic methods work well and outperform {\bf IPNM} in terms of running time. The iterations required by {\bf SBCPNM\_cycr} and {\bf SBCPNM\_cycrd} are similar to each other. When \(\mathbf{k}\) in Top-\(\mathbf{k}\) sampling equal to \(s\), {\bf SBCPNM\_topk} requires less number of iterations and performs faster than {\bf SBCPNM\_cycr} and {\bf SBCPNM\_cycrd}. The second column of Figure~\ref{fig:t_students} also illustrates that SBCPNM can achieve better convergence rate in terms of \(\|G(x^k)\|\) than sublinear when implemented. 
The last column of Figure~\ref{fig:t_students} displays the distance between iterates generated by each method and \(\bar{x}\), where \(\bar{x}\) is the value returned by IPNM. Superlinear convergence rate of {\bf SBCPNM\_topk} can be observed. 

\begin{figure}[h!]
\centering
\includegraphics[width = 0.325\textwidth]{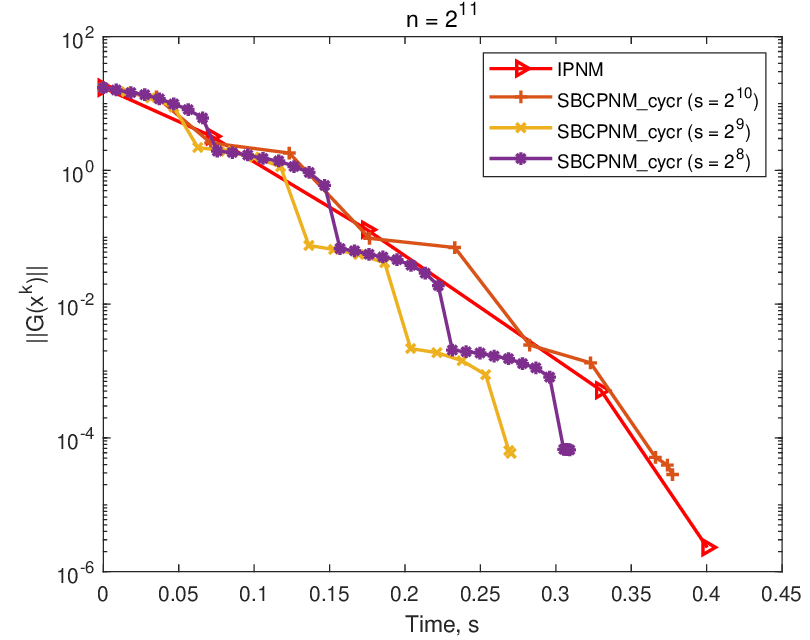}
\includegraphics[width = 0.325\textwidth]{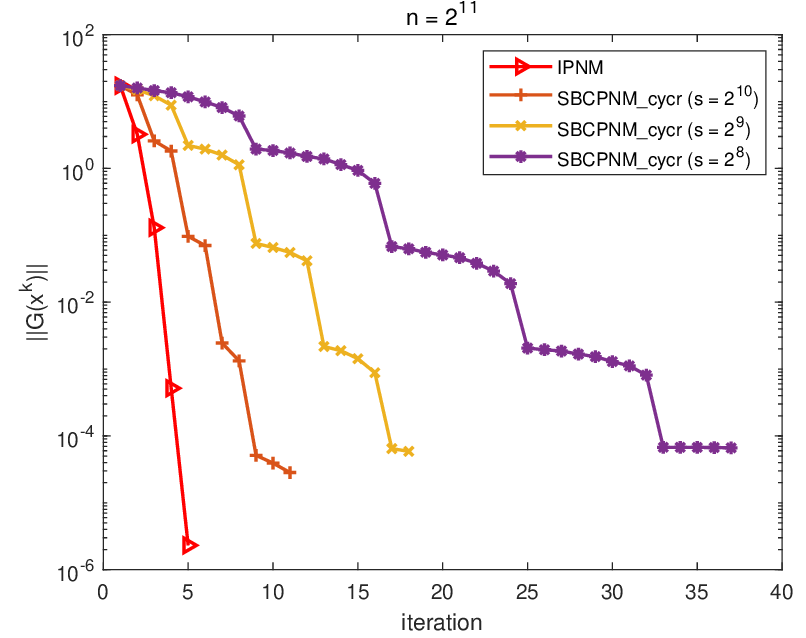}
\includegraphics[width = 0.325\textwidth]{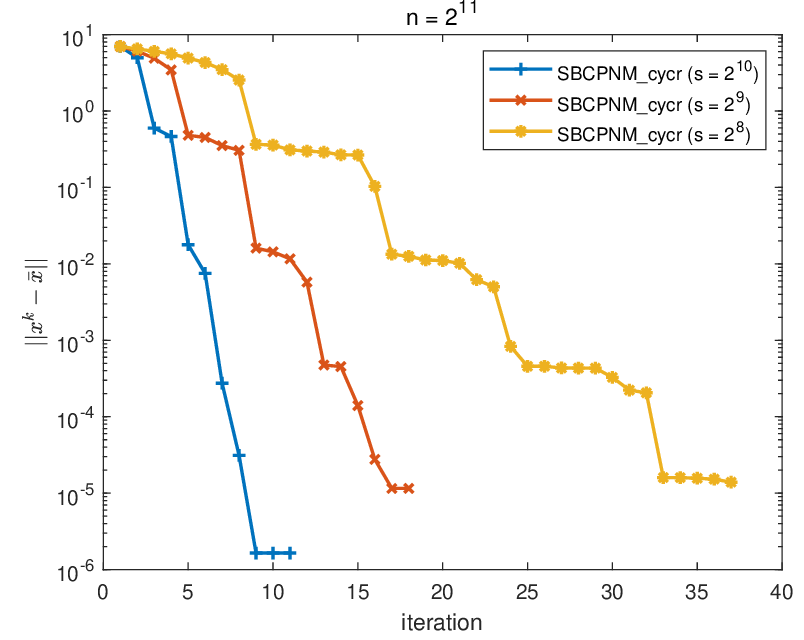}\\
\includegraphics[width = 0.325\textwidth]{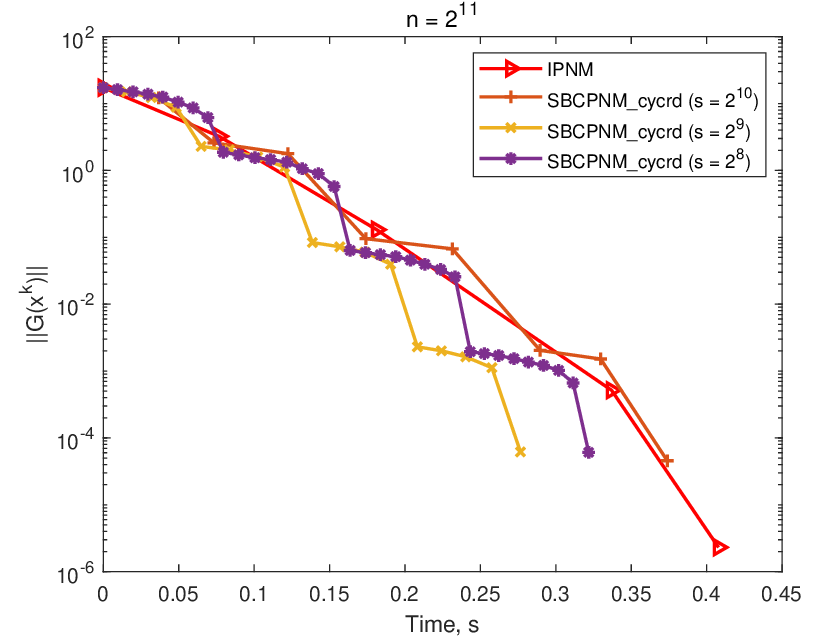}
\includegraphics[width = 0.325\textwidth]{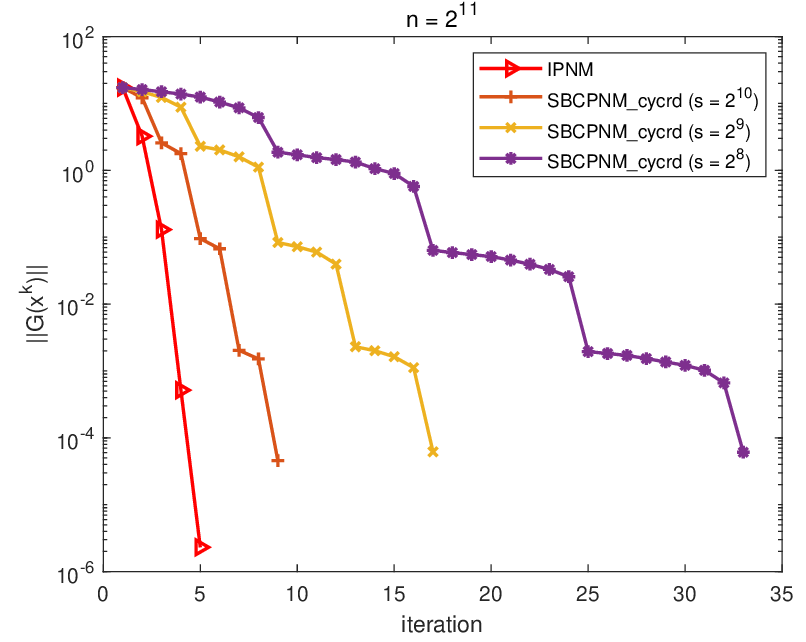}
\includegraphics[width = 0.325\textwidth]{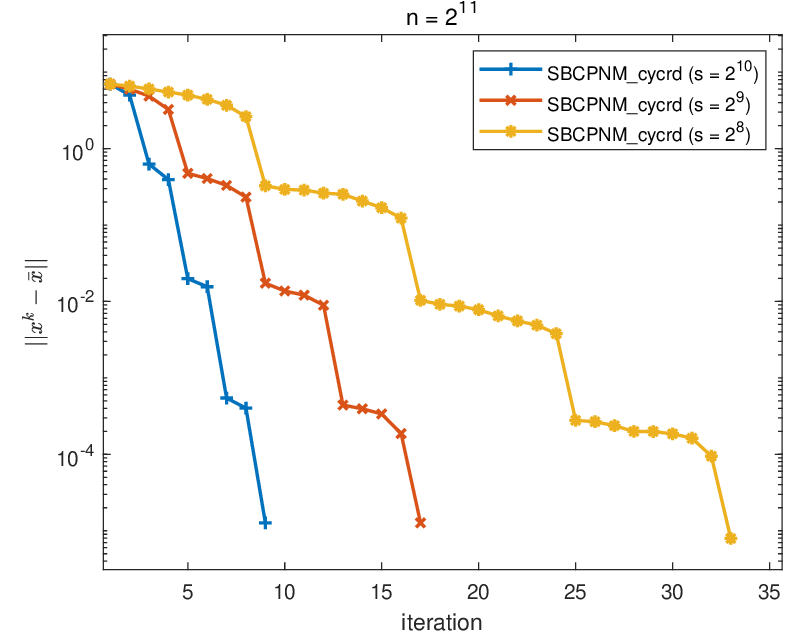}\\
\includegraphics[width = 0.325\textwidth]{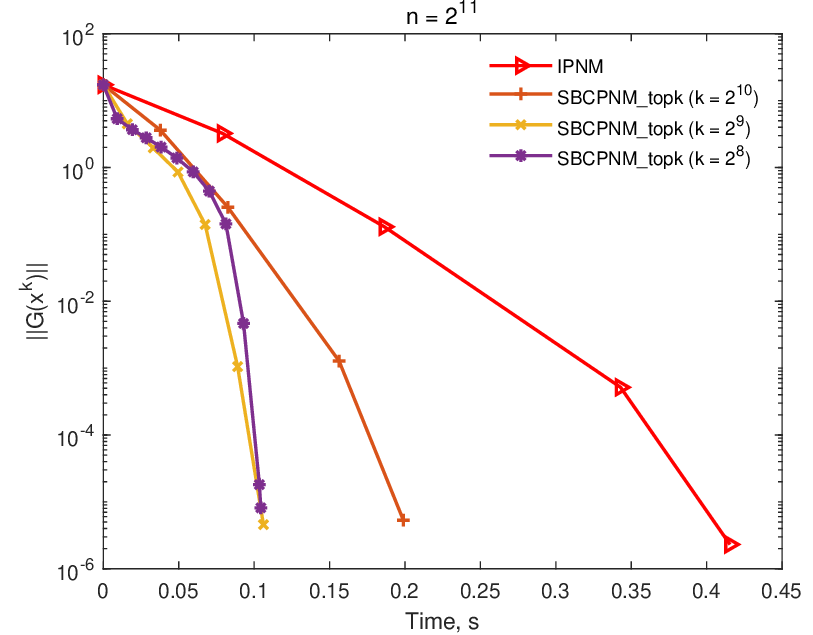}
\includegraphics[width = 0.325\textwidth]{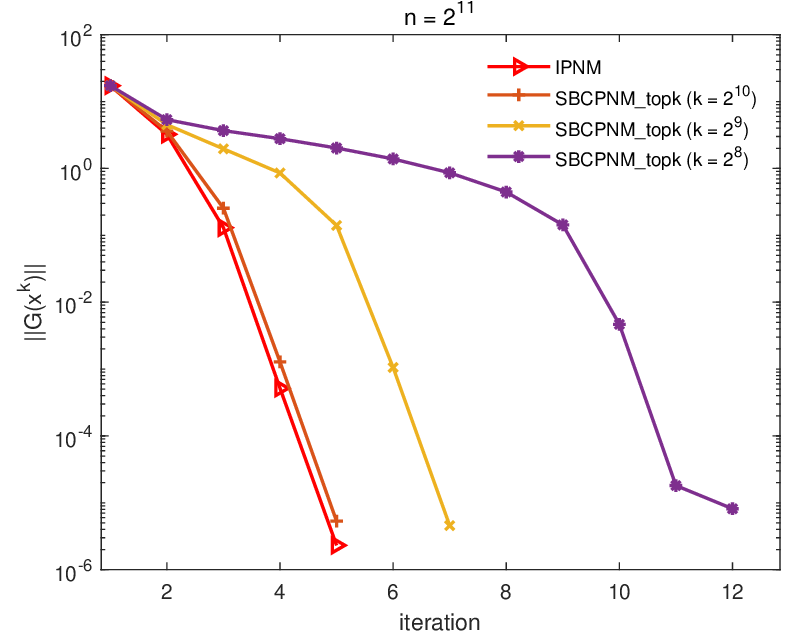}
\includegraphics[width = 0.325\textwidth]{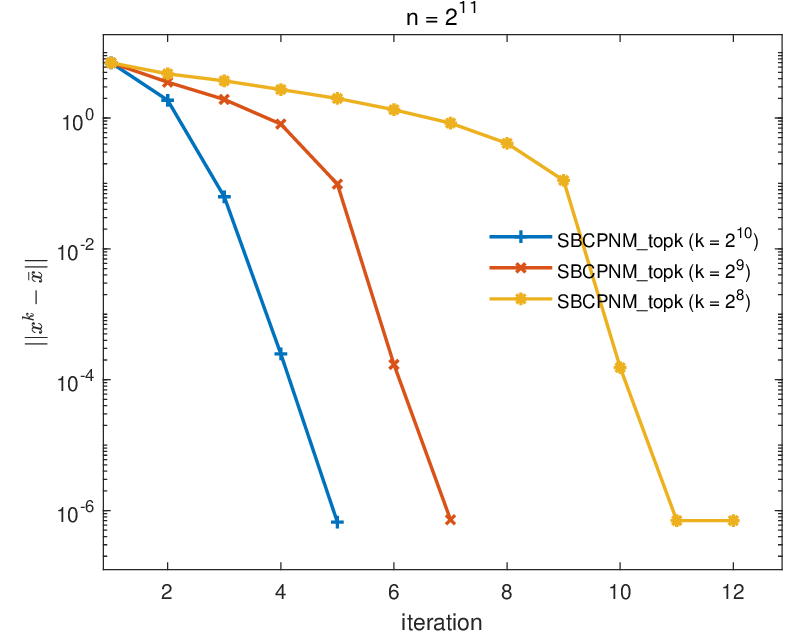}\\
\includegraphics[width = 0.325\textwidth]{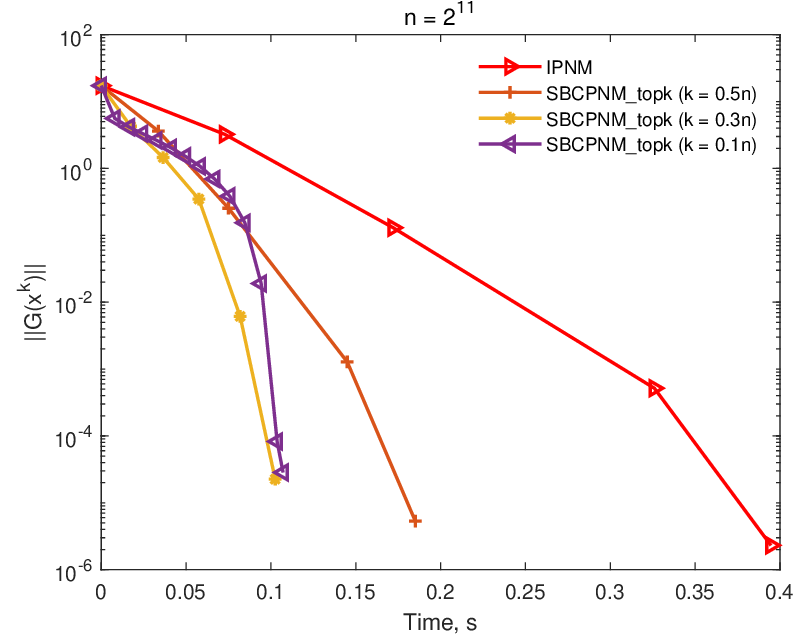}
\includegraphics[width = 0.325\textwidth]{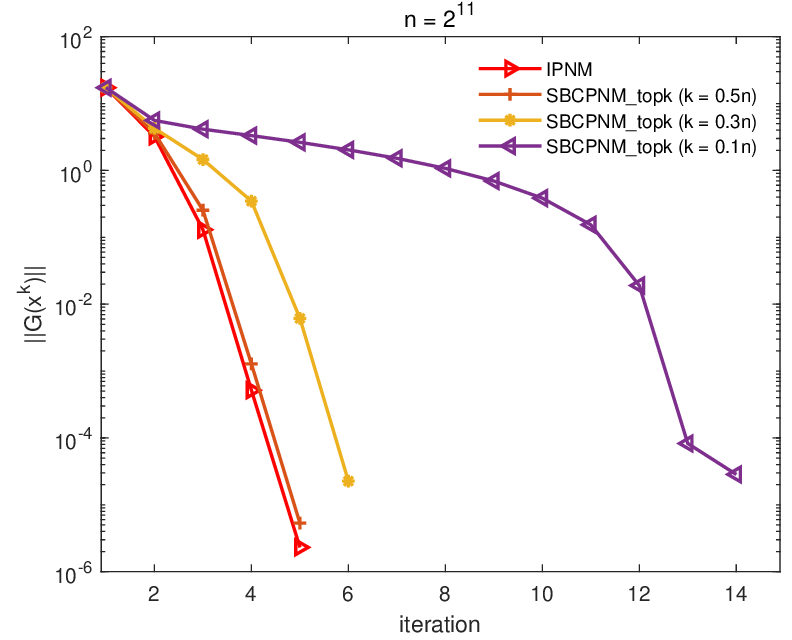}
\includegraphics[width = 0.325\textwidth]{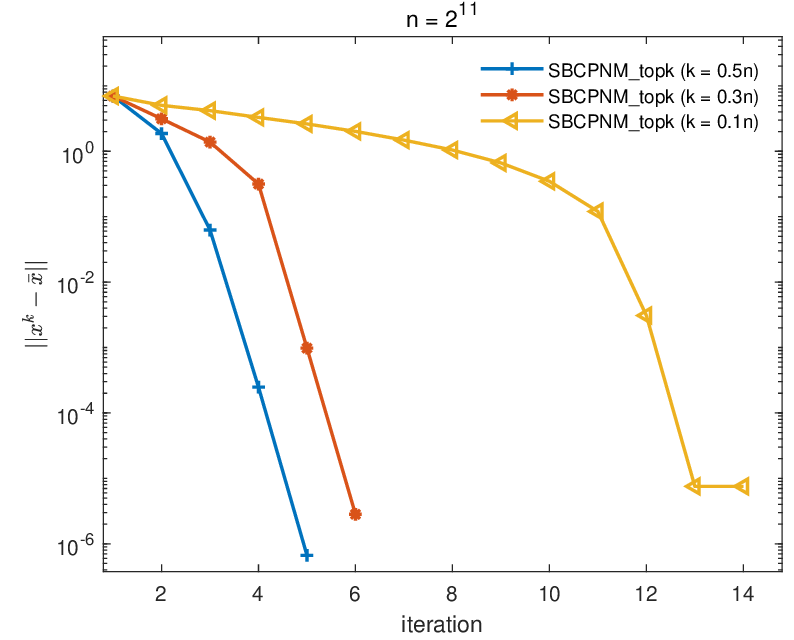}
\caption{Average performance of SBCPNM under different samplings over \(10\) trials.  Top line: {\bf SBCPNM\_cycr}; the second line: {\bf SBCPNM\_cycrd}; Bottom two lines: {\bf SBCPNM\_topk}.}\label{fig:t_students}
\end{figure}

%%%%%%%%%%%%%%%%%%%%%%%%%%%%%%%%%%%
%           nonconvex binary classification
%%%%%%%%%%%%%%%%%%%%%%%%%%%%%%%%%%%
\subsection{Nonconvex binary classification}

We study the following nonconvex binary classification problem:
\begin{equation}\label{eq:bc}
 \min_x f(x):=\frac{1}{m}\sum_{j=1}^m\ell(y_j - z_j^\top x) + \lambda \|x\|^2,
 \end{equation}
 where \(\ell(t) = \frac{2t^2}{t^2 + 4}\) is the Geman-McClure loss function, \(\lambda > 0\) is the regularized parameter and is fixed to \(0.001\) in the following tests, \(y_j\in \{0, 1\}\) is commonly referred to as class labels, and \(z_j\) satisfies \(\|z_j\| = 1\) is commonly referred to as features, \(j\in[m]\). Problem~\eqref{eq:bc} is a special case of Problem~\eqref{eq:ncp} with \(g(x) \equiv 0\). Notice that in this case, \(\arg\min_y\{q^k_{S_k}(y)\}\) is the unique solution of equation 
 \[
 \nabla f(x^k)_{S_k} + \left((Q_k)_{S_k} + \eta_k I\right)(y - y^k) = 0
 \]
 since \((Q_k)_{S_k} + \eta_k I \succeq 0\). We find the approximate solution \(\hat{y}^k\) satisfies~\(\|\nabla f(x^k)_{S_k} + \left((Q_k)_{S_k} + \eta_k I\right)(y - y^k)\|\leq \frac{\mu}{2}\|\hat{y}^k - y^k\|\) by using conjugate gradient (CG) method~\cite{NW06}. Notice that \(\nabla^2 f(x) = \frac{1}{m}\sum_{j=1}^m\ell^{''}(y_j - z_j^\top x)z_jz_j^\top + 2\lambda I = ZD(x)Z^\top + 2\lambda I\), where \(Z = [z_1, \cdots, z_m]\in\mathbb{R}^{n\times m}\), \(D(x) = {\rm Diag}(d_1, \ldots, d_m)\), and \(d_j = \frac{1}{m}\ell^{''}(y_j - z_j^\top x)\), \(j \in[m]\). We choose \(Q_k := \nabla^2f(x^k)\) and set \(\eta_k = 1.01\times\max\{-(2\lambda + \min_{1\leq j\leq m}\{d_j\}), \mu\}\) for each \(k\in\mathbb{N}\), where \(\mu = 10^{-5}\). 
 
 We consider random sampling (each iteration randomly samples \(s\) indicators, named as {\bf SBCPNM\_r}) and Top-\({\bf k}\) sampling in this test.  We stop {\bf SBCPNM\_r} and {\bf SBCPNM\_topk} when \(\|\nabla f(x^k)\| \leq 10^{-8}\) and set \(\tau = 10^{-5}\) and \(\theta = 0.6\), respectively. 
 We test on real data sets, including rcv1, and real-sim. The datasets can be downloaded from~\url{https://www.csie.ntu.edu.tw/~cjlin/libsvmtools/datasets/}.  
 We select a subsets from data rcv1 and real-sim and name them as rcv1\_sel and real\_sim\_sel, respectively. The size of rcv1\_sel and real\_sim\_sel is \([m, n] = [240, 47236]\) and \([m, n] = [180, 20958]\), respectively. Figures~\ref{fig:rev1} and~\ref{fig:real_sim} display the norm of \(\nabla f(x)\) at iterates generated by each method along with running time and iteration, respectively. It can be seen from Figure~\ref{fig:rev1} that when \(\mathbf{k} = 28000\), {\bf SBCPNM\_topk} outperforms {\bf SBCPNM\_r} and IPNM in terms of running time and the number of iterations. 
 Similar results can be observed from Figure~\ref{fig:real_sim} for \(\mathbf{k} = 16000\) in Top-\(\mathbf{k}\) sampling. It can be seen that {\bf SBCPNM\_r} and {\bf SBCPNM\_topk} can achieve better convergence rate in terms of \(\|\nabla f(x^k)\|\) than sublinear when implemented. The last column of Figure~\ref{fig:t_students} displays the distance between iterates generated by each method and \(\bar{x}\), where \(\bar{x}\) is the value returned by IPNM. Superlinear convergence rate of {\bf SBCPNM\_topk} can be observed. 
 At the bottom line of Figures~\ref{fig:rev1} and~\ref{fig:real_sim}, we also display the results obtained by {\bf SBCPNM\_topk} for lager size of selected data. It can be seen that, for the appropriate value of \(\mathbf{k}\), {\bf SBCPNM\_topk} exhibits an advantage in terms of running time. 

\begin{figure}[h!]
\centering
\includegraphics[width = 0.325\textwidth]{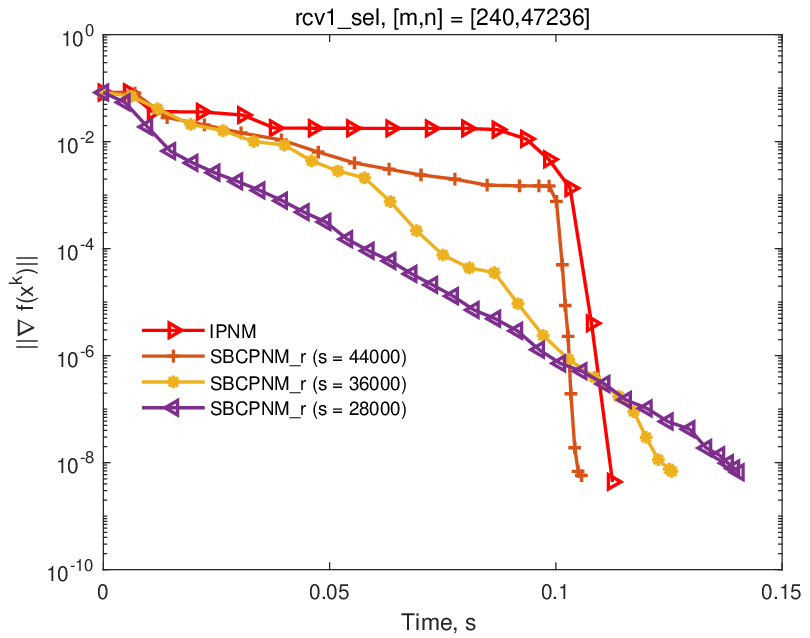}
\includegraphics[width = 0.325\textwidth]{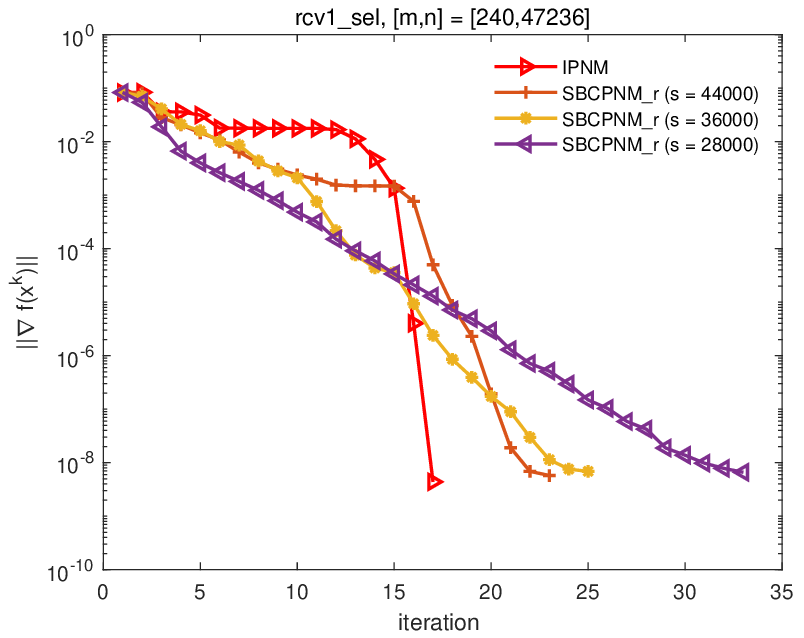}
\includegraphics[width = 0.325\textwidth]{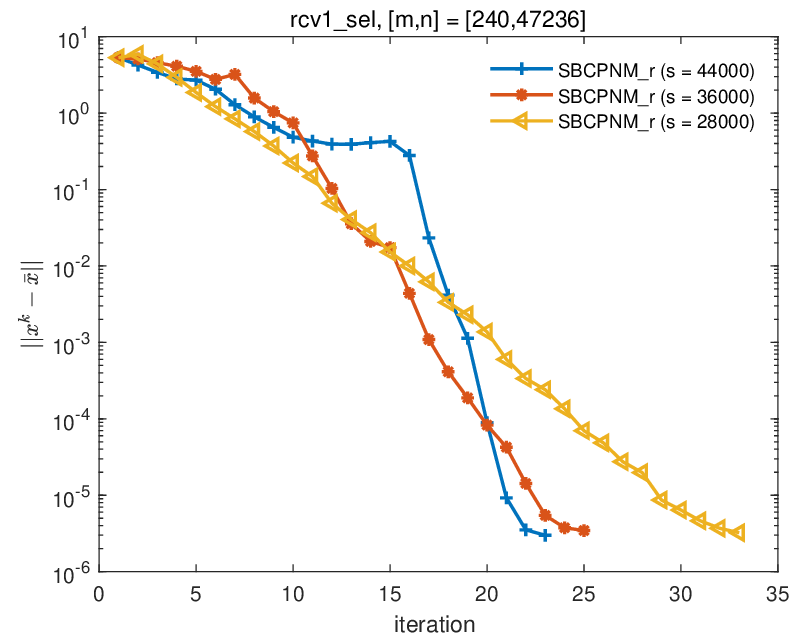}\\
\includegraphics[width = 0.325\textwidth]{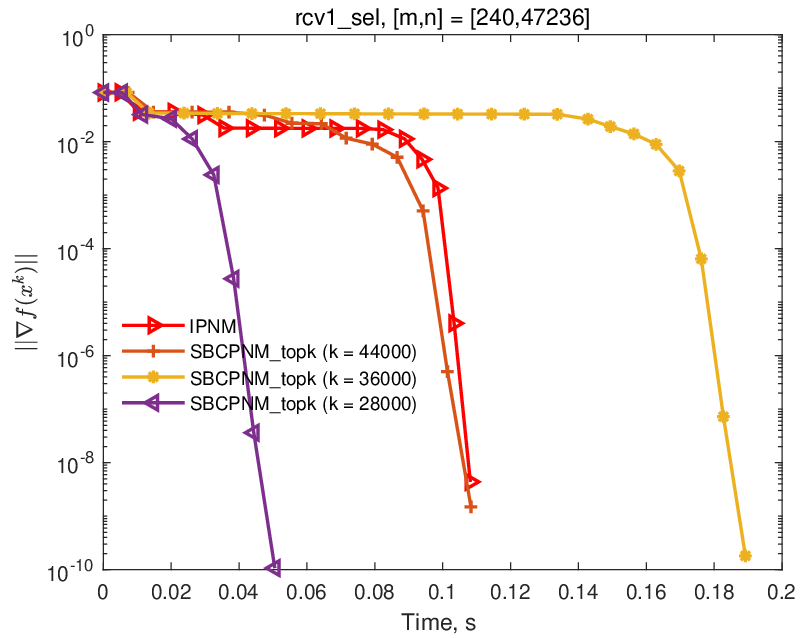}
\includegraphics[width = 0.325\textwidth]{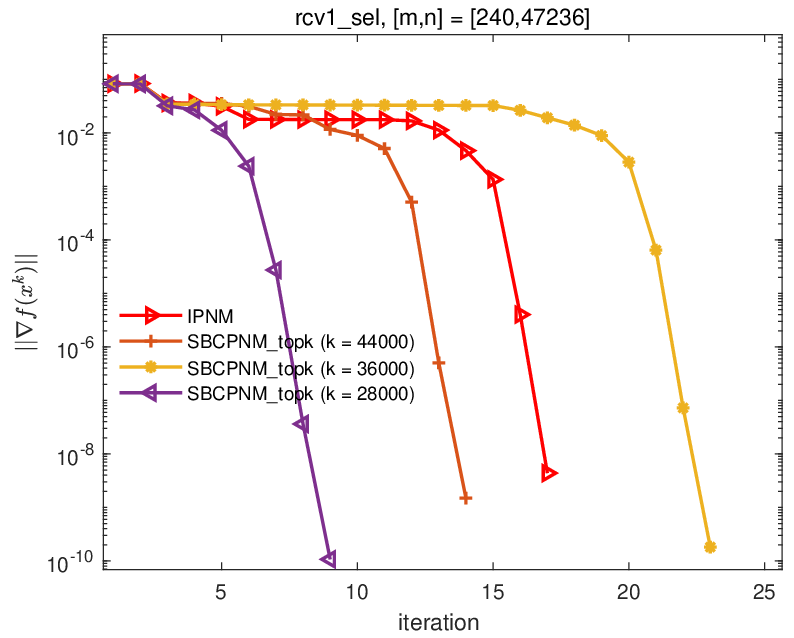}
\includegraphics[width = 0.325\textwidth]{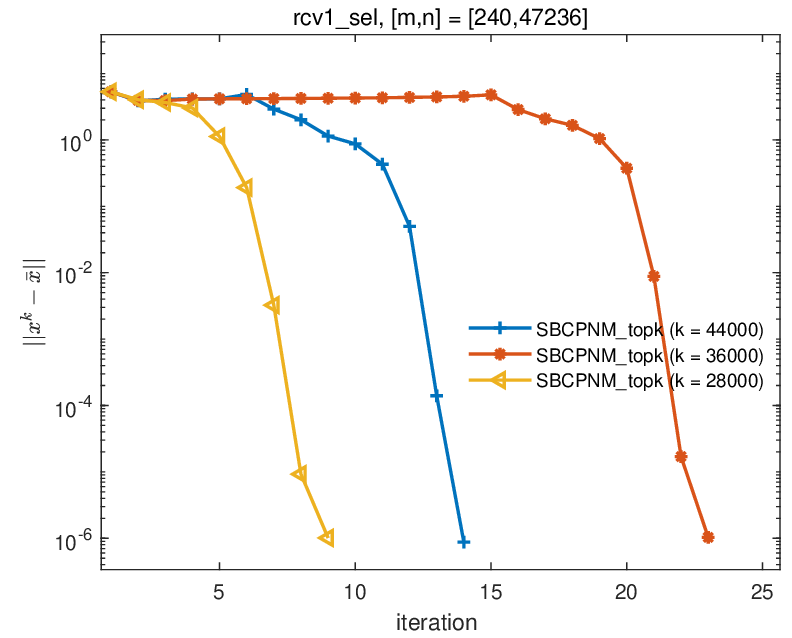}\\
\includegraphics[width = 0.325\textwidth]{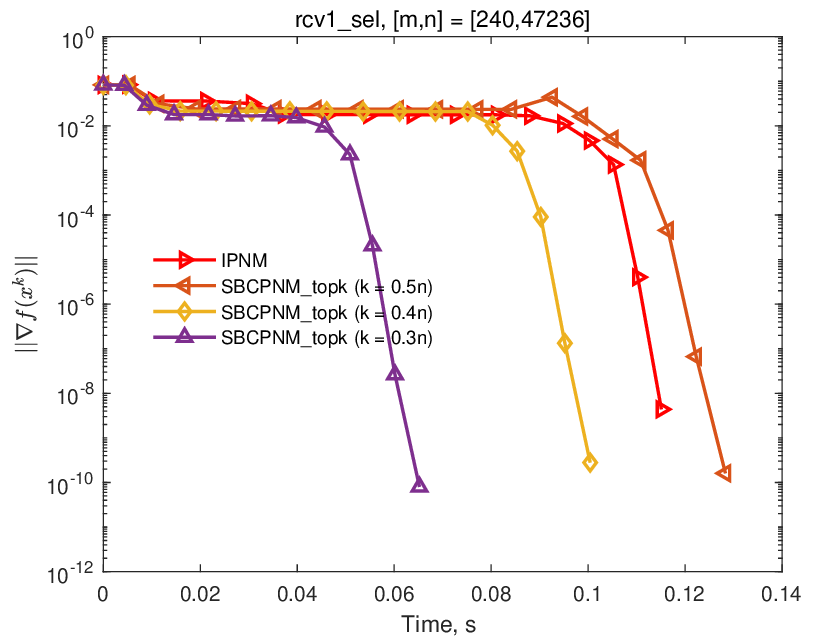}
\includegraphics[width = 0.325\textwidth]{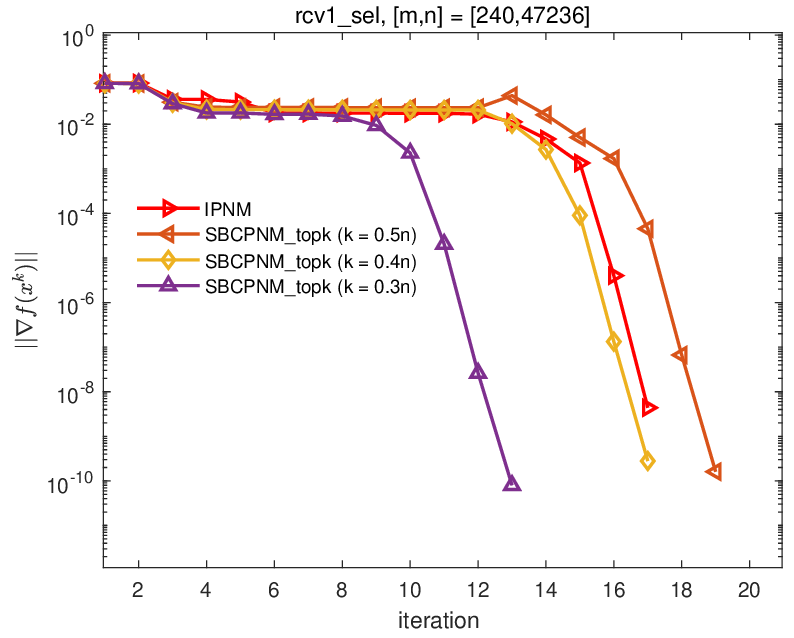}
\includegraphics[width = 0.325\textwidth]{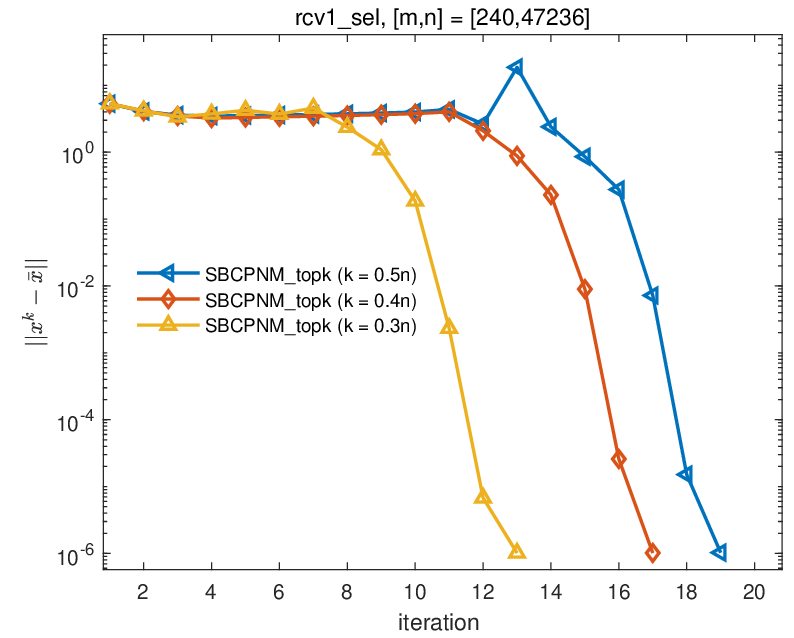}\\
\includegraphics[width = 0.325\textwidth]{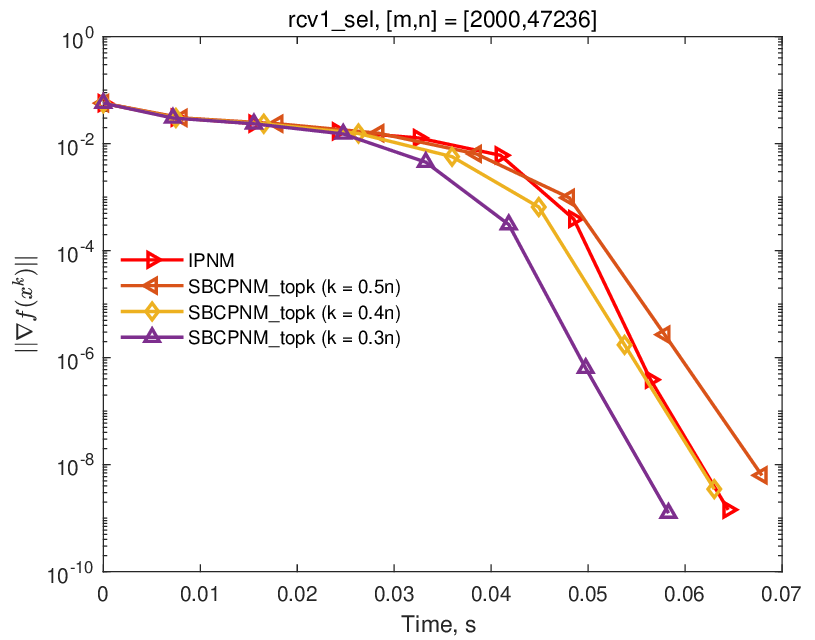}
\includegraphics[width = 0.325\textwidth]{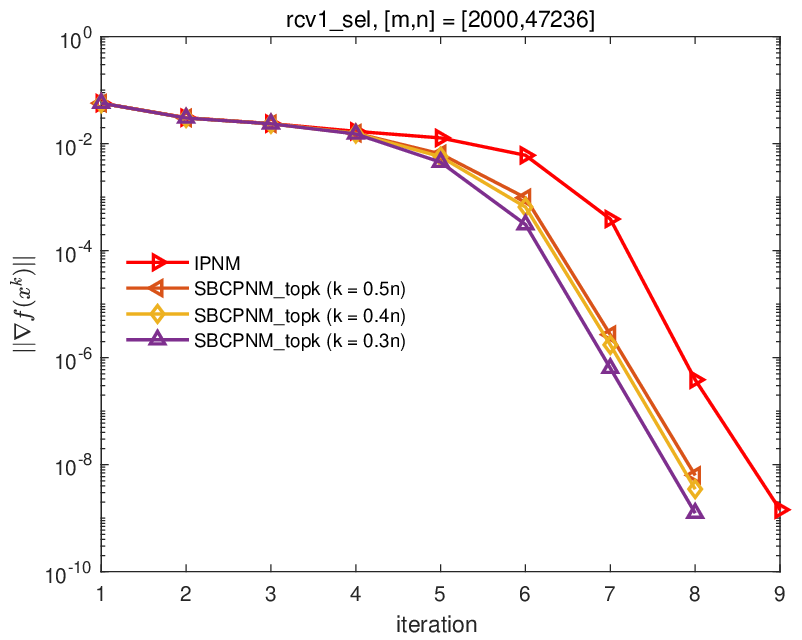}
\includegraphics[width = 0.325\textwidth]{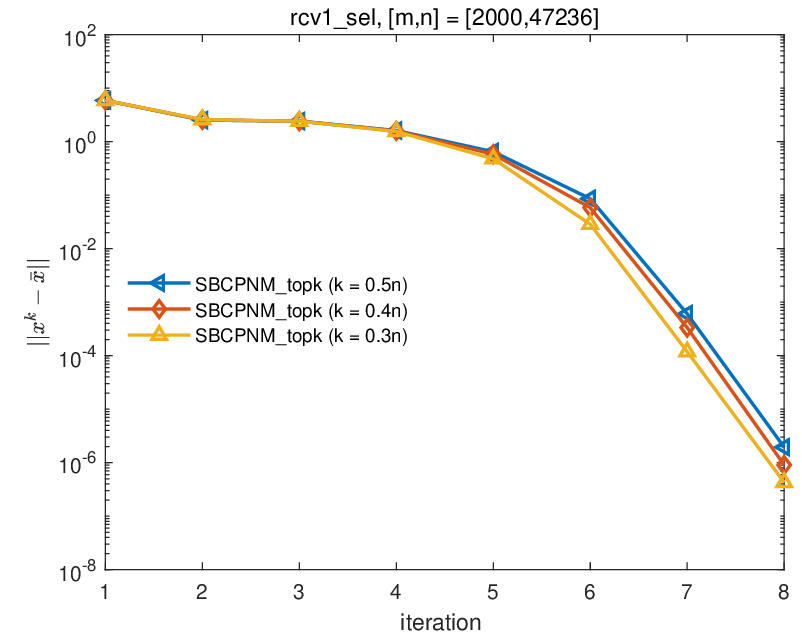}
\caption{Average performance of SBCPNM under different samplings over \(10\) trials on datasets rcv1\_sel. The top line: {\bf SBCPNM\_r}; the second line: {\bf SBCPNM\_{topk}}; the third line: {\bf SBCPNM\_{topk}} with \(\mathbf{k} = \{50\%n, 40\%n, 30\%n\}\) for selected data with \(m = 240\); the bottom line: {\bf SBCPNM\_{topk}} for selected data with \(m = 2000\).}\label{fig:rev1}
\end{figure}

\begin{figure}[h!]
\centering
\includegraphics[width = 0.325\textwidth]{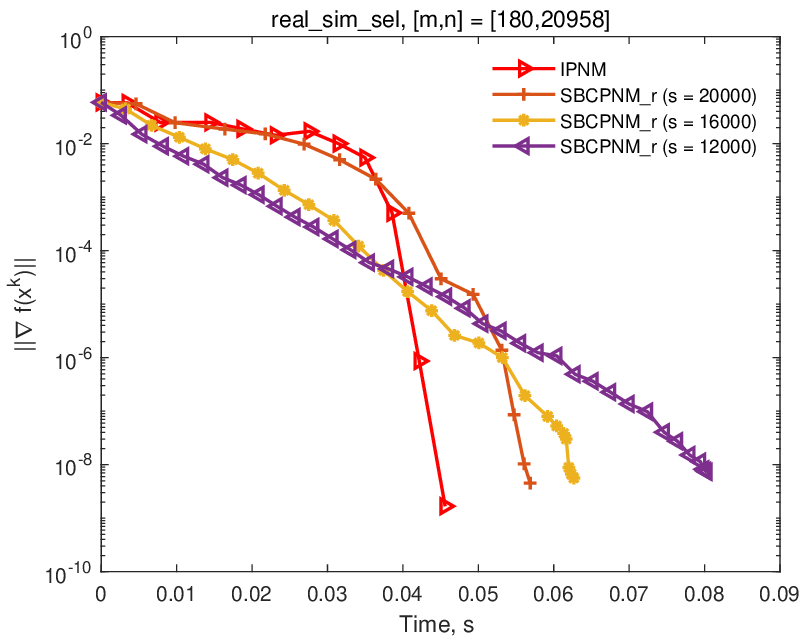}
\includegraphics[width = 0.325\textwidth]{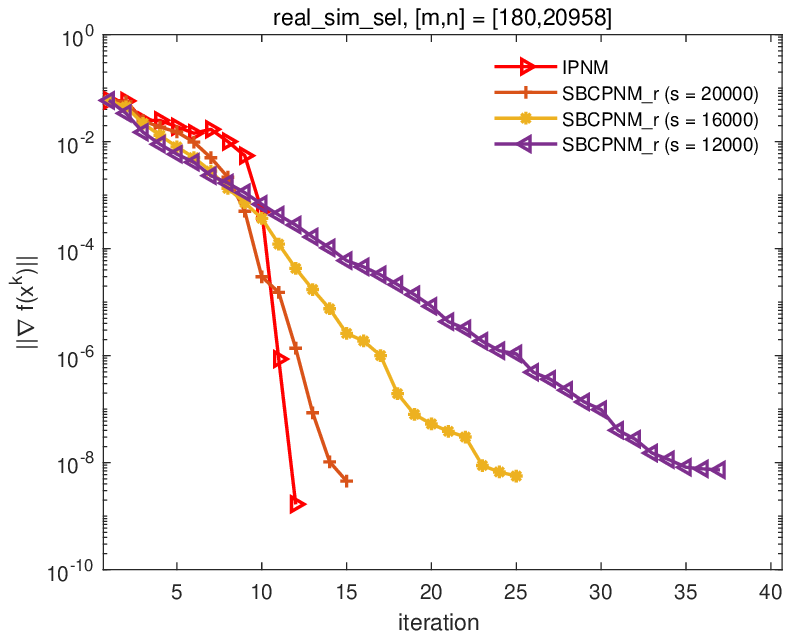}
\includegraphics[width = 0.325\textwidth]{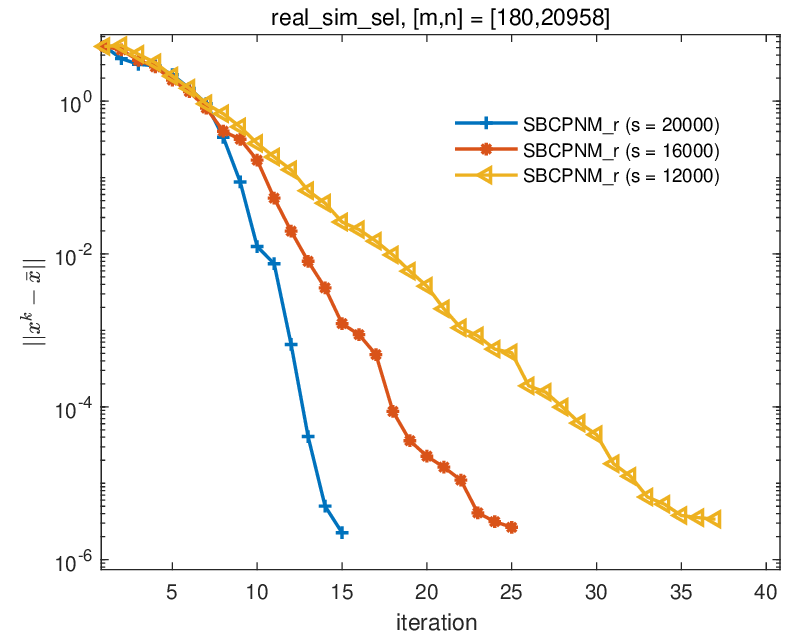}\\
\includegraphics[width = 0.325\textwidth]{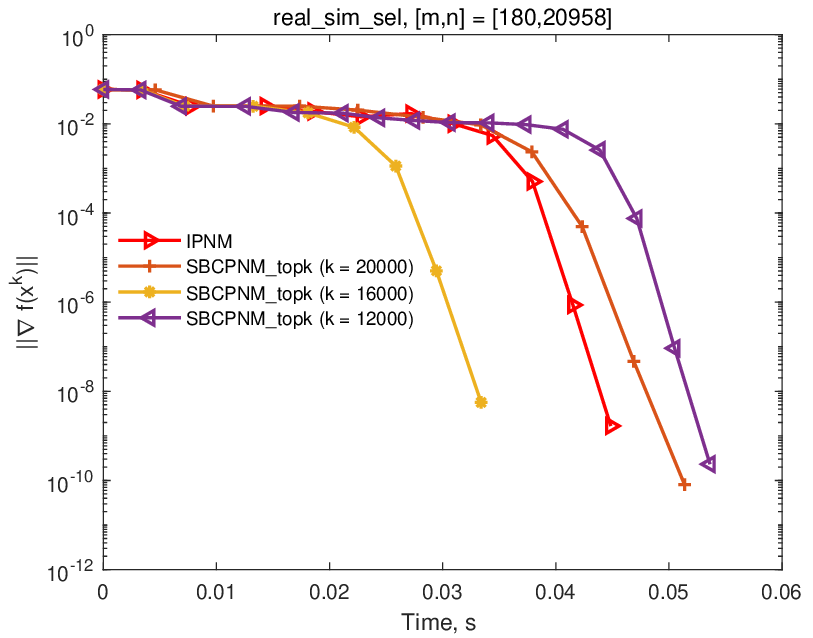}
\includegraphics[width = 0.325\textwidth]{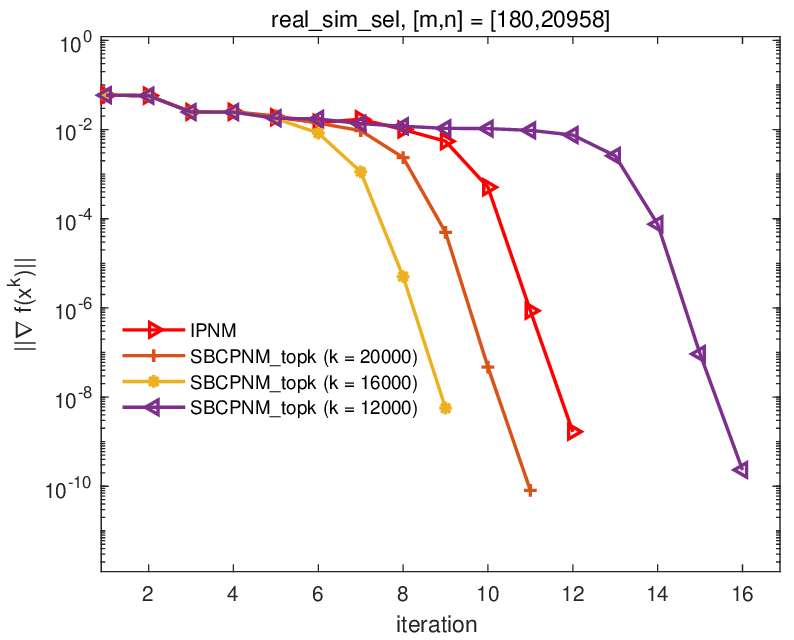}
\includegraphics[width = 0.325\textwidth]{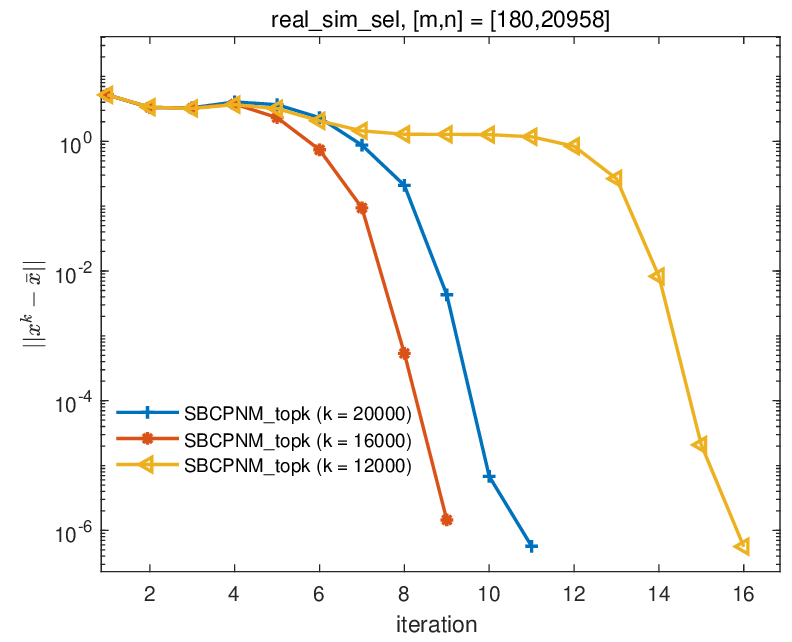}\\
\includegraphics[width = 0.325\textwidth]{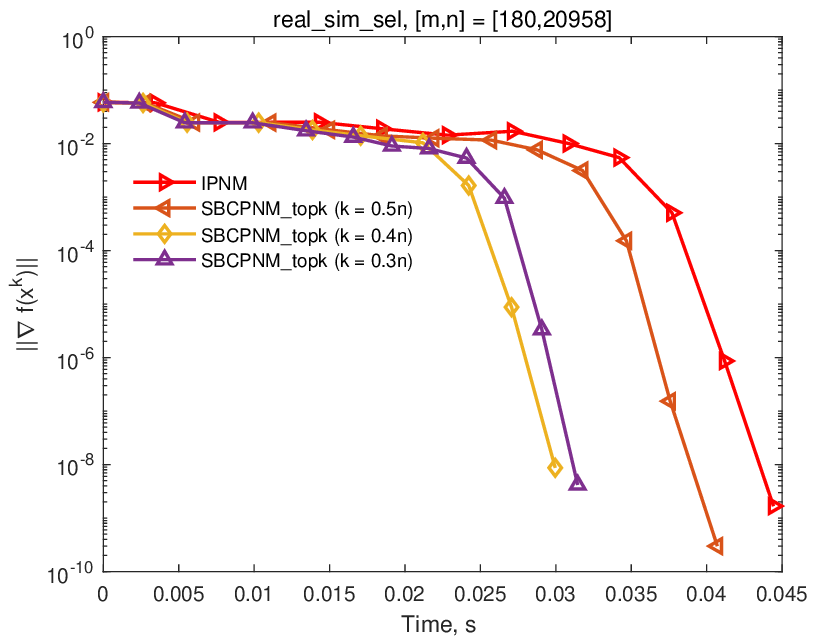}
\includegraphics[width = 0.325\textwidth]{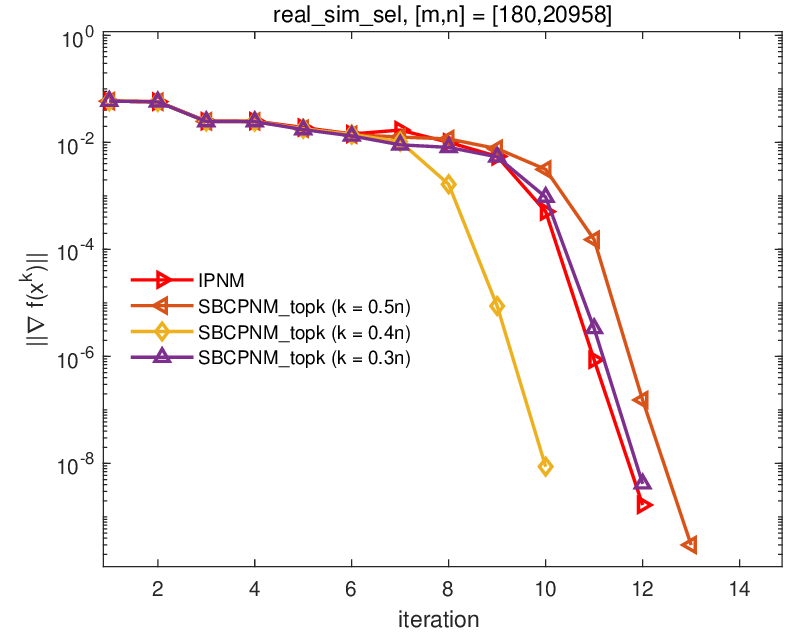}
\includegraphics[width = 0.325\textwidth]{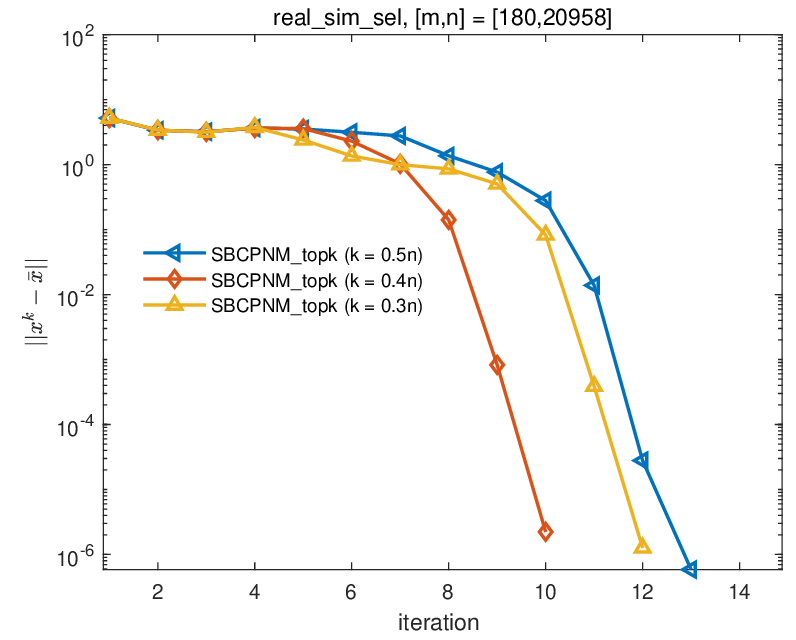}\\
\includegraphics[width = 0.325\textwidth]{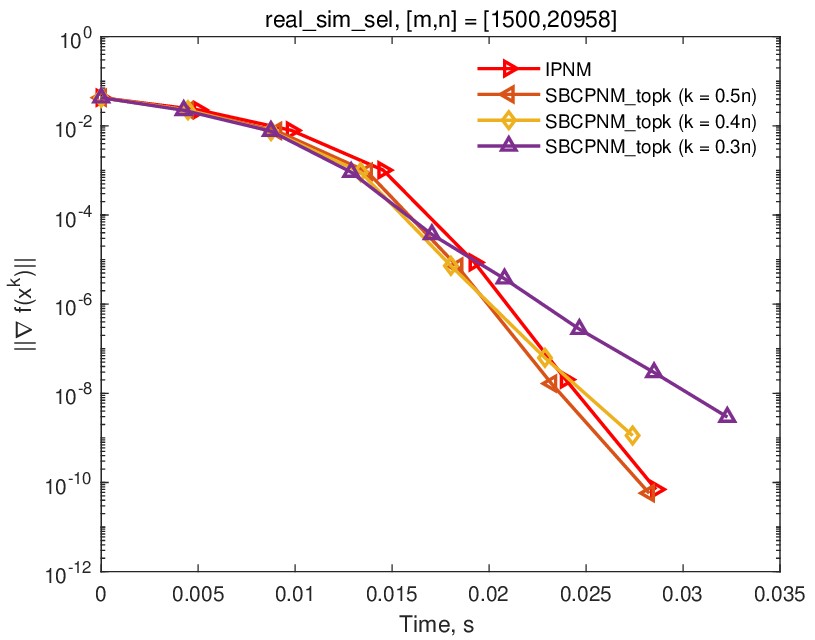}
\includegraphics[width = 0.325\textwidth]{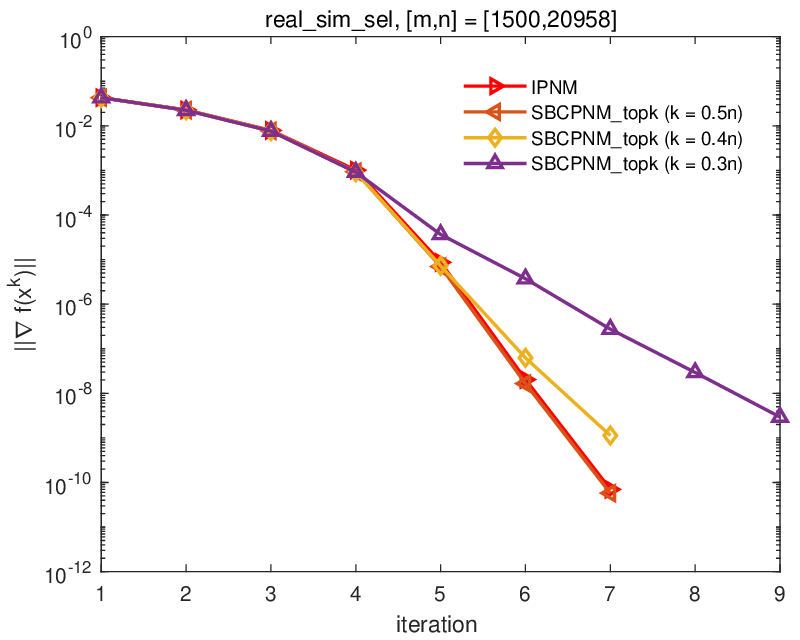}
\includegraphics[width = 0.325\textwidth]{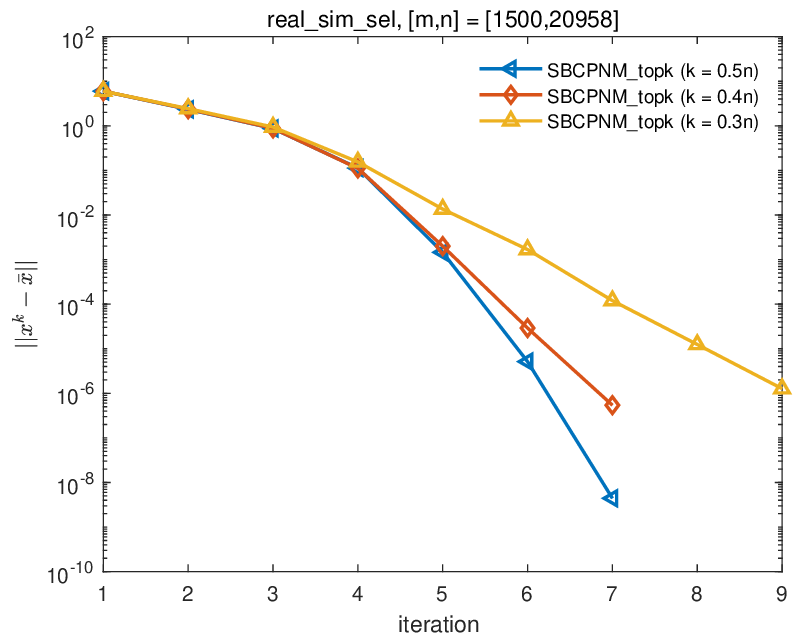}
\caption{Average performance of SBCPNM under different samplings over \(10\) trials on datasets real\_sim\_sel. The top line: {\bf SBCPNM\_r}; the second line: {\bf SBCPNM\_{topk}}; the third line: {\bf SBCPNM\_{topk}} with \(\mathbf{k} = \{50\%n, 40\%n, 30\%n\}\) for selected data with \(m = 180\); the bottom line: {\bf SBCPNM\_{topk}} for selected data with \(m = 1500\).}\label{fig:real_sim}
\end{figure}

%%%%%%%%%%%%%%%%%%%%%%%%%%%%%%%%%%%%%%%%%%%%%%%
% ******* Biweight loss with group regularization
%%%%%%%%%%%%%%%%%%%%%%%%%%%%%%%%%%%%%%%%%%%%%%%
\subsection{Biweight loss with group regularization}

We study the following nonconvex problem:
\begin{equation}\label{eq:bg}
 \min_x \frac{1}{m}\sum_{j=1}^m\phi(a_j^\top x - b_j) + \lambda \sum_{i=1}^{\lceil n/5\rceil}\sqrt{\sum_{t=1}^{\min\{5, n-5(p-1)\}}x_{5(p-1)+j}^2},
 \end{equation}
 where \(\phi(t) = \frac{t^2}{t^2 + 1}\), \(\lambda > 0\) is the regularized parameter and is fixed to \(0.001\) in the following tests, \(b_j\in \{-1, 1\}\) is commonly referred to as class labels, and \(a_j\) satisfies \(\|a_j\| = 1\) is commonly referred to as features, \(j\in[m]\). We can denote \(f(x) := \frac{1}{m}\sum_{j=1}^m\phi(a_j^\top x - b_j)\) and \(g(x): = \lambda \sum_{i=1}^{\lceil n/5\rceil}\sqrt{\sum_{t=1}^{\min\{5, n-5(p-1)\}}x_{5(p-1)+j}^2}\) for Problem~\eqref{eq:bg}. Each set of five consecutive coordinates is grouped into a single block. 

Notice that \(\nabla^2 f(x) = \frac{1}{m}\sum_{j=1}^m\phi^{''}(a_j^\top x - b_j)a_ja_j^\top = AD(x)A^\top\), where \(A = [a_1, \cdots, a_m]\in\mathbb{R}^{n\times m}\), \(D(x) = {\rm Diag}(d_1(x), \ldots, d_m(x))\), and \(d_j(x) = \frac{1}{m}\phi^{''}(a_j^\top x - b_j)\), \(j \in[m]\). We choose \(Q_k := A\widetilde{D}_kA^\top\) where \(\widetilde{D}_k = {\rm Diag}(\tilde{d}_1, \ldots, \tilde{d}_m)\) with \(\tilde{d}_j = \max\{\frac{1}{m}\phi^{''}(a_j^\top x^k - b_j), 10^8\}\), \(\eta_k = 0.01\mu\) if \(\min_j\{\tilde{d}_j\} + 0.01\mu \geq \mu\), and \(\eta_k = 1.01\mu\), otherwise, and \(\mu = 10^{-3}\). Similar to Problem~\eqref{eq:st} in subsection~\ref{subsec:st}, the approximate solution \(\hat{y}^k\) can be obtained by using the SSN method. 

We consider the cyclic sampling with \(\vert S_k\vert \equiv 5\). We compare Algorithm~\ref{alg:pnewton} with the inexact variable metric stochastic block-coordinate descent method (named as {\bf VM}) proposed in~\cite{LW20}. As in~\cite{LW20}, we solve each subproblem of {\bf VM} by using \(10\) SpaRSA~\cite{WNF09} iterations. Notice that we do not need to update the blocks satisfy \(x^k_{S_k} = 0\) and \(-\frac{1}{\lambda}\nabla f(x^k)_{S_k}\in \partial \|x\|\big\vert_{x^{k}_{S_k}}\). Figure~\ref{fig:group} displays the performance of {\bf SBCPNM} and {\bf VM} in terms of \(\|\mathcal{G}(x^k)\|\) and \(\|x^k - \bar{x}\|\), where \(\bar{x}\) is calculated by using IPNM. It can be seen that both {\bf BCPNM} and {\bf VM} can achieve better convergence rate in terms of \(\|\mathcal{G}(x^k)\|\) than sublinear when implemented. Algorithms {\bf BCPNM} and {\bf VM} follow the same change trend, but the running time and the number of iterations are different due to the different methods are used to solve subproblems (SSN vs SpaRSA). Both {\bf BCPNM} and {\bf VM} exhibit superlinear convergence.

\begin{figure}[h!]
\centering
\includegraphics[width = 0.325\textwidth]{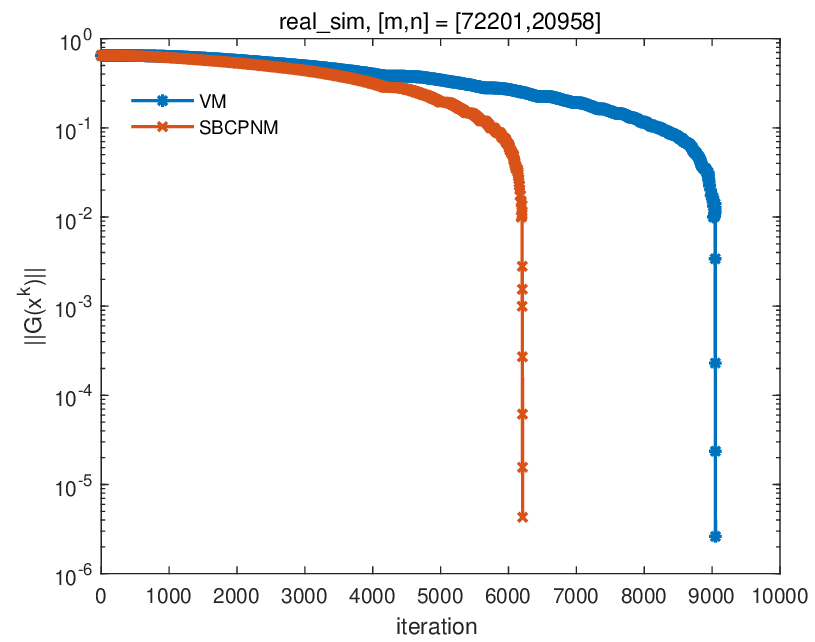}
\includegraphics[width = 0.325\textwidth]{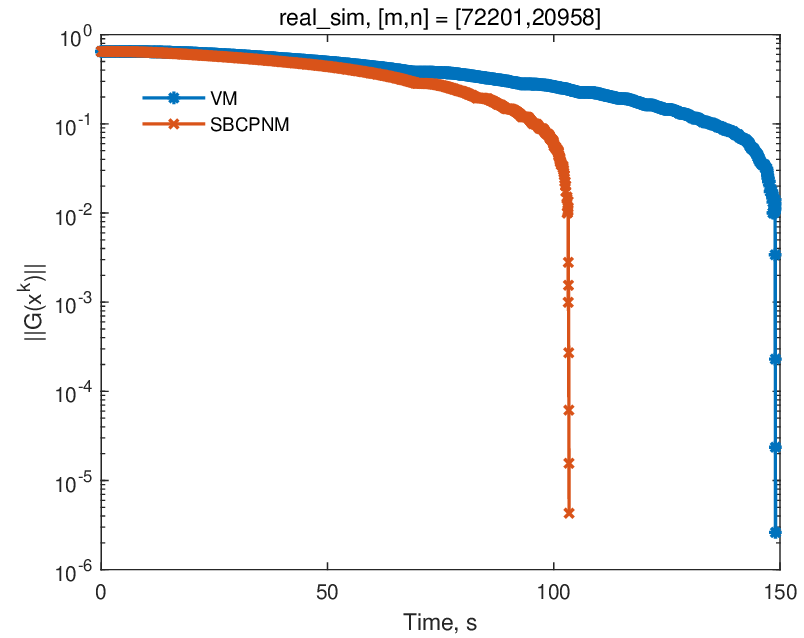}
\includegraphics[width = 0.325\textwidth]{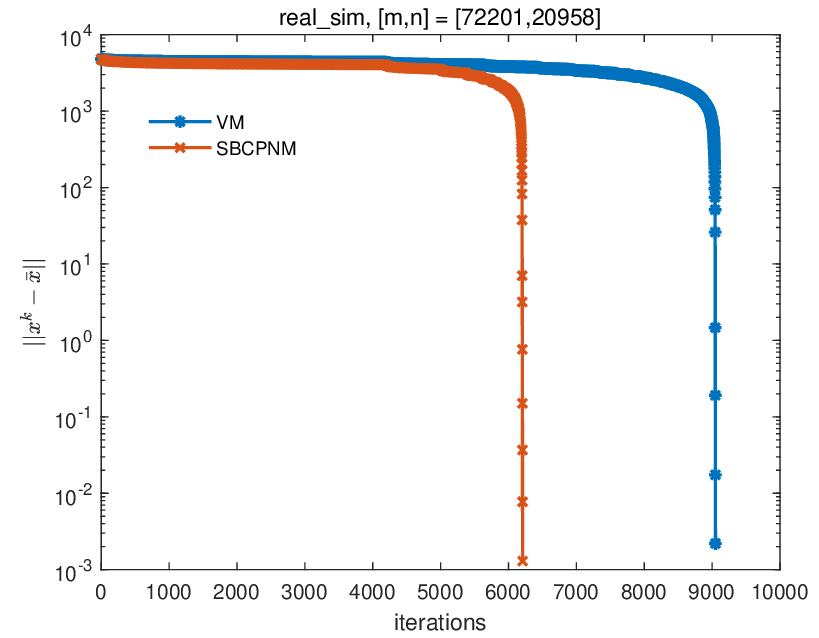}
\caption{Performance of {\bf SBCPNM} and {\bf VM} on dataset real\_sim.}\label{fig:group}
\end{figure}

%%%%%%%%%%%%%%%%%%%%%%%%%%%%%%%%%%%%%%%%%%%%%%
% ****** Conclusions
%%%%%%%%%%%%%%%%%%%%%%%%%%%%%%%%%%%%%%%%%%%%%%
\section{Conclusions}\label{sec:conclusions}

In this paper, we propose a stochastic block-coordinate proximal Newton method for minimizing the sum of a smooth (possibly nonconvex) function and a separable convex (possibly nonsmooth) function.
We establish the global convergence rate of the method under different assumptions on the sampling.
Under a suitable sampling assumption, we show that the unit step-size variant of the algorithm achieves a superlinear convergence rate for nonconvex composite optimization problems, and superlinear/quadratic convergence rates for convex composite optimization problems under difference error bound conditions. These results are consistent with those of the inexact proximal Newton method in the literature. 
Our experiments demonstrated that stochastic strategies are effective when \(n\) is large, and the algorithm demonstrates a convergence rate that is superior to sublinear in terms of the norm of residual mapping and the superlinear convergence rate in terms of iterates.

\bmhead{Acknowledgements}
This work was supported by NSFC (No. 12271217).

\section*{Declarations}

The data that support the findings of this study are available from the author upon reasonable request. The authors declare that they have no conflict of interest. 

%\section*{Competing Interests}
%Authors declare they have no financial interests. 

%\section*{Author Contributions}
%All authors contributed to the study conception and design. The first draft of the manuscript was written by Hong Zhu and all authors commented on previous versions of the manuscript. All authors read and approved the final manuscript.

\begin{appendices}

\section{Proof of Theorem~\ref{th:limitspwithoutls}}\label{app:proofsbcpnusz}

\begin{proof}
(a) The proof of statement (a) is similar to the proof of Lemma~\ref{lem:ngk}.

(b) Notice that Lemma~\ref{lemma:dqkng} still holds and we have
\begin{align*}%\label{eq:dphi}
\varphi(x^k) \geq& q_k(x^{k+1})  \nonumber\\
=& \varphi(x^{k\!+\!1}) \!-\! (f(x^{k\!+\!1}) \!-\! f(x^k) \!-\! \langle\nabla f(x^k), x^{k\!+\!1}  \!-\! x^k\rangle) \!+\!  \frac{1}{2}\langle Q_k(x^{k\!+\!1} \!-\! x^k), x^{k\!+\!1} \!-\! x^k\rangle  \nonumber \\
&+ \frac{\eta_k }{2}\|x^{k+1}  - x^k\|^2 \nonumber \\
\geq& \varphi(x^{k+1}) \!-\! \frac{L_{S_k}}{2}\|x^{k+1}  \!-\! x^k\|^2 \!+\!  \frac{1}{2}\langle Q_k(x^{k+1} \!-\! x^k), x^{k+1} \!-\! x^k\rangle \!+\! \frac{\eta_k }{2}\|x^{k+1} \!-\! x^k\|^2 \nonumber\\
\geq & \varphi(x^{k+1}) + \frac{\mu}{2}\|x^{k+1}  - x^k\|^2,
\end{align*}
where the last inequality follows from \(Q_k + (\eta_k - L_{S_k} - \mu)I_n\succeq 0\). 

(c)-(d) The proofs  are similar to those of Theorem~\ref{th:cfv} (i) and (ii), respectively.

(e)-(f) The proofs  are similar to those of Theorems ~\ref{th:clusterpoint} (i)-\ref{th:limitsppn} (i) and Theorems~\ref{th:clusterpoint} (ii)-\ref{th:limitsppn} (ii), respectively.
%(e) The proof  is similar to the proofs of Theorem~\ref{th:clusterpoint} (i) and Theorem~\ref{th:limitsppn} (i).

%(f) The proof is similar to the proofs of Theorem~\ref{th:clusterpoint} (ii) and Theorem~\ref{th:limitsppn} (ii).
\end{proof}

\end{appendices}

%%=============================%%

\bibliography{Manuscript_localconvex_new}% common bib file

\end{document}